\documentclass[11pt]{imsart}
\usepackage{amssymb,amsfonts,amsmath,amsthm,amscd,stackrel,dsfont,mathrsfs,eqsection}
\usepackage{graphicx}
\DeclareMathAlphabet{\mathpzc}{OT1}{pzc}{m}{it}

\newtheorem{propo}{Proposition}[section]
\newtheorem{lemma}[propo]{Lemma}
\newtheorem{definition}[propo]{Definition}
\newtheorem{coro}[propo]{Corollary}
\newtheorem{thm}[propo]{Theorem}

\newtheorem{remark}[propo]{Remark}

\def\endproof{\hfill$\Box$\vspace{0.4cm}}

\def\di{{\partial i}}
\def\Rep{{\cal R}}

\def\uB{\underline{B}}

\newcommand{\<}{\langle}
\renewcommand{\>}{\rangle}

\newcommand{\reals}{{\mathds R}}
\newcommand{\integers}{{\mathds Z}}
\newcommand{\naturals}{{\mathds N}}

\newcommand{\Z}{\field{Z}} 
\newcommand{\1}{{\bf 1}} 
\newcommand{\prob}{{\mathbb P}} 
\newcommand{\E}{{\mathbb E}} 

\def\Typ{{\sf Typ}}
\def\<{\langle}
\def\>{\rangle}
\def\cX{{\cal X}}
\def\vphi{\varphi}
\def\ve{\varepsilon}
\def\ve{\varepsilon}
\def\Var{{\rm Var}}
\def\reals{{\mathbb R}}

\def\bX{\overline{X}}

\def\reals{{\mathbb R}}
\def\cC{{\mathfrak C}}

\def\ind{{\mathbb I}}
\def\da{{\partial a}}
\def\H{{\mathbb H}}

\def\Tree{{\sf T}}
\def\CTree{{\sf CT}}
\def\me{\nu}
\def\normeq{\cong}
\def\di{\partial i}
\def\Tr{{\sf T}}
\def\atanh{{\rm atanh}\, }
\def\cN{{\cal N}}
\def\oalpha{\overline{\alpha}}
\def\F{{\sf F}}
\def\Ball{{\sf B}}
\def\cBall{\overline{\sf B}}
\def\sTV{\mbox{\tiny\rm TV}}
\def\de{{\rm d}}
\def\onu{\overline{\nu}}
\def\graph{{\mathbb G}}
\def\Pair{{\mathfrak P}}
\def\uu{\underline}
\def\ux{\underline{x}}

\def\wuX{\widetilde{\underline{X}}}
\def\wpsi{\widetilde{\psi}}
\def\wnu{\widetilde{\nu}}
\def\uy{\underline{y}}

\def\eps{\epsilon}

\def\cC{{\cal C}}
\def\l|{\left|\left|}
\def\r|{\right|\right|}

\def\E{\mathds E}
\def\1{\mathds 1}
\def\prob{{\mathds P}}
\def\ed{\stackrel{\rm d}{=}}
\def\ind{{\mathds I}}

\def\ve{\varepsilon}

\def\onu{\overline{\nu}}
\def\de{{\rm d}}
\def\reals{{\mathds R}}

\def\ux{\underline{x}}
\def\uy{\underline{y}}

\def\M{{\mathpzc M}}
\def\Tree{{\sf T}}

\def\Ball{{\sf B}}
\def\Edge{{\sf D}}

\def\Graph{{\sf G}}
\def\cBall{\overline{\sf B}}

\def\sTV{\mbox{\tiny\rm TV}}

\def\uh{\underline{h}}

\def\atanh{{\rm atanh}}

\def\uX{\underline{X}}
\def\cX{{\cal X}}

\def\node{P}
\def\edge{\rho}
\def\andeg{\overline{P}}
\def\aedeg{\overline{\rho}}
\def\ed{\stackrel{{\tiny \rm d}}{=}}

\def\dBall{\partial{\sf B}}

\def\root{{\o}}
\def\atanh{{\rm atanh}}

\def\eps{\epsilon}
\def\cC{{\cal C}}
\def\l|{\left|\left|}
\def\r|{\right|\right|}

\def\E{\mathds E}
\def\1{\mathds 1}
\def\prob{{\mathds P}}
\def\ed{\stackrel{\rm d}{=}}
\def\ind{{\mathds I}}

\def\ve{\varepsilon}

\def\onu{\overline{\nu}}
\def\de{{\rm d}}
\def\reals{{\mathds R}}

\def\M{{\mathpzc M}}
\def\Tree{{\sf T}}
\def\Ball{{\sf B}}
\def\Edge{{\sf D}}
\def\Graph{{\sf G}}
\def\cBall{\overline{\sf B}}

\def\sTV{\mbox{\tiny\rm TV}} 
\def\diam{{\rm diam}}
\def\dU{{\partial U}}
\def\di{{\partial i}}
\def\dj{{\partial j}}
\def\cU{{\cal U}}

\def\cZ{{\cal Z}}
\def\rh{\widehat{\rho}}
\def\M{{\sf M}}
\def\oU{\overline{U}}

\def\atanh{{\rm atanh}}

\def\cX{{\cal X}}

\def\sQ{{\sf Q}}

\def\vE{\vec{E}}

\def\1z{\vec{z}^{\,(1)}}
\def\2z{\vec{z}^{\,(2)}}
\def\3z{\vec{z}^{\,(3)}}

\def\upsi{\underline{\psi}}
\def\betaw{w}
\def\alp{a}
\def\bt{b}
\def\eham{{\mathcal E}}
\def\const{\kappa}
\def\GF{{\rm GF}(2)}
\def\ub{\underline{b}}

\def\uy{\underline{y}}

\def\CC{\mathcal C}
\def\ux{\underline{x}}

\def\v0{\vec{0}}

\def\vS{\vec{S}}
\def\bx{{\bf x}}
\def\vX{\vec{X}}
\def\poisson{{\sf Poisson}}

\def\gauss{{\sf G}}

\def\FF{\mathcal F}

\def\cF{{\cal F}}
\def\A{{\mathbb A}}

\def\K{{\mathcal K}}

\def\l|{\left|\left|}
\def\r|{\right|\right|}

\def\Var{{\rm Var}}

\def\E{\mathbb E}
\def\prob{{\mathbb P}}
\def\hprob{\widehat{\mathbb P}}

\def\ed{\stackrel{\rm d}{=}}
\def\G{{\mathcal G}}
\def\M{{\mathcal M}}
\def\K{{\mathcal K}}

\def\coeff{{\sf coeff}}

\def\vw{\vec{w}}
\def\vz{\vec{z}}
\def\vx{\vec{x}}

\def\vxp{\vec{x}\,{}'}

\def\ind{{\mathbb I}}

\def\ve{\varepsilon}

\def\tF{\widetilde{F}}
\def\tG{\widetilde{G}}
\def\tX{\widetilde{X}}
\def\tW{\widetilde{W}}
\def\Q{{\mathcal Q}}
\def\hq{\widehat{q}}
\def\htau{\widehat{\tau}}

\def\W{\widehat{W}}

\def\p{{\mathfrak p}}

\def\Z{\mathbb{Z}}

\def\vu{\vec{u}}

\def\vw{\vec{\omega}}

\def\vy{\vec{y}}

\def\covm{{\mathbb Q}}

\def\vF{\vec{F}}
\def\covd{{\mathbb G}}
\def\ggm{\gamma}

\def\de{{\rm d}}

\def\bJ{{\underline{J}}}

\begin{document}

\begin{frontmatter}
\title{Gibbs Measures and Phase Transitions on Sparse Random Graphs}
\runtitle{Gibbs Measures on Sparse Random Graphs}

\begin{aug}
  \author{\fnms{Amir} \snm{Dembo}
\thanksref{t1}
\ead[label=e1]{amir@math.stanford.edu}},
  \author{\fnms{Andrea} \snm{Montanari}
\thanksref{t1}
\ead[label=e2]{montanari@stanford.edu}}
\thankstext{t1}{Research partially funded by NSF grant \#DMS-0806211.}
  \runauthor{Dembo et al.}
  \address{Stanford University.\\
  \printead{e1,e2}}
\end{aug}

\newcommand{\eqnsection}{\renewcommand{\theequation}{\thesection.\arabic{equation}}
      \makeatletter \csname @addtoreset\endcsname{equation}{section}\makeatother}
\renewcommand{\theenumi}{{\roman{enumi}}}
\renewcommand{\labelenumi}{(\alph{enumi})}
\renewcommand{\labelenumii}{\roman{enumii}.}

\begin{abstract}
Many problems of interest in computer science and information
theory can be phrased in terms of a probability distribution over 
discrete variables
associated to the vertices of a large (but finite) sparse graph. In recent
years, considerable progress has been achieved by viewing these distributions
as Gibbs measures and applying to their study heuristic tools 
from statistical physics. We review this approach and provide some results
towards a rigorous treatment of these problems.
\end{abstract}

\begin{keyword}[class=AMS]
\kwd[Primary ]{60B10}
\kwd{60G60}\kwd{82B20}
\end{keyword}

\begin{keyword}
Random graphs, Ising model, Gibbs measures, Phase transitions, Spin models, 
          Local weak convergence.
\end{keyword}

\end{frontmatter}

\setcounter{tocdepth}{2}
\tableofcontents

\section{Introduction}\label{ch:Intro}
\setcounter{equation}{0}

Statistical mechanics is a rich source of fascinating 
phenomena that can be, at least in principle, fully understood in
terms of probability theory. Over the last two decades,
probabilists have tackled this challenge with much
success.  Notable examples include percolation theory \cite{Grimmett}, 
interacting particle systems \cite{Liggett}, and most recently, 
conformal invariance. Our focus here is on another area of
statistical mechanics, the theory of Gibbs measures, which
provides a very effective and flexible way to define collections
of `locally dependent' random variables. 

The general abstract theory of Gibbs measures is fully 
rigorous from a mathematical point of view
\cite{Georgii}. However,
when it comes to understanding the properties of specific Gibbs measures,
i.e. of specific models, a large gap persists between 
physicists heuristic methods and the scope of 
mathematically rigorous techniques. 

This paper is devoted to somewhat non-standard,
family of models, namely Gibbs measures on sparse random graphs.
Classically, statistical mechanics has been motivated by the 
desire to understand the physical behavior of materials, 
for instance the phase changes of water under temperature change,
or the permeation or oil in a porous material. This naturally
led to three-dimensional models for such phenomena. The discovery
of `universality' (i.e. the observation that many qualitative features do 
not depend on the microscopic details of the system),
led in turn to the study of models on three-dimensional lattices,
whereby the elementary degrees of freedom (spins)
are associated with the vertices of of the lattice.
Thereafter, $d$-dimensional lattices (typically $\Z^d$),
became the object of interest upon realizing that significant 
insight can be gained through such a generalization.

The study of statistical mechanics models `beyond $\Z^d$' is not directly 
motivated by physics considerations. Nevertheless, physicists have
been interested in models on other graph structures for quite a long 
time (an early example is \cite{Dyson}). Appropriate graph structures
can simplify considerably the treatment of a specific model, and sometimes
allow for sharp predictions. Hopefully some qualitative features of this 
prediction survive on $\Z^d$. 

Recently this area has witnessed significant progress and renewed interest
as a consequence of motivations coming from computer science, 
probabilistic combinatorics and statistical inference. In these
disciplines, one is often interested in understanding the 
properties of (optimal) solutions of a large 
set of combinatorial constraints. As a typical example, consider
a linear system over GF$[2]$, $A\,\ux = \ub$ mod $2$,
with $A$ an $n\times n$ binary matrix and $\ub$ a binary vector of 
length $n$. Assume that $A$ and $\ub$ are drawn from random matrix/vector 
ensemble. Typical questions are:
What is the probability that such a linear system admits a solution?
Assuming a typical realization does not admit a solution, 
what is the maximum number of equations that can, typically, be satisfied?

While probabilistic combinatorics 
developed a number of ingenious techniques to deal with these questions, 
significant progress has been achieved recently by employing novel
insights from statistical physics (see \cite{MezardMontanari}). 
Specifically, one first defines a Gibbs measure associated
to each instance of the problem at hand, then analyzes
its properties using statistical physics techniques, 
such as the cavity method. While non-rigorous, this approach 
appears to be very systematic and to provide many sharp 
predictions.
 
It is clear at the outset that, for `natural' distributions
of the binary matrix $A$, the above problem does not have any $d$-dimensional 
structure. Similarly, in many interesting examples, one 
can associate to the Gibbs measure a graph that is sparse and random, 
but of no finite-dimensional structure. Non-rigorous statistical mechanics 
techniques appear to provide detailed predictions about general
Gibbs measures of this type. It would be highly desirable
--and in principle possible-- to develop a fully mathematical theory of 
such Gibbs measures. The present paper provides
a unified presentation of a few results in this direction.

In the rest of this section, 
we proceed with a more detailed overview of the topic, proposing 
certain fundamental questions the answer to which plays an important 
role within the non-rigorous statistical mechanics analysis.
We illustrate these questions on the relatively well-understood 
Curie-Weiss (toy) model and explore a few additional 
motivating examples.

Section \ref{ch:IsingChapter} focuses on a specific 
example, namely the ferromagnetic Ising model on sequences
of locally tree-like graphs. Thanks to its monotonicity properties,
detailed information can be gained on this model.

A recurring prediction of statistical mechanics studies is that
Bethe-Peierls approximation is asymptotically tight in 
the large graph limit, for sequences of locally tree-like 
graphs. Section \ref{ch:Bethe} provides a mathematical formalization
of Bethe-Peierls approximation. We also prove there that, under 
an appropriate correlation decay condition, Bethe-Peierls approximation is
indeed essentially correct on graphs with large girth.

In Section \ref{chap:Coloring} we consider a more challenging, and as
of now, poorly understood, example: proper colorings of a sparse random graph.
A fascinating `clustering' phase transition is predicted to occur 
as the average degree of the graph crosses a certain threshold.
Whereas the detailed description and verification 
of this phase transition remains an open problem,
its relation with the appropriate notion of correlation decay 
(`extremality'), is the subject of Section \ref{ch:Reco}.

Finally, it is common wisdom in statistical mechanics 
that phase transitions should be accompanied by a specific
`finite-size scaling' behavior. 
More precisely, a phase transition corresponds to 
a sharp change in some property of the model when a control parameter
crosses a threshold. In a finite system, the dependence on 
any control parameter is smooth, and the change
and takes place in a window whose width decreases with the system size. 
Finite-size scaling broadly refers to a description of the system behavior
within this window. 
Section \ref{ch:XORSAT} presents a model in which 
finite-size scaling can be determined in detail.

%
%
\subsection{The Curie-Weiss model and some general definitions}
\label{sec:Curie-Weiss}

The Curie-Weiss model is deceivingly simple, but is a good 
framework to start illustrating some important ideas.
For a detailed study of
this model we refer to \cite{NewmanEllis}.

\subsubsection{A story about opinion formation}
\label{sec:Opinion}

At time zero, each of $n$ individuals takes one of two opinions
$X_i(0) \in \{ +1,-1\}$ independently and uniformly at random for
$i\in [n] = \{1,\dots,n\}$. At each subsequent time $t$, one 
individual $i$, chosen uniformly at random,
computes the opinion imbalance
\begin{eqnarray}
M \equiv \sum_{j=1}^n X_j\, ,
\end{eqnarray}
and $M^{(i)} \equiv M-X_i$. Then, he/she changes his/her opinion
with probability
\begin{equation}
p_{\mbox{\tiny flip}}(\uX) = \left\{
\begin{array}{l c} 
\exp(-2\beta |M^{(i)}|/n) & \mbox{ if } \; M^{(i)} X_i > 0 \
\, , 
\\
1 & \mbox{otherwise.}
\end{array}
\right.
\label{eq:PFlip}
\end{equation}
Despite its simplicity, this model raises several interesting questions.
\begin{enumerate}
\item[(a).] How long does is take for the process $\uX(t)$
to become approximately stationary?
\item[(b).] How often do individuals change opinion in the stationary state?
\item[(c).] Is the typical opinion pattern strongly polarized (\emph{herding})?
\item[(d).] If this is the case, how often does the popular opinion change?
\end{enumerate}
We do not address question (a) here, but we will address some
version of questions (b)--(d). 
More precisely, this dynamics (first studied in statistical physics
under the name of \emph{Glauber} or \emph{Metropolis} 
dynamics) is an aperiodic 
irreducible Markov chain whose unique stationary measure is 
\begin{eqnarray}
\mu_{n,\beta}
(\ux) = \frac{1}{Z_n(\beta)}\, \exp\Big\{\frac{\beta}{n}\sum_{(i,j)}x_i
x_j\Big\}\, .\label{eq:CurieWeiss}
\end{eqnarray}
To verify this, simply check that the dynamics given 
by (\ref{eq:PFlip}) is 
\emph{reversible} with respect to the measure $\mu_{n,\beta}$
of (\ref{eq:CurieWeiss}). Namely,
that $\mu_{n,\beta}(\ux) \prob(\ux\to \ux') = \mu_{n,\beta}(\ux')
 \prob(\ux'\to \ux)$
for any two configurations $\ux$, $\ux'$ (where $\prob(\ux\to \ux')$ 
denotes the one-step transition probability from $\ux$ to $\ux'$).

We are mostly interested in the large-$n$ (population size),
behavior of $\mu_{n,\beta}(\cdot)$ and its dependence on 
$\beta$ (the interaction strength). 
In this context, we have the following `static' versions 
of the preceding questions:
\begin{enumerate}
\item[(b').] What is the distribution of $p_{\mbox{\tiny flip}}(\ux)$
when $\ux$ has distribution $\mu_{n,\beta}(\,\cdot\,)$?
\item[(c').] What is the distribution of the opinion imbalance $M$?
Is it concentrated near $0$ (evenly spread opinions), or far from $0$
(herding)?
\item[(d').] In the herding case: how unlikely are balanced 
($M\approx 0$) configurations?
\end{enumerate}
%
%

\subsubsection{Graphical models}
\label{sec:GeneralDef}

A graph $G=(V,E)$ consists of a set $V$ of
vertices and a set $E$ of edges (where 
an edge is an unordered pair of vertices).
We always assume $G$ to be finite with $|V|= n$ and
often make the identification $V= [n]$.
With $\cX$ a finite set, called the 
\emph{variable domain}, we associate to
each vertex $i\in V$ a variable $x_i \in \cX$,
denoting by $\ux \in \cX^V$ the complete assignment of
these variables and by $\ux_U = \{x_i:\; i\in U\}$
its restriction to $U\subseteq V$. 
\begin{definition}\label{def:spec}
A \emph{bounded specification}
$\upsi\equiv\{\psi_{ij}
:\; (i,j)\in E\}$ for a graph $G$ and
variable domain $\cX$ is a family of
functionals $\psi_{ij}:\cX\times \cX\to [0,\psi_{\max}]$
indexed by the edges of $G$ with
$\psi_{\max}$ a given finite, positive constant
(where for consistency
$\psi_{ij}(x,x')=\psi_{ji}(x',x)$ for
all $x,x' \in \cX$ and $(i,j)\in E$). The 
specification may include in addition 
functions $\psi_{i} : \cX \to [0,\psi_{\max}]$ 
indexed by vertices of $G$.

A bounded specification $\upsi$ for $G$ is \emph{permissive} 
if there exists a positive constant $\kappa$ and
a `permitted state' $x_i^{\rm p}\in \cX$ for each $i\in V$,
such that
$\min_{i,x'} \psi_{i}(x') \ge \kappa \psi_{\max}$ and
$$
\min_{(i,j) \in E,  x' \in \cX}
\psi_{ij}(x_i^{\rm p},x')
=
\min_{(i,j) \in E,  x' \in \cX}
\psi_{ij}(x',x_j^{\rm p})
\ge \kappa \psi_{\max} \equiv \psi_{\min} \,.
$$
\end{definition}

The \emph{graphical model} associated with 
a graph-specification pair $(G,\upsi)$ is
the \emph{canonical probability measure} 
\begin{eqnarray}
\mu_{G,\upsi} (\ux) = \frac{1}{Z(G,\upsi)}\,
\prod_{(i,j)\in E}\,\psi_{ij}(x_i,x_j)
\prod_{i \in V}\,\psi_{i}(x_i) 
\label{eq:Canonical}
\label{eq:GeneralGraphModel}
\end{eqnarray}
and the corresponding \emph{canonical stochastic process}
is the collection $\uX = \{X_i:\, i\in V\}$ of 
$\cX$-valued random variables having joint distribution  
$\mu_{G,\upsi} (\cdot)$. 

One such example is the distribution
(\ref{eq:CurieWeiss}), where $\cX=\{+1,-1\}$, 
$G$ is the complete graph over 
$n$ vertices and $\psi_{ij}(x_i,x_j)=\exp(\beta x_i x_j/n)$.
Here $\psi_i(x) \equiv 1$. It is sometimes convenient to introduce
a `magnetic field' (see for instance Eq.~(\ref{eq:CurieWeissField}) 
below). This corresponds to
taking  $\psi_i(x_i)=\exp(B x_i)$.
\vspace{0.2cm}

Rather than studying graphical models at this level of generality, 
we focus on a few concepts/tools that have been the subject of recent
research efforts.
\vspace{0.2cm}

{\bf Coexistence.} Roughly speaking, we say that 
a model $(G,\upsi)$ exhibits coexistence if the 
corresponding measure $\mu_{G,\upsi}(\cdot)$ 
decomposes into a convex combination of well-separated lumps. 
To formalize this notion, we consider sequences of 
measures $\mu_n$ on graphs $G_n = ([n],E_n)$, and say that 
coexistence occurs if, for each $n$, there exists a partition 
$\Omega_{1,n},\dots, \Omega_{r,n}$ 
of the configuration space $\cX^n$ 
with $r=r(n) \ge 2$, such that 
\begin{enumerate}
\item[(a).] The measure of elements of the partition is uniformly 
bounded away from one:
\begin{eqnarray}
\max_{1 \le s \le r} \, \mu_n(\Omega_{s,n}) \le 1-\delta\, .
\end{eqnarray}
\item[(b).] The elements of the partition are separated by `bottlenecks'. 
That is, for some $\epsilon>0$, 
\begin{eqnarray}\label{eq:bot-def}
\max_{1 \le s \le r} 
\frac{\mu_n(\partial_{\epsilon}\Omega_{s,n})}{\mu_n(\Omega_{s,n})}\to 0\, ,
\end{eqnarray}
as $n \to \infty$, where $\partial_{\epsilon}\Omega$
denotes the $\epsilon$-boundary of $\Omega\subseteq \cX^n$,
\begin{eqnarray}
\partial_{\epsilon}\Omega \equiv\{ \ux\in\cX^n:\, 1\le 
d(\ux,\Omega)\le n\epsilon\}\,,
\end{eqnarray}
with respect to the Hamming 
\footnote{The Hamming distance
$d(\ux,\ux')$ between configurations $\ux$ and $\ux'$ is the number of positions
in which the two configurations differ. Given $\Omega\subseteq \cX^n$,
$d(\ux,\Omega) \equiv \min \{d(\ux,\ux'):\ux'\in \Omega\}$.} 
distance.
The normalization by $\mu_n(\Omega_{s,n})$ removes 
`false bottlenecks' and is in particular needed 
since $r(n)$ often grows (exponentially) with $n$.

Depending on the circumstances, one may further
specify a required rate of decay in (\ref{eq:bot-def}).
\end{enumerate}
We often consider families of models indexed by one (or more)
continuous parameters, such as the inverse temperature $\beta$ in the 
Curie-Weiss model. A phase transition will generically be a sharp threshold
in some property of the measure $\mu(\,\cdot\,)$ as one of these parameters 
changes. In particular, a phase transition can separate values of the 
parameter for which coexistence occurs from those
values for which it does not.
\vspace{0.2cm}

{\bf Mean field models.} 
Intuitively, these are models 
that lack any (finite-dimensional) geometrical structure. 
For instance, models of the form (\ref{eq:Canonical}) with 
$\psi_{ij}$ independent of $(i,j)$ and $G$ 
the complete graph or a regular random graph
are mean field models, whereas models in which $G$ is a 
finite subset of a finite dimensional lattice are not.
To be a bit more precise, the 
Curie-Weiss model belongs to a particular class 
of mean field models in which the measure $\mu(\ux)$
is exchangeable (that is, invariant under 
coordinate permutations).  A wider class of mean field models
may be obtained by considering \emph{random} distributions\footnote{A 
random distribution over $\cX^n$ is just a random variable taking 
values on the $(|\cX|^n-1)$-dimensional probability simplex.} $\mu(\cdot)$
(for example, when either $G$ or $\upsi$ are chosen at random in
(\ref{eq:Canonical})). In this context, 
given a realization of $\mu$, consider $k$ i.i.d. configurations
$\uX^{(1)},\dots ,\uX^{(k)}$, each having distribution $\mu$. 
These `replicas' 
have the unconditional, joint distribution  
\begin{eqnarray}
\mu^{(k)}(\ux^{(1)},\dots,\ux^{(k)}) = \E\left\{
\mu(\ux^{(1)})\cdots \mu(\ux^{(k)})\right\}\, .
\end{eqnarray}
The random distribution $\mu$ is 
a candidate to be a mean field model when  
for each fixed $k$ the measure $\mu^{(k)}$,
viewed as a distribution over $(\cX^k)^n$, is exchangeable
(with respect to permutations of the coordinate indices in $[n]$). 
Unfortunately, while this property suffices in many `natural' 
special cases, there are models that intuitively are not 
mean-field and yet have it. For instance, given a non-random 
measure $\nu$ and a uniformly random permutation $\pi$, 
the random distribution 
$\mu (x_1,\dots,x_n) \equiv \nu (x_{\pi(1)},\dots,x_{\pi(n)})$
meets the preceding requirement yet should not be considered a 
mean field model. While a 
satisfactory mathematical definition of
the notion of mean field models is lacking, 
by focusing on selective examples we examine 
in the sequel the rich array of interesting phenomena 
that such models exhibit.
\vspace{0.2cm}

{\bf Mean field equations.} Distinct variables may be correlated in  
the model (\ref{eq:Canonical}) in very subtle ways. 
Nevertheless, mean field models are often tractable because an
effective `reduction' to local marginals\footnote{In particular,
single variable marginals, or joint distributions of two variables 
connected by an edge.} takes place asymptotically for large sizes
(i.e. as $n\to\infty$).

Thanks to this reduction it is often possible to write a closed system
of equations for the local marginals that 
hold in the large size limit and determine 
the local marginals, up to possibly having finitely many solutions.
Finding the `correct' mathematical definition of this notion is 
an open problem, so we shall instead provide 
specific examples of such equations in a few special 
cases of interest (starting with the Curie-Weiss model).
 
\subsubsection{Coexistence in the Curie-Weiss model}

The model (\ref{eq:CurieWeiss}) appeared for the first time in 
the physics literature as a model for ferromagnets\footnote{A ferromagnet is
a material that acquires a macroscopic spontaneous magnetization at low temperature.}. In this context, the variables $x_i$
are called \emph{spins} and their value represents the direction in which
a localized magnetic moment (think of a tiny compass needle) is pointing.
In certain materials the different magnetic moments favor 
pointing in the same direction, and
physicists want to know whether such interaction may lead to a 
macroscopic magnetization (imbalance), or not.

In studying this and related problems it 
often helps to slightly generalize the model 
by introducing a linear term in the exponent 
(also called a `magnetic field'). More precisely, one considers the 
probability measures
\begin{eqnarray}
\label{eq:CurieWeissField}
\mu_{n,\beta,B}
(\ux) = \frac{1}{Z_n(\beta,B)}\, \exp\Big\{\frac{\beta}{n}\sum_{(i,j)}x_i
x_j+B\sum_{i=1}^n x_i\Big\}\, .
\end{eqnarray}
In this context $1/\beta$ is referred to as the `temperature'
and we shall always assume that $\beta\ge 0$ 
and, without loss of generality, also that $B\ge 0$.

The following estimates on 
the distribution of the magnetization per site 
are the key to our understanding of the large size 
behavior of the Curie-Weiss model (\ref{eq:CurieWeissField}).
\begin{lemma}\label{lemma:CWProbM}
Let $H(x) = -x\log x-(1-x)\log(1-x)$ denote the binary entropy function
and for $\beta \ge 0$, $B \in \reals$ and $m\in [-1,+1]$ set 
\begin{eqnarray}\label{eq:psi-def}
\vphi(m) \equiv 
\vphi_{\beta,B}(m) = B m+\frac{1}{2}\beta m^2 + H\left(\frac{1+m}{2}\right)\, .
\end{eqnarray}
Then, for $\bX \equiv n^{-1} \sum_{i=1}^n X_i$, 
a random configuration $(X_1,\ldots,X_n)$
from the Curie-Weiss model 
and each $m\in S_n \equiv \{-1, -1+2/n,\dots, 1-2/n,1\}$, 
\begin{eqnarray}
\frac{e^{-\beta/2}}{n+1}\, 
\frac{1}{Z_{n}(\beta,B)}
e^{n\vphi(m)}\le
\prob\{\bX = m\}\le \frac{1}{Z_n(\beta,B)}\, e^{n\vphi(m)}\, .
\label{eq:ProbBound}
\end{eqnarray}
\end{lemma}
\begin{proof} Noting that for $M= n m$, 
\begin{eqnarray*}
\prob\{\bX = m\} = 
 \frac{1}{Z_n(\beta,B)}\, \binom{n}{(n+M)/2}\, \exp\Big\{
B M+\frac{\beta M^2}{2n}-\frac{1}{2}\beta\Big\}\, ,
\end{eqnarray*}
our thesis follows by Stirling's approximation of the binomial 
coefficient (for example, see \cite[Theorem 12.1.3]{CoverThomas}).
\end{proof}

A  major role in determining the asymptotic properties of the
measures $\mu_{n,\beta,B}$ 
is played by the \emph{free entropy density} (the term `density' 
refers here to the fact that 
we are dividing by the number of variables),
\begin{eqnarray}\label{eq:free-ent}
\phi_n(\beta,B) = \frac{1}{n}\log Z_n(\beta,B)\, .
\end{eqnarray}
\begin{lemma}\label{lemma:CWFreeEnergy}
For all $n$ large enough we have the following bounds on
the free entropy density $\phi_n(\beta,B)$  
of the (generalized) Curie-Weiss model 
\begin{eqnarray*}
 \phi_*(\beta,B)-\frac{\beta}{2n}-\frac{1}{n}\log\{n(n+1)\}
\le\phi_n(\beta,B) \le \phi_*(\beta,B)+\frac{1}{n}\log(n+1) \, ,
\end{eqnarray*}
where  
\begin{eqnarray}
\phi_*(\beta,B) \equiv \sup\left\{\vphi_{\beta,B}(m):\, m\in[-1,1]\right\}\, .
\label{eq:CWFreeEnergy}
\end{eqnarray}
\end{lemma}
\begin{proof}
The upper bound follows upon summing 
over $m \in S_n$ the upper bound in (\ref{eq:ProbBound}). 
Further, 
from the lower bound in (\ref{eq:ProbBound}) we get that 
\begin{eqnarray*}
\phi_n(\beta,B)\ge
\max\Big\{ \vphi_{\beta,B}(m):\, m\in S_n\Big\}-\frac{\beta}{2n}-\frac{1}{n}
\log(n+1)\, . 
\end{eqnarray*} 
A little calculus shows that maximum of $\vphi_{\beta,B}(\cdot)$ 
over the finite set $S_n$ is not smaller that its maximum over 
the interval $[-1,+1]$ minus $n^{-1} (\log n)$, for all $n$ large enough.
\end{proof}

Consider the optimization problem in Eq.~(\ref{eq:CWFreeEnergy}).
Since $\vphi_{\beta,B}(\cdot)$ is continuous on $[-1,1]$
and differentiable in its interior, with $\vphi'_{\beta,B}(m)\to\pm \infty$
as $m\to \mp 1$, this maximum is achieved at one of the
points $m \in (-1,1)$ where $\vphi_{\beta,B}'(m) =0$. 
A direct calculation shows that the latter 
condition is equivalent to 
\begin{eqnarray}
m = \tanh(\beta m+B)\, .
\label{eq:CWMeanField}
\end{eqnarray}
Analyzing the possible solutions 
of this equation, one finds out that:
\begin{enumerate}
\item[(a).] For $\beta\le 1$, the equation (\ref{eq:CWMeanField})
admits a unique solution $m_*(\beta,B)$
increasing in $B$ with $m_*(\beta,B)\downarrow 0$ as $B\downarrow 0$.
Obviously, $m_*(\beta,B)$ maximizes $\vphi_{\beta,B}(m)$.
\item[(b).]
For $\beta > 1$ there exists $B_*(\beta)>0$ continuously 
increasing in $\beta$ with $\lim_{\beta\downarrow 1} B_*(\beta) = 0$ such that:
$(i)$ for $0\le B<B_*(\beta)$, Eq. (\ref{eq:CWMeanField})
admits three distinct solutions $m_-(\beta,B), m_0(\beta,B),
m_+(\beta,B) \equiv m_*(\beta,B)$ with 
$m_-< m_0\le 0\le  m_+ \equiv m_*$; $(ii)$ for $B=B_*(\beta)$ the solutions 
$m_-(\beta,B)= m_0(\beta,B)$ coincide; $(iii)$
and for $B>B_*(\beta)$ only the positive solution  
$m_*(\beta,B)$ survives. 

Further, for $B \ge 0$  the global maximum 
of $\vphi_{\beta,B}(m)$ over $m \in [-1,1]$ 
is
attained at $m=m_*(\beta,B)$, while $m_{0}(\beta,B)$ and
$m_{-}(\beta,B)$ are (respectively) a local minimum and a local maximum
(and a saddle 
point
when they coincide at $B= B_*(\beta)$).
Since $\vphi_{\beta,0}(\cdot)$ is an even function, in particular 
$m_0(\beta,0) = 0$ and $m_\pm (\beta,0) = \pm m_*(\beta,0)$.
\end{enumerate}

Our next theorem answers question (c') of Section \ref{sec:Opinion}
for the Curie-Weiss model. 
\begin{thm}\label{thm:LargeDevMagn}
Consider $\bX$ of Lemma \ref{lemma:CWProbM}
and the relevant solution 
$m_*(\beta,B)$ of equation (\ref{eq:CWMeanField}).
If either $\beta\le 1$ or $B>0$, then 
for any $\ve>0$ there exists $C(\ve)>0$ such that, for all $n$ large enough
\begin{eqnarray}
\prob\left\{\left|\bX-m_*(\beta,B)\right|\le \ve\right\}
\ge 1-e^{-nC(\ve)}\, .\label{eq:LDevMagn}
\end{eqnarray}
In contrast, if $B=0$ and $\beta>1$, then 
for any $\ve>0$ there exists $C(\ve)>0$ 
such that, for all $n$ large enough
\begin{eqnarray}
\prob \left\{\left|\bX-m_*(\beta,0)\right|\le \ve\right\}
=\prob\left\{\left|\bX+m_*(\beta,0)\right|\le \ve\right\}
\ge\frac{1}{2}-e^{-nC(\ve)}\, .\label{eq:LDevMagn2}
\end{eqnarray}
\end{thm}
\begin{proof}
Suppose first that either $\beta\le 1$ or $B>0$, in which case
$\vphi_{\beta,B}(m)$ has the unique non-degenerate global 
maximizer $m_*=m_*(\beta,B)$. Fixing $\ve >0$ and setting 
$I_\ve=[-1,m_*-\ve]\cup[m_*+\ve,1]$, by Lemma \ref{lemma:CWProbM}
\begin{eqnarray*}
\prob\{\bX \in I_\ve\} \le\frac{1}{Z_{n}(\beta,B)}\, (n+1)\,
\exp\Big\{n\max[\vphi_{\beta,B}(m):\, m \in I_\ve ]\Big\}\, .
\end{eqnarray*}
Using Lemma \ref{lemma:CWFreeEnergy} we then find that 
\begin{eqnarray*}
\prob\{\bX \in I_\ve\}\le (n+1)^3 e^{\beta/2}\,
\exp\Big\{n\max[\vphi_{\beta,B}(m)-\phi_*(\beta,B):\, m \in I_\ve]\Big\}\,,
\end{eqnarray*}
whence the bound of (\ref{eq:LDevMagn}) follows.

The bound of (\ref{eq:LDevMagn2}) is proved analogously, using the 
fact that $\mu_{n,\beta,0}(\ux)=\mu_{n,\beta,0}(-\ux)$.
\end{proof}

We just encountered our first example of coexistence (and of phase transition).
\begin{thm}\label{thm:cw-coex}
The Curie-Weiss model shows coexistence if and only if
$B=0$ and $\beta > 1$.
\end{thm}
\begin{proof}
We will limit ourselves to the `if' part of this statement: 
for $B=0$, $\beta >1$, the Curie-Weiss model shows coexistence. 
To this end, we simply check that the partition of 
the configuration space $\{+1,-1\}^n$ to  
$\Omega_{+} \equiv\{\ux:\, \sum_i x_i \ge 0\}$ and 
$\Omega_{-} \equiv\{\ux:\, \sum_i x_i < 0 \}$ 
satisfies the conditions in Section \ref{sec:GeneralDef}.
Indeed, it follows immediately from (\ref{eq:LDevMagn2})
that choosing a positive $\epsilon < m_*(\beta,0)/2$, we have
\begin{eqnarray*}
\mu_{n,\beta,B}(\Omega_{\pm})\ge \frac{1}{2}-e^{-Cn}, \;\;\;\;\; 
\mu_{n,\beta,B}(\partial_{\epsilon}\Omega_{\pm})\le e^{-Cn}\, ,
\end{eqnarray*}
for some $C>0$ and all $n$ large enough, which is the thesis.
\end{proof}
 
\subsubsection{The Curie-Weiss model: Mean field equations}

We have just encountered our first example of coexistence and 
our first example of phase transition. We further claim that 
the identity (\ref{eq:CWMeanField}) can be `interpreted' as
our  first example of a \emph{mean field equation}  
(in line with the discussion of Section \ref{sec:GeneralDef}). 
Indeed, assuming throughout this section   
not to be on the coexistence line $B=0$, $\beta>1$,
it follows from Theorem \ref{thm:LargeDevMagn} that
$\E\,X_i = \E\,\bX\approx m_*(\beta,B)$.\footnote{We 
use $\approx$ to indicate
that we do not provide the approximation error, nor plan to 
rigorously prove that it is small.} 
Therefore, the identity (\ref{eq:CWMeanField}) can be rephrased as 
\begin{eqnarray}
\E\, X_i\approx \tanh\Big\{B+\frac{\beta}{n}\sum_{j\in V}
\E\, X_{j}\Big\}\,,
\label{eq:CWMeanField2}
\end{eqnarray}
which, in agreement with our general description of mean field equations,
is a closed form relation between the 
local marginals under the measure $\mu_{n,\beta,B}(\cdot)$.

We next re-derive the equation (\ref{eq:CWMeanField2}) 
directly out of the concentration in probability of $\bX$. 
This approach is very useful, for in more complicated models 
one often has mild bounds on the fluctuations of $\bX$ while 
lacking fine controls such as in Theorem \ref{thm:LargeDevMagn}. 
To this end, we start by proving the following 
`cavity' estimate.\footnote{Cavity methods of statistical physics aim
at understanding thermodynamic limits $n \to \infty$ 
by first relating certain quantities for systems of size 
$n \gg 1$ to those in systems of size $n' = n + O(1)$.}
\begin{lemma}\label{lemma:ChangeSize}
Denote by $\E_{n,\beta}$ and $\Var_{n,\beta}$ the
expectation and variance with respect to the Curie-Weiss model
with $n$ variables at inverse temperature $\beta$ (and magnetic field $B$).
Then, for $\beta' = \beta (1+1/n)$, $\bX = n^{-1}\sum_{i=1}^n X_i$ 
and any $i\in [n]$,
\begin{eqnarray}
\left|\E_{n+1,\beta'}X_i - \E_{n,\beta} X_i\right|\le \beta\sinh(B+\beta)
\sqrt{\Var_{n,\beta}(\bX)}\,.
\end{eqnarray} 
\end{lemma}
\begin{proof}
By direct computation, for any function $F: \{+1,-1\}^n\to\reals$,
\begin{eqnarray*}
\E_{n+1,\beta'}\{F(\uX)\} = \frac{\E_{n,\beta}\{F(\uX)\cosh(B+\beta\bX)\}}
{\E_{n,\beta}\{\cosh(B+\beta\bX)\}}\, .
\end{eqnarray*}
Therefore, with $\cosh(a) \ge 1$ we get by Cauchy-Schwarz that  
\begin{eqnarray*}
&& |\E_{n+1,\beta'}\{F(\uX)\} - \E_{n,\beta}\{F(\uX)\}| \le
|{\rm Cov}_{n,\beta}\{\F(\uX), \cosh(B+\beta\bX)\}| \\
&\le& ||F||_{\infty}\sqrt{\Var_{n,\beta}(\cosh(B+\beta\bX))}
\le ||F||_{\infty}\beta\sinh(B+\beta) \sqrt{\Var_{n,\beta}(\bX)}\, ,
\end{eqnarray*}
where the last inequality is 
due to the Lipschitz behavior of $x\mapsto\cosh (B+\beta x)$ 
together with the bound $|\bX|\le 1$. 
\end{proof}

The following theorem provides a rigorous version of 
Eq.~(\ref{eq:CWMeanField2}) for $\beta\le 1$ or $B>0$. 
\begin{thm}
There exists a constant $C(\beta,B)$ such that for any $i \in [n]$,
\begin{eqnarray}
\Big|\E X_i- \tanh\big\{B+\frac{\beta}{n}\sum_{j\in V }
\E\, X_{j}\big\}\Big|\le C(\beta,B)\sqrt{\Var (\bX)}\, .
\end{eqnarray}
\end{thm}
\begin{proof} In the notations of Lemma \ref{lemma:ChangeSize}
recall that $\E_{n+1,\beta'}X_i$ is independent of $i$ and 
so upon fixing $(X_1,\ldots,X_n)$ we get by direct computation that
\begin{eqnarray*}
\E_{n+1,\beta'}\{X_{i}\} = 
\E_{n+1,\beta'}\{X_{n+1}\} = \frac{\E_{n,\beta}\sinh(B+\beta\bX)}
{\E_{n,\beta}\cosh(B+\beta\bX) }\, .
\end{eqnarray*}
Further notice that (by the Lipschitz property of $\cosh (B+\beta x)$ and
$\sinh (B+\beta x)$  together with the bound $|\bX|\le 1$),
\begin{eqnarray*}
|\E_{n,\beta}\sinh(B+\beta\bX)-\sinh(B+\beta\E_{n,\beta}\bX)|\le 
\beta \cosh(B+\beta)
\sqrt{\Var_{n,\beta} (\bX)}\, 
,\\
|\E_{n,\beta}\cosh(B+\beta\bX)-\cosh(B+\beta\E_{n,\beta}\bX)|\le 
\beta \sinh(B+\beta)
\sqrt{\Var_{n,\beta}(\bX)}\, 
.
\end{eqnarray*}
Using 
the inequality $|a_1/b_1-a_2/b_2|\le |a_1-a_2|/b_1 + a_2|b_1-b_2|/b_1
b_2$ we thus have here (with $a_i \ge 0$ and $b_i \ge \max(1,a_i)$), 
that 
\begin{eqnarray*}
\Big|\E_{n+1,\beta'} \{X_i\}- \tanh\big\{B+\frac{\beta}{n}\sum_{j=1}^n
\E_{n,\beta}\, X_{j}\big\}\Big|\le C(\beta,B)\sqrt{\Var_{n,\beta}(\bX)}\,.
\end{eqnarray*}
At this point you get our thesis by applying Lemma \ref{lemma:ChangeSize}.
\end{proof}


%
%
\subsection{Graphical models: examples}\label{sec:other-models}

We next list a few examples of graphical models,  
originating at different domains of science and engineering. 
Several other examples that fit the same framework are discussed
in detail  in \cite{MezardMontanari}.
\subsubsection{Statistical physics} $~$

\noindent
{\bf Ferromagnetic Ising model.} The ferromagnetic Ising model is arguably the 
most studied model in statistical physics. It is defined 
by the Boltzmann distribution
\begin{eqnarray}
\mu_{\beta,B}(\ux) = \frac{1}{Z(\beta,B)}\,\exp\Big\{\beta\sum_{(i,j)\in E}
x_ix_j + B\sum_{i\in V}x_i\Big\}\, ,
\label{eq:IsingModel}
\end{eqnarray}
over $\ux = \{x_i:\, i\in V\}$, with $x_i\in\{+1,-1\}$,
parametrized by the
`magnetic field' $B \in \reals$ 
and `inverse temperature' $\beta \ge 0$,
where the partition function $Z(\beta,B)$ is
fixed by the normalization condition $\sum_{\ux} \mu(\ux)=1$.
The interaction between vertices $i, j$ connected by an edge
pushes the variables $x_i$ and $x_j$ towards taking the same value.
It is expected that this leads to a global  alignment of
the variables (spins) at low temperature, for a large family of graphs.
This transition should be analogue to the one we found for
the Curie-Weiss model,
but remarkably little is known about Ising models on \emph{general graphs}. 
In Section \ref{ch:IsingChapter} 
we consider the case of random sparse graphs. 

\vspace{0.1cm}
\noindent
{\bf Anti-ferromagnetic Ising model.} This model takes the same form
(\ref{eq:IsingModel}), but with $\beta<0$.\footnote{In the literature 
one usually introduces
explicitly a minus sign to keep $\beta$ positive.} 
Note that if $B=0$ and the graph is bipartite (i.e. 
if there exists a  partition $V=V_1 \cup V_2$ such that 
$E \subseteq V_1 \times V_2$), then 
this model is equivalent to the ferromagnetic one
(upon inverting the signs of $\{x_i, i \in V_1\}$). 
However, on non-bipartite graphs the anti-ferromagnetic model 
is way more complicated than the ferromagnetic one, and 
even determining the most likely (lowest energy)
configuration is a difficult matter. Indeed, 
for $B=0$ the latter is equivalent to  
the celebrated max-cut problem from theoretical computer science.

\vspace{0.1cm}
\noindent
{\bf Spin glasses.} An instance of the Ising spin glass is
defined by a graph $G$, together with edge weights $J_{ij}\in \reals$,
for $(i,j)\in E$. Again variables are binary $x_i\in\{+1,-1\}$ and
\begin{eqnarray}\label{eq:spin-glass}
\mu_{\beta,B,\bJ}(\ux) = \frac{1}{Z(\beta,B,\bJ)}
\,\exp\Big\{\beta\sum_{(i,j)\in E} J_{ij}x_ix_j + B\sum_{i\in V}x_i\Big\}\, .
\end{eqnarray}
In a spin glass model the `coupling constants' $J_{ij}$ are random
with even distribution (the canonical examples being $J_{ij}\in \{+1,-1\}$
uniformly and $J_{ij}$ centered Gaussian variables).
One is interested in determining the asymptotic properties 
as $n = |V| \to\infty$ of 
$\mu_{n,\beta,B,\bJ}(\,\cdot\,)$ for a typical realization of the coupling
$\bJ \equiv \{J_{ij}\}$.
 
\subsubsection{Random constraint satisfaction problems}\label{sec:random-CSP}

A constraint satisfaction problem (CSP) consists of a
finite set $\cX$ (called the variable domain), 
and a class $\cC$ of possible 
constraints (i.e. indicator functions), each of 
which involves finitely many $\cX$-valued variables $x_i$. 
An instance of this problem is then specified by a positive  
integer $n$ (the number of variables), and a set of $m$ 
constraints
involving only the variables $x_1,\ldots,x_n$ (or a subset thereof).
A solution of this instance is an assignment in $\cX^n$ 
for the variables $x_1,\ldots,x_n$ 
which satisfies all $m$ constraints.

In this context, several questions are of interest within computer science:
\begin{enumerate}
\item \emph{Decision problem}. Does the given instance have a solution?
\item \emph{Optimization problem}. Maximize the number of satisfied constraints.%
\item \emph{Counting problem}. Count the number of solutions.
\end{enumerate}

There are many ways of associating a graphical model to an instance of CSP. 
If the instance admits a solution, then one option is to
consider the uniform measure over all such solutions. Let us see
how this works in a few examples.

\begin{figure}
\center{
\includegraphics[angle=0,width=0.35\columnwidth]{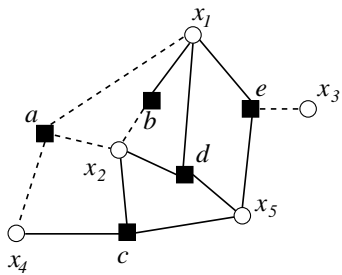}}
\caption{Factor graph representation of the satisfiability
formula $(\bar{x}_1\vee\bar{x}_2\vee\bar{x}_4)\wedge
(x_1\vee\bar{x}_2)\wedge(x_2\vee x_4\vee x_5)\wedge(x_1\vee x_2\vee x_5)
\wedge (x_1\vee\bar{x}_2\vee x_5)$. Edges are continuous or dashed depending
whether the corresponding variable is directed or negated in the clause.}
\label{fig:Factor}
\end{figure}

\vspace{0.1cm}

{\bf Coloring.} A proper $q$-coloring of a graph $G$ 
is an assignment of colors in $[q]$ to the vertices
of $G$ such that no edge has both endpoints of the same color.
The corresponding CSP has variable domain 
$\cX= [q]$ and the possible constraints in $\cC$ 
are indexed by pairs of indices $(i,j)\in V\times V$, where 
the constraint $(i,j)$ is
satisfied if and only if $x_i\neq x_j$.

Assuming that a graph $G$ admits a proper $q$-coloring, the uniform 
measure over the set of possible solutions is
\begin{eqnarray}
\mu_{G}(\ux) = \frac{1}{Z_G}\, \prod_{(i,j)\in E}\, \ind(x_i\neq x_j)\, , 
\end{eqnarray}
with $Z_G$ counting the number of proper $q$-colorings of $G$. 
\vspace{0.1cm}

{\bf $k$-SAT.} In case of $k$-satisfiability (in short, $k$-SAT), 
the variables are binary
$x_i\in \cX = \{0,1\}$ and each 
constraint is of
the form $(x_{i(1)},\dots,x_{i(k)})\neq (x^*_{i(1)},\dots,x^*_{i(k)})$
for some prescribed  $k$-tuple $(i(1),\dots,i(k))$ of indices in $V=[n]$
and their prescribed values $(x^*_{i(1)},\dots,x^*_{i(k)})$. In this
context constraints are often referred to as `clauses' and can be written
as the disjunction (logical OR) of $k$ variables or their negations.
The uniform measure over solutions of an instance of this problem,
if such solutions exist, 
is then 
\begin{eqnarray*}
\mu(\ux) = \frac{1}{Z}\, \prod_{a=1}^m\, \ind\Big((x_{i_a(1)},\dots,x_{i_a(k)})
\neq (x^*_{i_a(1)},\dots,x^*_{i_a(k)})\Big)\, , 
\end{eqnarray*}
with $Z$ counting the number of 
solutions. 
An instance can be associated to a \emph{factor graph}, cf. 
Fig.~\ref{fig:Factor}. This is a bipartite
graph having two types of nodes: variable nodes in $V = [n]$
denoting the unknowns $x_1,\ldots,x_n$ 
and function (or factor) nodes in $F =[m]$
denoting the specified constraints. Variable node $i$ 
and function node $a$ are connected by 
an edge in the factor graph 
if and only if variable $x_i$ appears in the $a$-th clause,
so $\da = \{i_a(1),\dots,i_a(k)\}$ and $\di$ corresponds to
the set of clauses in which $i$ appears.

In general, such 
a construction associates to arbitrary CSP instance
a factor graph $G=(V,F,E)$. The uniform measure 
over solutions of such an instance is then of the form 
\begin{eqnarray}\label{eq:FCanonical}
\mu_{G,\upsi} (\ux) 
= \frac{1}{Z(G,\upsi)}\, \prod_{a \in F} \psi_a(\ux_{\da}) \, , 
\end{eqnarray}
for a suitable choice of $\upsi \equiv \{\psi_a(\cdot): a \in F\}$.
Such measures can also 
be viewed as the zero temperature limit of certain Boltzmann
distributions.
We note in passing that the probability measure 
of Eq.~(\ref{eq:Canonical}) corresponds to the special case 
where all function nodes are of degree two.

\subsubsection{Communications, estimation, detection}

We describe next a canonical way of phrasing 
problems from mathematical engineering  
in terms of graphical models. Though
we do not detail it here, this approach
applies to many specific cases of interest.

Let $X_1,\dots, X_n$ be a collection of i.i.d. `hidden' random variables
with a common distribution $p_0(\,\cdot\,)$
over a finite alphabet $\cX$. We want to estimate 
these variables from a given collection of
observations $Y_1,\dots,Y_m$.
The $a$-th observation (for $a\in [m]$) is a random 
function of the $X_i$'s for which $i\in\da = \{i_a(1),\dots,i_a(k)\}$.
By this we mean that $Y_a$ is conditionally independent 
of all the other variables given $\{X_i:\, i\in \da\}$ and
we write
\begin{eqnarray}
\prob\left\{Y_a\in A|\uX_{\da}=\ux_{\da}\right\} = Q_a(A|\ux_{\da})\, .
\end{eqnarray}
for some probability kernel $Q_a(\,\cdot\,|\,\cdot\,)$.

The \emph{a posteriori} distribution of the hidden variables
given the observations is thus  
\begin{eqnarray}
\mu(\ux|\uy) = \frac{1}{Z(\uy)}\, \prod_{a=1}^{m} Q_a(y_a|\ux_{\da})
\prod_{i=1}^n p_0(x_i)\, .\label{eq:Obs}
\end{eqnarray}
%

%
%
\subsubsection{Graph and graph ensembles}

The structure of the underlying graph $G$ is of 
much relevance for the general measures 
$\mu_{G,\upsi}$ of (\ref{eq:Canonical}).
The same applies in the specific examples we 
have outlined in Section \ref{sec:other-models}. 

As already hinted, we focus here 
on (random) graphs that lack finite dimensional 
Euclidean structure. A few well known    
ensembles of such graphs (c.f. \cite{Rgraphs}) are:
\begin{enumerate}
\item[I.] \emph{Random graphs with a given degree distribution.}
Given a probability distribution $\{P_l\}_{l\ge 0}$ 
over the non-negative integers, for each value of $n$
one draws the graph $G_n$ uniformly at random from the 
collection of all graphs with $n$ vertices of
which precisely $\lfloor n P_k \rfloor$ are of degree $k \ge 1$ 
(moving one vertex from degree $k$ to $k+1$ if 
needed for an even sum of degrees). We will denote this
ensemble by $\graph(P,n)$.
\item[II.] The ensemble of \emph{random $k$-regular graphs} 
corresponds to $P_k=1$ (with $kn$ even).
Equivalently, this is defined by the set of all graphs $G_n$ 
over $n$ vertices with degree $k$,
endowed with the uniform measure.  With a slight abuse
of notation, we will denote it by $\graph(k,n)$.
\item[III.] \emph{Erd\"os-Renyi graphs.} 
This is the ensemble of all graphs $G_n$ with $n$
vertices and $m = \lfloor n\alpha  \rfloor$ edges 
endowed with the uniform measure. 
A slightly modified ensemble is the one in which each edge $(i,j)$
is present independently with probability $n\alpha /\binom{n}{2}$.
We will denote it as $\graph(\alpha,n)$.
\end{enumerate}
As further shown in Section \ref{sec:loc-conv},
an important property of these graph ensembles is that 
they converge locally to trees.  Namely, for any integer $\ell$, 
the depth-$\ell$ neighborhood $\Ball_i(\ell)$
of a uniformly chosen random vertex $i$
converges in distribution as $n \to \infty$
to a certain random tree of depth (at most) $\ell$.
%


\subsection{Detour: The Ising model on the integer lattice}

In statistical physics it is most natural to consider models
with local interactions on a finite dimensional integer 
lattice $\integers^d$, where $d=2$ and $d=3$ are often 
the physically relevant ones. While such models are of course
non-mean field type, taking a short detour we
next present a classical result about
ferromagnetic Ising models on finite subsets of $\integers^2$. 
\begin{thm}
Let $\E_{n,\beta}$ denote expectations
with respect to the ferromagnetic Ising measure (\ref{eq:IsingModel})
at zero magnetic field, in case $G=(V,E)$ 
is a square grid of side $\sqrt{n}$. 
Then, for large $n$ the average magnetization $\bX = n^{-1}\sum_{i=1}^n X_i$
concentrates around zero for high temperature but not for low temperature. 
More precisely, for some $\beta_o>0$,
\begin{eqnarray}
\label{2d-Ising-lowT}
&& \lim_{\beta \to \infty} \inf_n \E_{n,\beta} \{\, |\bX|\,\} = 1 \,, \\
&& \lim_{n \to \infty} \E_{n,\beta} \{\, |\bX|^2\,\} = 0 
\qquad \qquad \qquad \forall \beta < \beta_o  \,. 
\label{2d-Ising-highT}
\end{eqnarray}
\end{thm}
While this theorem and its proof refer to $\integers^2$,  
the techniques we use are more general.

\noindent
{\bf Low temperature: Peierls argument.}
The proof of (\ref{2d-Ising-lowT}) is taken from \cite{Gri64}
and based on the Peierls contour representation for
the two dimensional Ising model.
We start off by reviewing this representation. 
First, given a square grid $G=(V,E)$ of side $\sqrt{n}$ 
in $\integers^2$, for each $(i,j)\in E$ 
draw a perpendicular edge of length one, centered 
at the midpoint of $(i,j)$. Let 
$E^*$ denote the collection of all these perpendicular edges and
$V^*$ the collection of their end points, 
viewed as a finite subset of $\reals^2$.
A \emph{contour} is a simple path on the `dual' graph $G^*=(V^*,E^*)$,    
either closed or with both ends at boundary 
(i.e. degree one) vertices. A closed contour
$C$ divides $V$ to two subsets, the \emph{inside} of
$C$ and the \emph{outside} of $C$. We further 
call as `inside' the smaller of the two subsets 
into which a non-closed contour divides $V$ (an 
arbitrary convention can be used in case the 
latter two sets are of equal size).
A Peierls contours configuration $(\CC,s)$ consists 
of a sign $s \in \{+1,-1\}$ and an 
edge-disjoint finite collection $\CC$ of 
non-crossing contours (that is, whenever two 
contours share a vertex, each of them bends there). 
Starting at an Ising configuration $\ux \in \Omega \equiv \{+1,-1\}^V$ 
note that the set $V_+(\ux) = \{v \in V: x_v = +1\}$ is separated
from $V_- (\ux) = \{v \in V: x_v =-1\}$ by an 
edge-disjoint finite collection $\CC=\CC(\ux)$ of 
non-crossing contours. Further, it is not hard to check
that the non-empty set
$U(\ux) = \{v \in V : v $ not inside any contour from $\CC \}$ 
is either contained in $V_+(\ux)$, in which case 
$s(\ux)=+1$ or in $V_-(\ux)$, in which case $s(\ux)=-1$,
partitioning $\Omega$ to $\Omega_+ = \{ \ux : s(\ux)=+1\}$
and $\Omega_- = \{ \ux : s(\ux)=-1\}$. In the reverse 
direction, the Ising configuration is read off 
a Peierls contours configuration $(\CC,s)$ 
by setting $x_v=s$ when the number of contours $C \in \CC$ 
such that $v \in V$ lies in the inside of $C$ is even 
while $x_v=-s$ when it is odd. The mapping 
$\ux \mapsto -\ux$ exchanges $\Omega_+$ with $\Omega_-$ so 
\begin{equation}\label{eq:peier1}
\E_{n,\beta} [|\bX|] \ge
2 \E_{n,\beta} [\bX \ind(\uX \in \Omega_+)]
= 1 - \frac{4}{n} \E_{n,\beta} [ |V_-(\uX)| \ind(\uX \in \Omega_+)] \,.
\end{equation} 
If $\ux$ is in $\Omega_+$ then $|V_-(\ux)|$ is bounded by the total 
number of vertices of $V$ inside contours of $\CC$, which 
by isoperimetric considerations 
is at most $\sum_{C \in \CC} |C|^2$ 
(where $|C|$ denotes the length of contour $C$). Further, 
our one-to-one correspondence 
between Ising and Peierls contours configurations 
maps the Ising measure at $\beta>0$ to uniform $s \in \{+1,-1\}$ 
independent of $\CC$ whose distribution is the Peierls measure  
$$
\mu_*(\CC) = \frac{1}{Z_* (\beta)} \prod_{C \in \CC} e^{-2\beta |C|} \,.
$$ 
Recall that if a given contour $C$ is in some  
edge-disjoint finite collection $\CC$ of non-crossing contours,
then $\CC'=\CC \setminus C$ is another such collection, with 
$\CC \mapsto \CC'$ injective, from which we easily deduce 
that $\mu_*(C \in \CC) \le \exp(-2 \beta |C|)$ for any fixed contour $C$.
Consequently, 
\begin{align}\label{eq:peier2}
\E_{n,\beta} [ |V_-(\uX)| \ind(\uX \in \Omega_+)] 
&\le \sum_{C} |C|^2 \mu_*(C \in \CC) \nonumber \\
&\le \sum_{\ell \ge 2} \ell^2 N_c (n,\ell) e^{-2\beta \ell} \,,
\end{align}
where $N_c (n,\ell)$ denotes the number of contours of 
length $\ell$ for the square grid of side $\sqrt{n}$. Each
such contour is a length $\ell$ path of a non-reversing nearest
neighbor walk in $\integers^2$ starting at some point in $V^*$.
Hence, $N_c (n,\ell) \le |V^*| 3^{\ell} \le n 3^{\ell+1}$. 
Combining this bound with (\ref{eq:peier1}) and (\ref{eq:peier2}) 
we conclude that for all $n$, 
$$
\E_{n,\beta} [|\bX|] 
\ge 1 - \frac{4}{n} \sum_{\ell \ge 2} \ell^2 N_c (n,\ell) e^{-2\beta \ell} 
\ge 1 - 12 \sum_{\ell \ge 2} \ell^2 3^\ell e^{-2\beta \ell} \,.
$$ 
We are thus done, as this lower bound converges to one for 
$\beta \to \infty$.

\noindent
{\bf High-temperature expansion.}
The proof of (\ref{2d-Ising-highT}), taken from \cite{Fis67}, is
by the method of high-temperature expansion which serves us 
again when dealing with the unfrustrated XORSAT model in Section 
\ref{sec-xorsat}. As in the
low-temperature case, the first step consists of finding an 
appropriate `geometrical' representation. To this end, given a
subset $U \subseteq V$ of vertices, let
$$
Z_U(\beta) = \sum_{\ux} x_U \exp\big\{ \beta \sum_{(i,j) \in E} x_i x_j \big\} 
$$
and denote by $\G(U)$ the set of subgraphs of $G$ having an odd-degree at
each vertex in $U$ and an even degree at all other vertices. Then, 
with $\theta \equiv \tanh(\beta)$ and
$F \subseteq E$ denoting both a subgraph of $G$ and its set of edges,
we claim that
\begin{equation}\label{eq:ht-2dIsing}
Z_U(\beta) = 2^{|V|} (\cosh \beta)^{|E|}\sum_{F \in \G(U)} \theta^{|F|} \,. 
\end{equation}
Indeed, 
$e^{\beta y} = \cosh(\beta) [ 1 + y \theta]$ for $y \in \{+1,-1\}$, 
so by definition
\begin{align*}
Z_U(\beta) &= (\cosh \beta)^{|E|} \sum_{\ux} x_U 
\prod_{(i,j) \in E} [1+ x_i x_j \theta] \\
&= (\cosh \beta)^{|E|} \sum_{F \subseteq E} \theta^{|F|} 
\sum_{\ux} x_U \prod_{(i,j) \in F} x_i x_j \,.
\end{align*}
By symmetry $\sum_{\ux} x_R$ is zero unless each $v \in V$ appears
in the set $R$ an even number of times, in which case the sum is $2^{|V|}$. 
In particular, the latter applies for 
$x_R = x_U \prod_{(i,j) \in F} x_i x_j$ if and only if $F \in \G(U)$
from which our stated high-temperature expansion (\ref{eq:ht-2dIsing}) follows.

We next use this expansion to get a uniform in $n$ decay of 
correlations at all $\beta < \beta_o \equiv \atanh(1/3)$, 
with an exponential rate with respect to the graph distance 
$d(i,j)$. 
More precisely, we claim that 
for any such $\beta$, $n$ and $i,j \in V$ 
\begin{equation}\label{eq:ht-Ising-cor}
\E_{n,\beta} \{ X_i X_j \} \le (1-3\theta)^{-1} (3\theta)^{d(i,j)} \,.
\end{equation} 
Indeed, from (\ref{eq:ht-2dIsing}) we know that 
$$
\E_{n,\beta} \{ X_i X_j \} = \frac{Z_{(i,j)}(\beta)}{Z_{\emptyset}(\beta)}
=\frac{ \sum_{F \in \G(\{i,j\})} \theta^{|F|} }
{\sum_{F' \in \G(\emptyset)} \theta^{|F'|} } \,.
$$
Let $\FF (i,j)$ denote the collection of all simple paths
from $i$ to $j$ in $\integers^2$ and
for each such path $F_{i,j}$, denote by $\G(\emptyset,F_{i,j})$ 
the sub-collection of graphs in $\G(\emptyset)$ that have no 
edge in common with $F_{i,j}$.
The sum of vertex degrees in a connected component of a graph $F$
is even, hence any $F \in \G(\{i,j\})$ contains some path
$F_{i,j} \in \FF (i,j)$. Further, 
$F$ is the edge-disjoint union of $F_{i,j}$ and 
$F' = F \setminus F_{i,j}$ with $F'$ having an even degree 
at each vertex. As $F' \in \G(\emptyset,F_{i,j})$ we thus deduce that   
$$
\E_{n,\beta} \{ X_i X_j \} \le \sum_{F_{i,j} \in \FF(i,j)} 
\theta^{|F_{i,j}|} \frac{ \sum_{F' \in \G(\emptyset, F_{i,j})} \theta^{|F'|} }
{\sum_{F' \in \G(\emptyset)} \theta^{|F'|} } \le
\sum_{F_{i,j} \in \FF(i,j)} \theta^{|F_{i,j}|} \,.
$$   
The number of paths in $\FF (i,j)$ of length $\ell$ is at most $3^\ell$
and their minimal length is $d(i,j)$. Plugging this in the preceding 
bound establishes our correlation decay bound (\ref{eq:ht-Ising-cor}).

We are done now, for there are at most $8d$ vertices in $\integers^2$
at distance $d$ from each $i \in \integers^2$. Hence,
\begin{align*}
\E_{n,\beta} \{\,|\bX|^2\,\} &= \frac{1}{n^2} \sum_{i,j \in V}  
\E_{n,\beta} \{ X_i X_j \} \\
&\le 
\frac{1}{n^2 (1-3\theta)} \sum_{i,j \in V} (3\theta)^{d(i,j)}
\le \frac{1}{n (1-3 \theta)} \sum_{d=0}^\infty 8d (3\theta)^d \,,
\end{align*}
which for $\theta<1/3$ decays to zero as $n \to \infty$.
%
%

%
\section{Ising models on locally tree-like graphs}
\label{ch:IsingModel}
\label{ch:IsingChapter}
\setcounter{equation}{0}

A ferromagnetic \emph{Ising model on the finite 
graph $G$} (with vertex set $V$, 
and edge set $E$) is defined by 
the Boltzmann distribution $\mu_{\beta,B}(\ux)$ of (\ref{eq:IsingModel})
with $\beta \ge 0$. In the following it is understood that,
unless specified otherwise, the model is ferromagnetic, and we will
call it `Ising model on $G$.'

For sequences of graphs $G_n = (V_n,E_n)$ of diverging size $n$,
non-rigorous statistical mechanics techniques, such as the `replica'
and `cavity methods,' make a number of predictions
on this model when the graph $G$ `lacks 
any finite-dimensional structure.' 
The most basic quantity in this context is the 
asymptotic \emph{free entropy density}, cf. Eq.~(\ref{eq:free-ent}),
\begin{eqnarray}
\phi(\beta,B) \equiv \lim_{n\to\infty}\frac{1}{n}\log Z_n(\beta,B)\,. 
\label{eq:FreeEnergy}
\end{eqnarray}
The Curie-Weiss model, cf. Section \ref{sec:Curie-Weiss},   
corresponds to the complete graph $G_n=K_n$.
Predictions exist for a much wider class of models and 
graphs, most notably, sparse random graphs with bounded average degree
that arise in a number of problems from 
combinatorics and theoretical computer science (c.f. 
the examples of Section \ref{sec:random-CSP}).
An important new feature of  sparse graphs is that one can
introduce a notion of distance between vertices
as the length of shortest path connecting them.
Consequently, 
phase transitions and coexistence
can be studied with respect to the correlation decay properties
of the underlying measure. It turns out that this approach 
is particularly fruitful and allows to characterize these phenomena 
in terms of appropriate features of Gibbs measures on infinite trees.
This direction is pursued in \cite{OurPNAS} in the case of 
random constraint satisfaction problems.

Statistical mechanics also provides methods for approximating 
the local marginals of the Boltzmann measure of (\ref{eq:IsingModel}).
Of particular interest is the algorithm
known in artificial intelligence and computer 
science under the name of \emph{belief propagation}.
Loosely speaking, this procedure consists of solving 
by iteration certain mean field (cavity) equations.
Belief propagation is shown in \cite{ising} 
to converge exponentially fast for an Ising model on any graph
(even in a low-temperature regime lacking uniform decorrelation), 
with resulting asymptotically 
tight estimates for large locally tree-like graphs 
(see Section \ref{sec:algorithmic}).

\subsection{Locally tree-like graphs and conditionally independent trees}
\label{sec:loc-conv}

We follow here \cite{ising},
where the asymptotic free entropy density 
(\ref{eq:FreeEnergy}) is determined rigorously  
for certain sparse graph sequences $\{G_n\}$
that converge locally to trees.
In order to make this notion more precise, we denote by
$\Ball_i(t)$ the subgraph induced by vertices of $G_n$ 
whose distance from $i$ is at most $t$.
Further, given two rooted trees $T_1$ and $T_2$ of the
same size, we write $T_1\simeq T_2$ if 
$T_1$ and $T_2$
are identical upon 
labeling their vertices in a 
breadth first fashion following lexicographic order among siblings.
\begin{definition}
Let $\prob_n$ denote the law of the ball $\Ball_i(t)$
when $i\in V_n$ is a uniformly 
chosen 
random vertex.
We say that $\{G_n\}$ \emph{converges locally} to the random rooted
tree $\Tree$ if, for any finite $t$ and any 
rooted tree $T$ of depth at most $t$, 
\begin{eqnarray}
\lim_{n\to\infty}\prob_n\{\Ball_i(t) \simeq T\} = \prob\{
\Tree(t) \simeq T\}\,, 
\end{eqnarray}
where $\Tree(t)$ denotes the subtree of first $t$ generations 
of $\Tree$.
 
We also say that $\{G_n\}$ is \emph{uniformly sparse} if  
\begin{eqnarray}
\lim_{l \to \infty} \limsup_{n \to \infty}
\frac{1}{|V_n|} \sum_{i\in V_n}|\di|\, \ind(|\di|\ge l) = 0 \, ,
\end{eqnarray}
where $|\di|$ denotes the size of the set
$\di$ of neighbors of $i\in V_n$ (i.e. the degree of $i$).
\end{definition}

The proof that for locally tree-like graphs 
$\phi_n(\beta,B) =\frac{1}{n} \log Z_n(\beta,B)$ converges 
to (an explicit) limit $\phi(\beta,B)$ consists of two steps
\begin{enumerate}
\item[(a).] Reduce the computation of 
$\phi_n(\beta,B)$ to computing expectations
of local (in $G_n$)
quantities with respect to the Boltzmann measure (\ref{eq:IsingModel}).
This is achieved by noting that the derivative of $\phi_n(\beta,B)$ 
with respect to $\beta$ is a sum of such expectations.
\item[(b).] Show that under the Boltzmann measure (\ref{eq:IsingModel})
on $G_n$ expectations of local quantities are, for $t$ and $n$ large,
well 
approximated by the same expectations with respect to an Ising model on the
associated random tree $\Tree(t)$ 
(a philosophy related to that of \cite{Aldous}).
\end{enumerate}

The key is of course step (b), and the challenge is  
to carry it out when the parameter 
$\beta$ is large and we no longer have 
uniqueness of the Gibbs measure on the limiting tree $\Tree$. 
Indeed, this is done in \cite{ising} for 
the following collection of
trees of conditionally independent (and of bounded
average) offspring numbers.
\begin{definition}\label{def-cit}
An infinite labeled tree $\Tree$ rooted at the vertex $\root$
is called \emph{conditionally independent} if
for each integer $k \ge 0$, conditional on
the subtree $\Tree(k)$ of the first $k$ generations of $\Tree$,
the number of offspring $\Delta_j$ for $j \in \partial\Tree(k)$
are independent of each other, where
$\partial\Tree(k)$ denotes the set of vertices at generation $k$.
We further assume that the (conditional on $\Tree(k)$)
first moments of $\Delta_j$ are uniformly bounded by a given
non-random finite constant $\Delta$ and say that 
an unlabeled rooted  tree $\Tree$ is 
conditionally independent if $\Tree \simeq \Tree'$ for some 
conditionally independent labeled rooted tree $\Tree'$.
\end{definition}
 
As shown in \cite[Section 4]{ising} (see also 
Theorem \ref{thm:TreeDecay}),
on such a tree, local expectations are
insensitive to boundary conditions that stochastically dominate
the free boundary condition. Our program then follows 
by monotonicity arguments. An example of the monotonicity properties 
enjoyed by the Ising model is provided by Lemma \ref{lem-comp}. 

We next provide a few examples of well known 
random graph ensembles that are uniformly sparse
and converge locally to conditionally independent trees.
To this end, 
let $\node = \{\node_k:\, k\ge 0\}$ be a probability distribution over 
the non-negative integers, with finite,
positive first moment $\andeg$, set 
$\edge_k = (k+1) \node_{k+1}/\andeg$ and denote its mean as $\aedeg$.
We denote by $\Tree(\edge,t)$ the rooted Galton-Watson tree of 
$t\ge 0$ generations, i.e. the random tree such that 
 each node has offspring distribution 
$\{\edge_k\}$, and the offspring numbers at different nodes are
independent. Further, $\Tree(\node,\edge,t)$ 
denotes the modified ensemble where only the offspring 
distribution at the root is changed to $\node$. In particular, 
$\Tree(\node,\edge,\infty)$ is clearly conditionally independent.
Other examples of conditionally independent trees include: 
$(a)$
deterministic trees with bounded degree;
$(b)$ 
percolation clusters on such trees;
$(c)$ 
multi-type branching processes.

When working with random graph ensembles, it is often
convenient to work with the \emph{configuration models} \cite{Bol80}
defined as follows. In the case of the Erd\"os-Renyi random graph,
one draws $m$ i.i.d. edges by choosing their endpoints $i_a,j_a$
independently and uniformly at random for $a=1,\dots,m$.
For a graph with given degree distribution $\{P_k\}$, one first
partitions the vertex sets into subsets  $V_0$, of $\lfloor nP_0\rfloor$
vertices,  $V_1$ of $\lfloor nP_1\rfloor$ vertices, 
$V_2$ of $\lfloor nP_2\rfloor$ vertices, etc.
Then associate $k$ half-edges to the vertices in $V_k$ for each $k$
(eventually adding one half edge to the last node, to make their total 
number even). Finally, recursively match two uniformly random 
half edges until there is no unmatched one. Whenever we need to 
make the distinction we denote by $\prob_*(\,\cdot\,)$ probabilities 
under the corresponding configuration model.

The following simple observation transfers results from configuration
models to the associated uniform models.
\begin{lemma}\label{lemma:Contiguous}
Let $A_n$ be a sequence of events, such that, 
under the configuration model
\begin{eqnarray}
\sum_{n}\prob_*(G_n\not\in A_n)<\infty\, .
\end{eqnarray}
Further, assume $m =\lfloor \alpha n \rfloor$ with $\alpha$ fixed
(for Erd\"os-Renyi random graphs), or $\{P_k\}$ fixed, with bounded first moment
(for general degree distribution).
Then, almost surely under the uniform model, property $A_n$ holds for all 
$n$ large enough.
\end{lemma}
\begin{proof} 
The point is that, the graph chosen under the configuration model 
is distributed uniformly when further conditional on the 
property $L_n$ that it has neither 
self-loops nor double edges (see \cite{Rgraphs}). 
Consequently, 
$$
\prob(G_n\not\in A_n)=
\prob_*(G_n\not\in A_n|L_n)\le \prob_*(G_n\not\in A_n)/\prob_*(L_n) \,.
$$
The thesis follows by recalling that $\prob_*(L_n)$ is bounded 
away from $0$ uniformly in $n$ for the models described here 
(c.f. \cite{Rgraphs}), and applying the Borel-Cantelli lemma.
\end{proof}

Our next lemma ensures that we only need to check 
the local (weak) convergence in expectation with respect 
to the configuration model.
\begin{lemma}\label{lemma:GraphConcentration}
Given a finite  rooted tree $T$ of at most $t$ generations, 
assume that 
\begin{eqnarray}\label{eq:mean-conv-ball}
\lim_{n\to\infty}\prob_*\{\Ball_i(t)\simeq T\} = \covm_T\,,
\end{eqnarray}
for a uniformly random vertex $i \in G_n$.
Then, under both the configuration and the uniform models
of Lemma \ref{lemma:Contiguous},
$\prob_n\{\Ball_i(t)\simeq T\} \to \covm_T$ almost surely.
\end{lemma}
\begin{proof} Per given value of $n$ consider the random variable 
$Z \equiv \prob_n \{\Ball_i(t) \simeq T\}$.
In view of Lemma \ref{lemma:Contiguous} and the assumption 
(\ref{eq:mean-conv-ball}) that $\E_* [Z] = \prob_* \{\Ball_i(t) \simeq T\}$
converges to $\covm_T$, it suffices 
to show that $\prob_*\{|Z - \E_*[Z]| \ge \delta\}$ is summable 
(in $n$), for any fixed $\delta>0$. To this end, let
$r$ denote the maximal degree of $T$. The
presence of an edge $(j,k)$ in the resulting 
multi-graph $G_n$ affects the event $\{\Ball_i(t) \simeq T\}$ 
only if there exists a path 
of length at most $t$ in $G_n$ between $i$ and $\{j,k\}$, 
the maximal degree along which is at most $r$.
Per given choice of $(j,k)$ there are at most 
$u=u(r,t) \equiv 2 \sum_{l=0}^t r^l$ 
such values of $i \in [n]$, hence the Lipschitz norm 
of $Z$ as a function of the location of the 
$m$ edges of $G_n$ is bounded by $2 u/n$. 
Let $G_n(t)$ denote the graph formed by the first $t$
edges (so $G_n(m)=G_n$), and introduce the martingale
$Z(t) = \E_* [Z | G_n(t)]$, so $Z(m)=Z$ and $Z(0)=\E_*[Z]$. 
A standard argument 
(c.f. \cite{AlonSpencer,UrbankeBook}),
shows that the conditional 
laws $\prob_*(\, \cdot\,|G_n(t))$ and $\prob_*(\, \cdot\,|G_n(t+1))$ of
$G_n$ can be coupled in such a way that the resulting 
two (conditional) 
realizations of $G_n$ differ by at most two edges. Consequently, 
applying Azuma-Hoeffding inequality we deduce 
that for 
any $T$, $M$ and $\delta>0$, some $c_0=c_0 (\delta,M,u)$ 
positive and all $m \le n M$,
\begin{equation}\label{eq:conc-bd}
\prob_* ( \big|Z - \E_*[Z] \,\big| \ge \delta) =
\prob_* ( \big|Z_m - Z_0 \, \big| \ge \delta) \le 2 e^{-c_0 n}\,,
\end{equation}
which is more than enough for completing the proof.
\end{proof}

\begin{propo}\label{prop:local-trees1}
Given a distribution $\{\node_l\}_{l \ge 0}$ 
of finite mean, let $\{G_n\}_{n\ge 1}$ be a sequence of graphs whereby 
$G_n$ is distributed according to the ensemble $\graph(\node,n)$
with degree distribution $\node$.
Then the sequence $\{G_n\}$ is almost surely uniformly sparse
and converges locally to $\Tree(\node,\edge,\infty)$. 
\end{propo}
\begin{proof}
Note that for any random graph $G_n$ of degree distribution $\node$,
\begin{equation}\label{eq:enl-def}
E_n(l) \equiv 
\sum_{i\in V_n}|\di|\, \ind(|\di|\ge l) \le 1 + n \sum_{k \ge l} k P_k 
\equiv 1 + n \andeg_l \,.
\end{equation}
Our assumption that $\andeg = \sum_k k P_k$ is finite implies that
$\andeg_l \to 0$ as $l \to \infty$, so 
any such sequence of graphs $\{G_n\}$ is uniformly sparse. 

As the collection of finite rooted trees of {\em finite depth} is countable,
by Lemma \ref{lemma:GraphConcentration} we have the almost sure 
local convergence of $\{G_n\}$ to $\Tree(\node,\edge,\infty)$
once we show that 
$\prob_*(\Ball_i(t)\simeq T)\to \prob(\Tree(\node,\edge,t)\simeq T)$
as $n \to \infty$, where  $i \in G_n$ is a uniformly random vertex 
and $T$ is any fixed finite, rooted tree of at most $t$ generations.

To this end, we opt to describe the distribution of $\Ball_i(t)$ under 
the configuration model as follows. First fix a non-random partition 
of $[n]$ to subsets $V_k$ with $|V_k| = \lfloor n P_k\rfloor$,
and assign $k$ half-edges to each vertex in $V_k$.
Then, draw a uniformly random vertex $i \in [n]$. 
Assume it is in $V_k$, i.e.
has $k$ half-edges. Declare these half-edges `active'. 
Recursively sample $k$ unpaired (possibly active) 
half-edges, and pair the active half-edges to them. 
Repeat this procedure for the vertices thus connected to $i$ and 
proceed in a breadth first fashion 
for $t$ generations (i.e. until all edges of $\Ball_i(t)$ are
determined). Consider now the modified procedure in which, each time an 
half-edge is selected, the corresponding vertex is put in a separate list,
and replaced by a new one with the same number of half-edges, in the graph.
Half-edges in the separate list are active, but they are not among the
candidates in the sampling part. This modification yields 
$\Ball_i(t)$ which is a random tree, specifically, an instance of 
$\Tree(\widetilde{\node}^{(n)},\widetilde{\edge}^{(n)},t)$, where 
$\widetilde{\node}^{(n)}_k = \lfloor n\, P_k\rfloor/ \sum_l\lfloor n\, P_l\rfloor$.
Clearly, $\Tree(\widetilde{\node}^{(n)},\widetilde{\edge}^{(n)},t)$ 
converges in distribution as $n \to \infty$ to $\Tree(\node,\edge,t)$.
The proof is thus complete by providing a coupling in which 
the probability that either $\Ball_i(t) \simeq T$ under the modified
procedure and $\Ball_i(t) \not\simeq T$ under the original procedure
(i.e. the configurational model), or vice versa, is at most $4 |T|^2/n$. 
Indeed, after $\ell$ steps, a new vertex $j$ is sampled by the pairing
with probability $p_j\propto k_j(\ell)$ in
the original procedure and  $p_j'\propto k_j(0)$ in the modified one, where 
$k_j(\ell)$  is the number of free half-edges associated to vertex $j$ at 
step $\ell$. Having to consider at most $|T|$ steps and stopping once 
the original and modified samples differ, we get the stated coupling upon 
noting that $||p-p'||_{\rm TV}\le 2 |T|/n$ (as both samples must then be 
subsets of the given tree $T$).  
\end{proof}

\begin{propo}\label{prop:local-trees2}
Let $\{G_n\}_{n\ge 1}$ be a sequence of Erd\"os-Renyi random graphs,
i.e. of graphs drawn either from the ensemble $\graph(\alpha,n)$
or from the uniform model with $m=m(n)$ edges, where 
$m(n)/n \to \alpha$. Then, the sequence $\{G_n\}$ is
almost surely uniformly sparse and converges locally to 
the Galton-Watson tree $\Tree(\node,\edge,\infty)$
with Poisson$(2\alpha)$ offspring 
distribution $\node$
(in which case $\edge_k=\node_k$). 
\end{propo}
\begin{proof} We denote by 
$\prob^{\<m\>}(\cdot)$ and $\E^{\<m\>}(\cdot)$ the 
probabilities and expectations with respect to a 
random graph $G_n$ chosen uniformly 
from the ensemble of all graphs of $m$ edges, with
$\prob_*^{\<m\>}(\cdot)$ and $\E^{\<m\>}_*(\cdot)$ 
in use for the corresponding configuration model. 

We start by proving the almost sure uniform sparsity for 
graphs $G_n$ from the uniform ensemble of $m=m(n)$ edges provided
$m(n)/n \le M$ for all $n$ and some finite $M$. To this end, 
by Lemma \ref{lemma:Contiguous} it suffices to prove 
this property for the corresponding configuration model. 
Setting 
$Z \equiv n^{-1} E_n(l)$ for $E_n(l)$ of (\ref{eq:enl-def})
and $\node^{\<m\>}$ to be the Binomial$(2m,1/n)$ distribution of 
the degree of each vertex of $G_n$ in this configuration model, 
note that $\E_*^{\<m\>}[Z] = \andeg^{\<m\>}_l \le \andeg_l$ for
$\andeg_l \equiv \sum_{k \ge l} k \node_k$ of  
the Poisson$(4 M)$ degree distribution $\node$,
any $n \ge 2$ and $m \le n M$.
Since $\sum_k k P_k$ is finite, necessarily 
$\andeg_l \to 0$ as $l \to \infty$ and  
the claimed almost sure uniform sparsity follows 
from the summability in $n$, per fixed $l$ and $\delta>0$ of  
$\prob_*^{\<m\>} \{Z - \E_*^{\<m\>} [Z] \ge \delta\}$, 
uniformly in $m \le n M$. Recall that the presence of an edge
$(j,k)$ in the resulting multi-graph $G_n$ changes the  
value of $E_n(l)$ by at most $2 l$, hence the Lipschitz norm 
of $Z$ as a function of the location of the 
$m$ edges of $G_n$ is bounded by $2 l/n$. Thus, 
applying the Azuma-Hoeffding inequality along the lines of 
the proof of Lemma \ref{lemma:GraphConcentration}
we get here a uniform in $m \le n M$ and 
summable in $n$ bound of the form of (\ref{eq:conc-bd}). 

As argued in proving Proposition \ref{prop:local-trees1}, by Lemma
\ref{lemma:GraphConcentration} we further have the claimed 
almost sure local convergence of graphs from the 
uniform ensembles of $m=m(n)$ edges, once we verify that 
(\ref{eq:mean-conv-ball}) holds for $\prob_*^{\<m\>}(\cdot)$ 
and $\covm_T = \prob \{ \Tree(\node,\edge,t) \simeq T \}$ with
the Poisson$(2\alpha)$ offspring distribution $\node$. 
To this end, fix a finite rooted tree $T$ of depth at most $t$ and   
order its vertices from $1$ (for $\root$) to $|T|$ 
in a breadth first fashion following lexicographic 
order among siblings. Let $\Delta_v$ denote
the number of offspring of $v \in T$ with 
$T(t-1)$ the 
sub-tree of vertices within distance $t-1$ 
from the root of $T$ (so $\Delta_v=0$ 
for $v \notin T(t-1)$), and denoting by 
$b \equiv \sum_{v \le T(t-1)} \Delta_v = |T|-1$ 
the number of edges of $T$. Under our 
equivalence relation between trees there are 
$$
\prod_{v=1}^{b} \frac{n-v}{\Delta_v!} 
$$
distinct embeddings of $T$ 
in $[n]$ for which the root of 
$T$ is mapped to $1$. Fixing such an embedding, 
the event $\{ \Ball_1(t) \simeq T \}$ specifies 
the $b$ edges in the restriction of 
$E_n$ to the vertices of $T$ and further
forbids having any edge in $E_n$ between 
$T(t-1)$ and a vertex outside $T$. Thus, 
under the configuration model $\prob_*^{\<m\>}(\cdot)$ 
with $m$ edges chosen with replacement uniformly 
among the $n_2 \equiv \binom{n}{2}$ possible edges, 
the event $\{ \Ball_1(t) \simeq T \}$ occurs 
per such an embedding for precisely 
$(n_2-a-b)^{m-b} m!/(m-b)!$ of the
$n_2^m$ possible edge selections, 
where $a =(n-|T|) |T(t-1)| + \binom{b}{2}$.
With $\prob_*^{\<m\>} (\Ball_i(t) \simeq T)$ independent 
of $i \in [n]$, it follows that 
$$
\prob_*^{\<m\>} (\Ball_i(t) \simeq T) =
\frac{2^b m!}{n^b (m-b)!} \, \Big(1-\frac{a+b}{n_2}\Big)^{m-b} 
\prod_{v=1}^{b} \frac{n-v}{(n-1) \Delta_v!} \,.
$$
Since $b$ is independent of $n$ and $a = n|T(t-1)| + O(1)$,
it is easy to verify that for $n \to \infty$ and $m/n \to \alpha$  
the latter expression converges to 
$$
\covm_T \equiv 
(2 \alpha)^b e^{-2\alpha |T(t-1)|} \prod_{v=1}^{b} \frac{1}{\Delta_v!}
= \prod_{v=1}^{|T(t-1)|} \node_{\Delta_v} 
= \prob \{ \Tree(\node,\edge,t) \simeq T \} 
$$
(where $\node_k = (2\alpha)^k e^{-2\alpha}/k!$, hence 
$\edge_k=\node_k$ for all $k$). 
Further, fixing $\gamma < 1$ and denoting by $I_n$ the interval 
of width $2 n^\gamma$ around $\alpha n$, it is not hard to check that 
$\prob_*^{\<m\>} (\Ball_i(t) \simeq T) \to \covm_T$ 
uniformly over $m \in I_n$. 

Let $\prob^{(n)}(\cdot)$ and $\E^{(n)}(\cdot)$ denote 
the corresponding laws and expectations with respect  
to random graphs $G_n$ from the ensembles $\graph(\alpha,n)$, 
i.e. where each edge is chosen independently with probability 
$q_n = 2\alpha/(n-1)$. The preceding almost sure local convergence 
and uniform sparseness extend to these graphs since  
each law $\prob^{(n)}(\cdot)$ is a mixture
of the laws $\{ \prob^{\<m\>}(\cdot), m=1,2,\ldots\}$ with 
mixture coefficients $\prob^{(n)}(|E_n|=m)$ that are
concentrated on $m \in I_n$. Indeed, 
by the same argument as in the proof of 
Lemma \ref{lemma:Contiguous}, for any sequence of events $A_n$, 
\begin{equation}\label{eq:mix-bc}
\prob^{(n)}(G_n \notin A_n) \le 
\prob^{(n)}(|E_n| \notin I_n) + 
\eta^{-1} \sup_{m \in I_n} \prob_*^{\<m\>} (G_n \notin A_n) \,,
\end{equation}
where 
$$
\eta = \liminf_{n \to \infty} \inf_{m \in I_n} \prob_*^{\<m\>}(L_n) \,,
$$ 
is strictly positive (c.f. \cite{Rgraphs}). Under $\prob^{(n)}(\cdot)$ 
the random variable $|E_n|$ has the Binomial$(n(n-1)/2,q_n)$ 
distribution (of mean $\alpha n$). Hence, upon 
applying Markov's inequality, 
we find that for some finite $c_1=c_1(\alpha)$ and all $n$,
$$
\prob^{(n)}(|E_n| \notin I_n) 
\le n^{-4 \gamma} \, \E^{(n)} [(|E_n|-\alpha n)^4] 
\le c_1 n^{2-4 \gamma} \,,
$$
so taking $\gamma>3/4$ guarantees the summability (in $n$), 
of $\prob^{(n)}(|E_n| \notin I_n)$. For given
$\delta>0$ we already proved the summability in $n$ of 
$\sup_{m \in I_n} \prob_*^{\<m\>} (G_n \notin A_n)$ both for
$A_n = \{ n^{-1} E_n(l) < \andeg_l + \delta \}$ and for
$A_n = \{ |\prob_n(\Ball_i(t) \simeq T) - \covm_T| < 2\delta\}$.
In view of this, considering (\ref{eq:mix-bc}) for the former choice
of $A_n$ yields the almost sure uniform sparsity of 
Erd\"os-Renyi random graphs from $\graph(\alpha,n)$, while
the latter choice of $A_n$ yields the almost 
sure local convergence of these random graphs to 
the Galton-Watson tree $\Tree(\node,\edge,\infty)$
with Poisson$(2\alpha)$ offspring distribution.
\end{proof}

\begin{remark}\label{rem:local-trees1}
As a special case of Proposition \ref{prop:local-trees1}, almost 
every sequence of uniformly random $k$-regular graphs 
of $n$ vertices converges locally to the (non-random) 
rooted $k$-regular infinite tree $T_k(\infty)$. 
\end{remark}

Let $T_k(\ell)$ denote the tree induced by the first 
$\ell$ generations of $T_k(\infty)$,
i.e.  $T_k(0) = \{\root\}$ and for $\ell \ge 1$ the tree
$T_k(\ell)$ has $k$ offspring at $\root$ and $(k-1)$ offspring 
for each vertex at generations $1$ to $\ell-1$.
It is easy to check that for any $k \ge 3$,
the sequence of finite trees $\{T_k(\ell)\}_{\ell\ge 0}$ 
\emph{does not converge locally to $T_k(\infty)$}.
Instead, it converges to the following random $k$-canopy tree 
(c.f. \cite{AizenmanCanopy} for a closely related definition).
\begin{lemma}\label{dfn:canopy} 
For any $k \ge 3$, 
the sequence of finite trees $\{T_k(\ell)\}_{\ell\ge 0}$ 
converges locally to the $k$-\emph{canopy tree}. This 
random infinite tree, denoted $\CTree_k$, is formed by the union
of the infinite ray $\vec{R} \equiv 
\{(r,r+1), r \ge 0 \}$ 
and additional finite trees $\{ T_{k-1}(r), r \ge 0\}$
such that $T_{k-1}(r)$ is rooted at the $r$-th vertex 
along $\vec{R}$. The root of $\CTree_k$ is 
on $\vec{R}$ with 
$\prob(\CTree_k$ rooted at $r)=(k-2)/(k-1)^{r+1}$ for $r\ge 0$.
\end{lemma}
\begin{proof} This local convergence is immediate upon noting that 
there are exactly $n_r = k(k-1)^{r-1}$ vertices at generation 
$r \ge 1$ of $T_k(\ell)$, hence 
$|T_k(\ell)| = [k(k-1)^{\ell}-2]/(k-2)$ 
and $n_{\ell-r}/|T_k(\ell)| \to \prob(
\CTree_k$ rooted at $r)$ as 
$\ell \to \infty$, for each fixed $r \ge 0$ and $k \ge 3$
(and $\Ball_i(\ell)$ matches for each $i$ of generation 
$\ell-r$ in $T_k(\ell)$ the 
ball $\Ball_r(\ell)$ of the $k$-canopy tree). 
\end{proof}

\begin{remark}
Note that the $k$-canopy tree is not conditionally independent.
\end{remark}

\subsection{Ising models on conditionally independent trees}
Following \cite{ising} it is convenient to extend the
model (\ref{eq:IsingModel}) by 
allowing for vertex-dependent magnetic fields $B_i$, i.e. to consider 
\begin{eqnarray}
\mu(\ux) =\frac{1}{Z(\beta,\uB)}\, \exp\Big\{\beta\sum_{(i,j)\in E}x_ix_j
+\sum_{i\in V}B_ix_i\Big\}\,  .
\label{eq:IsingModelGen}
\end{eqnarray}
In this general context, it is possible to
prove correlation decay results for Ising models on 
conditionally independent trees. Beyond their  independent interest,
such results play a crucial role in our analysis of models on sparse graph
sequences.

To state these results
denote 
by $\mu^{\ell,0}$ the Ising model (\ref{eq:IsingModelGen}) on 
$\Tree(\ell)$ with magnetic fields $\{B_i\}$ 
(also called free boundary conditions),  
and by $\mu^{\ell,+}$ the modified Ising 
model corresponding to the limit $B_i \uparrow +\infty$ 
for all $i\in \partial\Tree(\ell)$ (also called plus boundary conditions), 
using $\mu^{\ell}$ for statements that apply to both free 
and plus boundary conditions. 
\begin{thm}\label{thm:TreeDecay}
Suppose $\Tree$ is a conditionally independent infinite tree
of average offspring numbers bounded by $\Delta$,
as in Definition \ref{def-cit}.
Let $\<\,\cdot\,\>^{(r)}_{i}$ denote the expectation with respect to
the Ising distribution 
on the subtree of $i$ and all its descendants in
$\Tree(r)$ and $\<x;y\>\equiv \<xy\>-\<x\>\<y\>$ denotes the 
centered two point correlation function.
There exist $A$ finite and $\lambda$ positive, depending only on 
$0 < B_{\min} \le B_{\max}$, $\beta_{\max}$ and $\Delta$ finite, 
such that if $B_i \le B_{\max}$ for all $i \in \Tree(r-1)$ and 
$B_i\ge B_{\min}$ for all $i \in \Tree(\ell)$, then for any $r \le \ell$ 
and $\beta \le \beta_{\max}$, 
\begin{eqnarray}
\E\Big\{\sum_{i\in \partial\Tree(r)} \,\<x_{\root};x_i\>_{\root}^{(\ell)}
\Big\}\le A \, e^{-\lambda r} \,. 
\label{eq:ExponentialDecay}
\end{eqnarray}
If in addition $B_i \le B_{\max}$ for all $i \in \Tree(\ell-1)$ then 
for some $C=C(\beta_{\max},B_{\max})$ finite 
\begin{eqnarray}\label{eq:bd-influence}
\E\, ||\mu_{\Tree(r)}^{\ell,+}-\mu_{\Tree(r)}^{\ell,0}||_{\sTV}\le
A e^{-\lambda (\ell-r)} \, \E\{C^{|\Tree(r)|}\}\, .
\end{eqnarray}
\end{thm}

The proof of this theorem, given in \cite[Section 4]{ising}, relies on  
monotonicity properties of the Ising measure, and in particular on the 
following classical inequality.
\begin{propo}[Griffiths inequalities]
Given a finite set $V$ and parameters 
$\bJ=(J_R, R \subseteq V)$ with $J_R \geq 0$,
consider the extended ferromagnetic Ising measure 
\begin{eqnarray}
\mu_{\bJ}(\ux) = 
\frac{1}{Z(\bJ)} 
\exp\Big\{\sum_{R \subseteq V} J_R x_R\Big\}\, ,\label{eq:extIsing}
\end{eqnarray}
where $\ux \in \{+1,-1\}^V$ and 
$x_R\equiv \prod_{u \in R} \, 
x_u$. Then, for $\uX$ of law $\mu_{\bJ}$ and any $A,B \subseteq V$, 
\begin{align}
\E_\bJ [ X_A ] &= \frac{1}{Z(\bJ)}\, \sum_{\ux} x_A 
\exp\Big\{\sum_{R \subseteq V} J_R x_R\Big\} \geq 0 \,,
\\
\frac{\partial\phantom{J_B}}{\partial J_B} 
\E_\bJ [X_A] &= {\rm Cov}_\bJ (X_A,X_B) \geq 0 \,.
\end{align}
\end{propo}
\noindent{\bf Proof.}
See \cite[Theorem IV.1.21]{Liggett} (and consult \cite{Ginibre}
for generalizations of this result). 

Note that the measure $\mu(\cdot)$ of (\ref{eq:IsingModelGen}) 
is a special case of $\mu_{\bJ}$ 
(taking $J_{\{i\}}=B_i$, $J_{\{i,j\}} = \beta$ for
all $(i,j) \in E$ and $J_R=0$ for all other subsets of $V$).
Thus, Griffiths inequalities allow us to compare certain marginals
of the latter measure 
for a graph $G$ and non-negative $\beta$, $B_i$ with 
those for other choices of $G$, $\beta$ and $B_i$. To demonstrate
this, we state (and prove) the following well 
known general comparison results. 
\begin{lemma}\label{lem-comp}
Fixing $\beta \ge 0$ and $B_i \geq 0$, for any finite graph $G=(V,E)$  
and $A \subseteq V$ let $\<x_A\>_G=\mu(x_A=1)-\mu(x_A=-1)$ denote the
mean of $x_A$ under the corresponding Ising measure on $G$.
Similarly, for $U \subseteq V$ 
let $\<x_A\>^0_U$ and $\<x_A\>^+_U$ denote 
the magnetization induced by the Ising measure 
subject to free (i.e. $x_u=0$) and plus (i.e. $x_u=+1$) 
boundary conditions, respectively, at all $u \notin U$.
Then, $\<x_A\>^0_U \leq \<x_A\>_G \leq \<x_A\>^+_U$ for any $A \subseteq U$. 
Further, $U \mapsto \<x_A\>^0_U$ is monotone non-decreasing and 
$U \mapsto \<x_A\>^+_U$ is monotone non-increasing, both with respect
to set inclusion (among sets $U$ that contain $A$). 
\end{lemma}
\begin{proof} From Griffiths inequalities we know that 
$\bJ \mapsto \E_\bJ [X_A]$ is monotone non-decreasing (where 
$\bJ \geq \widehat{\bJ}$ if and only if $J_R \geq \widehat{J}_R$
for all $R \subseteq V$).  
Further, $\<x_A\>_G=\E_{\bJ^0} [X_A]$
where $J^0_{\{i\}}=B_i$, $J^0_{\{i,j\}}=\beta$ when $(i,j)\in E$ and all 
other values of $\bJ^0$ are zero. Considering 
$$
J^{\eta,U}_R = J^0_R + \eta \ind(R \subseteq U^c, |R|=1) \,,
$$ 
with $\eta \mapsto \bJ^{\eta,U}$ non-decreasing, 
so is $\eta \mapsto \E_{\bJ^{\eta,U}}[X_A]$.
In addition, $\mu_{\bJ^{\eta,U}}\,(x_u=-1) \leq C\, e^{-2 \eta}$
whenever $u \notin U$.   
Hence, as $\eta \uparrow \infty$ the measure $\mu_{\bJ^{\eta,U}}$ 
converges to $\mu_\bJ$ 
subject to plus boundary conditions $x_u=+1$ for $u \notin U$. 
Consequently,
$$
\<x_A\>_G \leq \E_{\bJ^{\eta,U}}[X_A] \uparrow \<x_A\>^+_U\,.
$$ 
Similarly, let $J^U_R=J^0_R \ind(R \subseteq U)$ 
noting that under $\mu_{\bJ^U}$ the random vector $\ux_U$ is 
distributed according to the Ising measure $\mu$ restricted to 
$G_U$ (alternatively, having free boundary conditions $x_u=0$
for $u \notin U$). With $A \subseteq U$ we thus deduce that 
$$
\<x_A\>^0_U = \E_{\bJ^U} [X_A] \leq \E_{\bJ^0} [X_A] = \<x_A\>_G
\,.
$$ 
Finally,  
the stated monotonicity of $U \mapsto \<x_A\>^0_U$ and 
$U \mapsto \<x_A\>^+_U$ 
are in view of Griffiths inequalities the direct consequence of 
the monotonicity
(with respect to set inclusions) of $U \mapsto \bJ^U$ 
and $U \mapsto \bJ^{\eta,U}$, respectively. 
\end{proof}

In addition to Griffiths inequalities, the proof of 
Theorem \ref{thm:TreeDecay} uses also the GHS inequality 
\cite{GHS} which regards the effect of a magnetic field $\uB$ on 
the local magnetizations at various vertices. It further uses 
an extension of Simon's inequality (about the
centered two point correlation functions 
in ferromagnetic Ising models with zero magnetic field,
see \cite[Theorem 2.1]{Simon}), to arbitrary magnetic 
field, in the case of Ising models on trees. 
Namely, \cite[Lemma 4.3]{ising} 
states that 
if edge $(i,j)$ is on the
unique path from $\root$ to $k\in\Tree(\ell)$, with 
$j$ a descendant of $i \in \partial \Tree(t)$, $t \ge 0$, then 
\begin{eqnarray}\label{ineq-simon}
\<x_{\root};x_k\>^{(\ell)}_{\root}\le 
\cosh^2(2\beta+B_i) 
\; \<x_{\root};x_{i}\>^{(t)}_{\root}
\<x_{j};x_{k}\>^{(\ell)}_{j}\, .
\end{eqnarray}

\subsection{Algorithmic implications: belief propagation}\label{sec:algorithmic}

The `belief propagation' (BP) algorithm consists of 
solving by iterations a collection of  Bethe-Peierls (or cavity) 
mean field equations. 
More precisely, for the Ising model (\ref{eq:IsingModel}) 
we associate to each directed edge in the graph $i\to j$,
with $(i,j)\in G$, a distribution (or `message')
$\nu_{i\to j}(x_i)$ over $x_i\in \{+ 1,-1\}$, 
using then the following update rule 
\begin{eqnarray}
\nu_{i\to j}^{(t+1)}(x_i) =\frac{1}{z^{(t)}_{i\to j}}\, e^{Bx_i}
\prod_{l\in\di\setminus j}\sum_{x_l}e^{\beta x_ix_l}\nu_{l\to i}^{(t)}
(x_l)
\label{eq:BPIteration}
\end{eqnarray}
starting at a \emph{positive} initial condition, namely where
$\nu_{i\to j}^{(0)}(+1)\ge
\nu_{i\to j}^{(0)}(-1)$ at each directed edge.

Applying Theorem \ref{thm:TreeDecay} we establish 
in \cite[Section 5]{ising} the uniform exponential 
convergence of the BP iteration to the same fixed
point of (\ref{eq:BPIteration}), irrespective of its   
positive initial condition. As we further show there,
for tree-like graphs the limit of the BP iteration 
accurately approximates local marginals of the 
Boltzmann measure (\ref{eq:IsingModel}).
\begin{thm}\label{thm:BPConvergence}
Assume $\beta \ge 0$, $B>0$ and $G$ is a graph of finite
maximal degree $\Delta$. Then, there exists $A=A(\beta,B,\Delta)$ 
and $c=c(\beta,B,\Delta)$ finite, $\lambda=\lambda(\beta,B,\Delta)>0$ 
and a fixed point $\{\nu^*_{i\to j}\}$ of the BP iteration 
(\ref{eq:BPIteration}) such that for any positive initial condition 
$\{\nu^{(0)}_{l \to k}\}$ and all $t \ge 0$, 
\begin{eqnarray}
\label{eq:BPIexp-conv}
\sup_{(i,j) \in E} \| \nu_{i\to j}^{(t)} - \nu^*_{i \to j} \|_{\sTV}
\le A \exp(-\lambda t) \,.
\end{eqnarray}
Further, for any $i_o \in V$, if $\Ball_{i_o}(t)$ is a tree then
for $U \equiv \Ball_{i_o}(r)$ 
\begin{eqnarray}
||\mu_U - \nu_U||_{\sTV} \le 
\exp\Big\{c^{r+1}-\lambda(t-r)\Big\}\, ,
\end{eqnarray}
where $\mu_U(\, \cdot\, )$ is the law 
of $\ux_U \equiv \{x_i:\, i\in U\}$ 
under the Ising model (\ref{eq:IsingModel}) and  
$\nu_U$ the probability distribution 
\begin{eqnarray}
\nu_U(\ux_U)= \frac{1}{z_U}\exp\Big\{\beta
\sum_{(i,j)\in E_U}x_ix_j+B\sum_{i\in U\setminus \partial U}x_i\Big\}
\prod_{i\in\partial U}\nu^*_{i\to j(i)}(x_i)\, ,\label{eq:LocalMarg}
\end{eqnarray}
with $E_U$ the edge set of $U$ whose border is
$\partial U$ 
(i.e. the set of its vertices at distance $r$ from $i_o$), 
and $j(i)$ is any fixed neighbor in $U$ of $i$.
\end{thm} 

\subsection{Free entropy density, from trees to graphs}\label{sec:Ising-freeent}

Bethe-Peierls approximation (we refer to 
Section \ref{sec:BetheInformal}
for a general introduction), allows 
us
to predict the asymptotic
free entropy density for sequences of graphs that converge
locally to conditionally independent trees. We
start by explaining this prediction in a general setting,
then state a rigorous result which verifies it for a specific
family of graph sequences. 

To be definite, assume that $B>0$.
Given a graph sequence $\{G_n\}$ that converges to a conditionally
independent 
tree $\Tree$ with bounded average offspring number, 
let $L=\Delta_\root$ be the degree of its 
root. Define the 'cavity fields' $\{h_1,\dots,h_L\}$ by 
letting $h_j=\lim_{t\to\infty} h_j^{(t)}$ with 
$h_j^{(t)} \equiv \atanh[\<x_j\>^{(t)}_{j}]$, where 
$\<\,\cdot\,\>^{(t)}_{j}$ denotes expectation with respect
to the Ising distribution 
on the sub-tree induced by $j \in \partial \root$ and all its descendants in
$\Tree(t)$ (with free boundary conditions). We note in passing
that $t \mapsto h_j^{(t)}$ is stochastically monotone  
(and hence has  a limit in law)  by Lemma \ref{lem-comp}.
Further $\{h_1,\dots,h_{L}\}$ are conditionally independent given $L$.
Finally, define $\theta=\tanh(\beta)$ and 
\begin{equation}\label{eq:rec-mes}
h_{-j} = B + \sum_{k=1, k \ne j}^{L} \atanh[\theta \tanh (h_k)]\,.
\end{equation}

The Bethe-Peierls free energy density is 
given by 
\begin{align}\label{eqn:phi}
\varphi(\beta &,B) \equiv  
\frac{1}{2}\E\{L\} \ggm (\theta) -
\frac{1}{2}\, \E\Big\{ \sum_{j=1}^L
\log[1+\theta \tanh(h_{-j})\tanh(h_j)] \Big\}
\nonumber \\ 
&+ \E\log\big\{ e^B \prod_{j=1}^{L} [1 +  
\theta \tanh(h_j)] + e^{-B} \prod_{j=1}^{L} 
[1 - \theta \tanh(h_j) ] \big\}\, ,
\end{align}
for $\ggm(u)=-\frac{1}{2}\log(1-u^2)$.
We refer to Section \ref{sec:Bethe-examples} 
where this formula is obtained as a special case of the 
general expression for a Bethe-Peierls free energy.
The prediction is extended to $B<0$ by letting 
$\varphi(\beta,B)=\varphi(\beta,-B)$, and to $B=0$ by letting
$\varphi(\beta,0)$ be the limit of $\varphi(\beta,B)$ as $B \to 0$.

As shown in \cite[Lemma 2.2]{ising}, when $\Tree=\Tree(\node,\edge,\infty)$
is a Galton-Watson tree,
the random variables $\{h_j\}$ 
have a more explicit characterization in 
terms of the following fixed point distribution.
\begin{lemma}\label{lemma:Recursive}
In case $\Tree=\Tree(\node,\edge,\infty)$ 
consider the random variables $\{h^{(t)}\}$ where $h^{(0)} \equiv 0$
and for $t\ge 0$,
\begin{equation}\label{eqn:h_recursion}
h^{(t+1)} \ed B + \sum_{i=1}^{K} \atanh[\theta \tanh (h^{(t)}_i)]\,,
\end{equation}
with $h^{(t)}_i$ i.i.d. copies of $h^{(t)}$ that are independent of the 
variable $K$ of distribution $\edge$. 
If $B>0$ and $\aedeg < \infty$ then  
$t \mapsto h^{(t)}$ 
is stochastically monotone (i.e. there exists a coupling under 
which $\prob(h^{(t)} \le h^{(t+1)})=1$ for all $t$), and 
converges in law to the unique fixed point $h^*$ 
of (\ref{eqn:h_recursion}) that is 
supported on $[0,\infty)$. In this case, $h_j$ of (\ref{eqn:phi})
are i.i.d. copies of $h^*$ that are independent of $L$.
\end{lemma}

The main result of \cite{ising} confirms
the statistical physics prediction for the free entropy density. 
\begin{thm}\label{thm:free_energy}
If $\aedeg$ is finite then 
for any $B \in \reals$, $\beta \ge 0$ and
sequence $\{G_n\}_{n\in\naturals}$ 
of uniformly sparse graphs
that converges locally to $\Tree(\node,\edge,\infty)$,
\begin{equation}\label{eqn:free_energy}
\lim_{n\to\infty}\frac{1}{n}\log Z_n(\beta,B) =\varphi(\beta,B)\,.
\end{equation}
\end{thm}

We proceed to sketch the outline of 
the proof of Theorem \ref{thm:free_energy}.
For uniformly sparse graphs that converge locally
to $\Tree(\node,\edge,\infty)$ the model (\ref{eq:IsingModel})  
has a line of first order phase transitions for $B=0$ 
and $\beta>\beta_{\rm c}$ (that is, where the 
continuous function $B \mapsto \varphi(\beta,B)$ 
exhibits a discontinuous derivative). Thus, the
main idea 
is to utilize 
the magnetic field $B$ to explicitly break the $+/-$ symmetry, 
and to carefully exploit the monotonicity properties of the 
ferromagnetic Ising model in order to establish the result 
even at $\beta>\beta_{\rm c}$. 

Indeed, since $\phi_n(\beta,B) \equiv \frac{1}{n}\log Z_n(\beta,B)$
is invariant under $B\to-B$ and is uniformly (in $n$)
Lipschitz continuous in $B$ with Lipschitz constant one,
for proving the theorem it suffices 
to fix $B>0$ and show that $\phi_n(\beta,B)$ 
converges as $n \to \infty$ to 
the predicted expression $\varphi(\beta,B)$ of (\ref{eqn:phi}). 
This is obviously true for $\beta=0$ since 
$\phi_n(0,B) = \log (2\cosh B) = \varphi (0,B)$. 
Next, denoting by $\< \, \cdot \, \>_n$ the expectation 
with respect to the Ising measure on $G_n$ (at
parameters $\beta$ and $B$), it is easy to see that 
\begin{eqnarray}
\partial_\beta \phi_n(\beta,B) = \frac{1}{n}\sum_{(i,j)\in E_n}
\<x_i x_j\>_n = \frac{1}{2} \E_n \Big[ \sum_{j \in \di} \<x_i x_j\>_n \Big] \,.
\label{eq:DerivativeGraph}
\end{eqnarray}
With $|\partial_\beta \phi_n(\beta,B)| \le |E_n|/n$ 
bounded by the assumed uniform sparsity, it is thus 
enough to show that the expression in (\ref{eq:DerivativeGraph})
converges to the partial derivative of 
$\varphi(\beta,B)$ with respect to $\beta$. 
Turning to compute the latter derivative, 
after a bit of real analysis we find that 
the dependence of $\{h_j,h_{-j}\}$ on $\beta$ can be ignored
(c.f. \cite[Corollary 6.3]{ising} for the proof of this fact
in case $\Tree=\Tree(\node,\edge,\infty)$).
That is, hereafter we simply compute the partial derivative 
in $\beta$ of the expression (\ref{eqn:phi}) while 
considering the law of $\{h_j\}$ and $\{h_{-j}\}$ 
to be independent of $\beta$. To this end, setting 
$z_j=\tanh(h_j)$ and $y_j=\tanh(h_{-j})$, 
the relation (\ref{eq:rec-mes}) amounts to 
$$
y_j = \frac{e^B \prod_{k \ne j} (1+\theta z_k) - e^{-B} \prod_{k \ne j} (1-\theta z_k)}
{e^B \prod_{k \ne j} (1+\theta z_k) + e^{-B} \prod_{k \ne j} (1-\theta z_k)} 
$$
for which it follows that  
\begin{eqnarray*}
\frac{\partial\phantom{\theta}}{\partial \theta} 
\Big\{ \sum_{j=1}^l \log (1+\theta z_j y_j) \Big\} &=&
\frac{\partial\phantom{\theta}}{\partial \theta}
 \log\Big\{ e^B \prod_{j=1}^l (1 + \theta z_j) +  
e^{-B} \prod_{j=1}^l (1-\theta z_j) \Big\} \\
&=&
\sum_{j=1}^l \frac{z_j y_j}{1+\theta  z_j y_j} 
\end{eqnarray*}
and hence a direct computation of the derivative in (\ref{eqn:phi})
leads to  
\begin{eqnarray}
\partial_\beta \,\varphi (\beta,B) = 
\frac{1}{2} 
\E\Big[ \sum_{j \in \partial \root} \<x_\root x_j\>_{\Tree} \Big] \, ,
\label{eq:FreeDer}
\end{eqnarray}
where $\< \cdot \>_{\Tree}$ denotes the expectation with 
respect to the Ising model 
\begin{eqnarray}
\mu_{\Tree}(x_\root,x_1,\ldots,x_L) = \frac{1}{z}
\,\exp\Big\{ \beta \sum_{j=1}^L x_\root x_j+ B x_\root + \sum_{j=1}^L h_j x_j\Big\}\, ,\label{eq:muTree}
\end{eqnarray} 
on the `star' $\Tree(1)$ rooted at $\root$ and the 
random cavity fields $h_j$ of (\ref{eqn:phi}). 

In comparison, fixing a positive integer $t$ and  
considering Lemma \ref{lem-comp} for 
$A \equiv \{i,j\}$ and $U \equiv \Ball_i(t)$, 
we find that the correlation 
$\<x_i x_j\>_n$ lies between the correlations 
$\<x_i x_j\>^0_{\Ball_{i}(t)}$ and $\<x_i x_j\>^+_{\Ball_{i}(t)}$
for the Ising model on the subgraph  
$\Ball_{i}(t)$ with free and plus, respectively, 
boundary conditions at $\partial \Ball_{i}(t)$. 
Thus, in view of (\ref{eq:DerivativeGraph}) 
$$ 
\frac{1}{2}\E_n \{ F_0(\Ball_{i}(t)) \} \le
\partial_\beta \, \phi_n(\beta,B) \le 
\frac{1}{2}\E_n \{ F_+(\Ball_{i}(t)) \} \, ,
$$
where $F_{0/+} (\Ball_{i}(t)) \equiv \sum_{j \in \di} 
\<x_i x_j\>^{0/+}_{\Ball_i(t)}$. 

Next, taking $n \to \infty$ we rely on 
the following consequence of the local convergence 
of a uniformly sparse graph sequence $\{G_n\}$ 
(c.f. \cite[Lemma 6.4]{ising} for the derivation of
a similar result).
\begin{lemma}\label{lemma:EdgeTree}
Suppose a uniformly sparse graph sequence $\{G_n\}$ 
converges locally to the random tree $\Tree$.
Fix an integer $t \ge 0$ and a function $F(\cdot)$ on the collection 
of all possible subgraphs that may occur as $\Ball_i(t)$, 
such that $F(\Ball_i(t))/(|\di|+1)$ is uniformly bounded and 
$F(T_1)=F(T_2)$ whenever $T_1 \simeq T_2$. Then,
\begin{eqnarray}\label{eq:edge-conv}
\lim_{n\to\infty} \E_n \{ F(\Ball_{i}(t)) \}  =
\E\{ F(\Tree(t))\}\, .
\end{eqnarray}
\end{lemma}
Indeed, applying this lemma for the functions $F_0(\cdot)$ and
$F_+(\cdot)$ we find that 
\begin{eqnarray*}
\frac{1}{2} \E \{ F_0(\Tree(t)) \} \le
\liminf_{n\to\infty}  \partial_\beta \, \phi_n(\beta,B)  
\le \limsup_{n\to\infty}  \partial_\beta \, \phi_n(\beta,B)  
\le 
\frac{1}{2}
\E \{ F_+ (\Tree(t)) \} \,.
\end{eqnarray*}
To compute $F_{0/+} (\Tree(t))$ we first sum over the 
values of $x_k$ for $k \in \Tree(t) \setminus \Tree(1)$. 
This has the effect of reducing $F_{0/+} (\Tree(t))$ to the 
form of $\sum_{j \in \partial\root} \<x_\root x_j\>_{\Tree}$
and the cavity fields are taken as $h_j^{(t),0/+} 
\equiv \atanh[\<x_j\>^{(t),0/+}_{j}]$. 
Further, from (\ref{eq:bd-influence}) we deduce that 
as $t \to \infty$ 
both sets of cavity fields converge in law to the same limit
$\{h_j\}$. Since $\E[ \<x_\root x_j\>_{\Tree}]$ 
are continuous with respect to such convergence 
in law, we get by (\ref{eq:FreeDer}) that 
$$
\lim_{t \to \infty} \frac{1}{2} \E \{ F_{0/+} (\Tree(t)) \} =
\partial_\beta \,\varphi (\beta,B) \,, 
$$
which completes the proof of the theorem.
%

\subsection{Coexistence at low temperature}

We focus here on the ferromagnetic Ising model on a random $k$-regular
graph with $k \ge 3$ and zero magnetic field. 
In order to simplify derivations, it is convenient to 
use the so-called \emph{configuration model} for random regular graphs
\cite{Bol80}. 
A graph from this ensemble is generated by associating $k$ half
edges to each $i\in [n]$ (with $kn$ even) and pairing
them uniformly at random. In other words, the collection $E_n$ of $m
\equiv kn/2$ edges is obtained by pairing the $kn$ half-edges. 
Notice that the resulting object is in fact a \emph{multi-graph}
i.e. it might include double edges and self-loops. However
the number of such `defects' is $O(1)$ as $n\to\infty$ and hence
the resulting random graph model shares many properties with random 
$k$-regular graphs.
With a  slight abuse of notation, we keep denoting by
$\graph(k,n)$  the  multi-graph ensemble.

For $\beta \ge 0$ we consider the distribution 
\begin{eqnarray}\label{eq:kreg-Ising} 
\mu_{n,\beta,k}
(\ux) = \frac{1}{Z(G_n)} \,\exp\Big\{\beta\sum_{(i,j)\in E_n}x_ix_j\Big\}\,.
\end{eqnarray} 

Recall Remark \ref{rem:local-trees1}
that any sequence of random graphs 
$G_n$ from the ensembles $\graph(k,n)$ is 
almost surely uniformly sparse and converges 
locally to the infinite $k$-regular tree. Thus, considering
the function 
\begin{align*}
\varphi_k (\beta,h) \equiv
& \frac{k}{2} \big\{  
\ggm(\theta) - \log[1+\theta \tanh^2(h)] \big\}
\nonumber \\
&+ \log\big\{ [1 + \theta \tanh(h)]^k +  
[1 - \theta \tanh(h) ]^k \big\} \,,
\end{align*} 
of $h \in \reals$ and $\theta=\tanh(\beta)$, 
we have from Theorem \ref{thm:free_energy} that 
\begin{equation}\label{eqn:kreg-free_energy}
\lim_{n\to\infty}\frac{1}{n}\log Z(G_n) = \varphi_k (\beta,h^*) \,,
\end{equation}
where the cavity field $h^*$ is the largest solution of
\begin{equation}\label{eqn:fx-point}
g(h) \equiv (k-1) \atanh[\theta \tanh (h)] - h =0 \,.
\end{equation} 
Indeed, the expression for $\varphi_k(\beta,h^*)$
is taken from (\ref{eqn:phi}), noting that here 
$L=K+1=k$ is non-random, hence so are $h_{-j}=h_j=h^*$.
It is not hard to check by calculus that the limit as 
$B \downarrow 0$ of the unique positive solution of 
$g(h)=-B$ is strictly positive if and only if 
$\beta > \beta_{\rm c} \equiv \atanh(1/(k-1))$,
in which case $g(-h)=-g(h)$ is zero if and only if
$h \in \{0,\pm h^*\}$ with $g'(0)>0$ and $g'(h^*)<0$
(c.f. \cite{Martinelli}).

We expect coexistence in this model if and only if 
$\beta>\beta_{\rm c}$ (where we have a line of 
first order phase transitions for the
asymptotic free entropy at $B=0$), and 
shall next prove the `if' part.  
\begin{thm}
With probability one, 
the ferromagnetic Ising measures 
$\mu_{n,\beta,k}$ 
on uniformly random $k$-regular 
multi-graphs from the ensemble $\graph(k,n)$
exhibit coexistence 
if $(k-1) \tanh(\beta)>1$. 
\end{thm}
\noindent{\bf Proof.} 
As in the proof of Theorem \ref{thm:cw-coex} 
(for the Curie-Weiss model), we again 
consider the partition of $\cX^n$ to  
$\Omega_{+} \equiv\{\ux:\, \sum_i x_i \ge 0\}$ and
$\Omega_{-} \equiv\{\ux:\, \sum_i x_i < 0 \}$. From 
the invariance of $\mu_{n,\beta,k}$ with respect to
the sign change $\ux \mapsto -\ux$ it follows that 
$\mu_{n,\beta,k}(\Omega_+) = \mu_{n,\beta,k}(\Omega_0)
+ \mu_{n,\beta,k}(\Omega_-)$ where
$\Omega_r \equiv \{ \ux: \sum_i x_i =r \}$. 
Hence, to prove coexistence it suffices to show that for
$\epsilon>0$ small enough, with probability one
\begin{equation}\label{eq:basic-bd}
\limsup_{n \to \infty} \frac{1}{n} \log \Big\{
\sum_{|r| \le n \epsilon} \mu_{n,\beta,k}(\Omega_r) \Big\} < 0 \,.
\end{equation} 
To this end, note that 
$\mu_{n,\beta,k}(\Omega_r) = Z_r (G_n) / Z(G_n)$ for  
the \emph{restricted partition function} 
\begin{eqnarray}\label{eq:rest-def} 
Z_r (G_n) \equiv
\sum_{\ux \in \Omega_r} \exp\Big\{\beta\sum_{(i,j)\in E_n }x_ix_j\Big\}\,.
\end{eqnarray} 
Further, 
recall that by Markov's inequality and the Borel-Cantelli lemma, for any
positive random variables $Y_n$, with probability one
$$
\limsup_{n \to \infty} \frac{1}{n} \log Y_n \leq 
\limsup_{n \to \infty} \frac{1}{n} \log \E (Y_n) \,.
$$ 
Thus, combining (\ref{eqn:kreg-free_energy}) with the latter inequality for 
$Y_n = \sum_{|r| \le n \epsilon} Z_r (G_n)$ we arrive 
at the inequality (\ref{eq:basic-bd}) upon proving the following lemma 
(c.f. \cite[Section 5]{GerschenMon1}).
\begin{lemma} 
Considering even values of $n$ and assuming 
$\beta > \beta_{\rm c}$, we have that 
\begin{eqnarray}
\lim_{\epsilon \to 0} \limsup_{n \to \infty} 
\frac{1}{n} \log \Big\{ \sum_{|r| \le n \epsilon}
\E Z_r (G_n) \Big\} = \varphi_k(\beta,0) < \varphi_k(\beta,h^*) \,.
\end{eqnarray}
\end{lemma} 
\begin{proof} 
First, following the calculus preceding (\ref{eq:FreeDer}) we get 
after some algebraic manipulations  that 
$$
\partial_h \varphi_k(\beta,h) = \frac{k \theta}{\cosh^2(h)} 
[ f(\tanh(u),c) - f(\tanh(h),c) ]
$$ 
for $c = c(h) \equiv \theta \tanh(h)$ and $u=u(h) \equiv
(k-1) \atanh(c)$, 
where for $c \ge 0$ the function 
$f(x,c)=x/(1+cx)$ is monotone increasing in $x \ge 0$. 
With $\beta > \beta_{\rm c}$ we know already that 
$g(h)>0$
(for $g(\cdot)$ of (\ref{eqn:fx-point})), 
hence $u(h)>h$ for any $h \in (0,h^*)$. From 
the preceding expression for $\partial_h \varphi_k(\beta,h)$ 
and the monotonicity of $f(\cdot,c)$ we thus deduce 
that $\varphi_k(\beta,h^*) > \varphi_k(\beta,0)$.

Next, since $Z_r(G)=Z_{-r}(G)$ we  shall consider hereafter only
$r \ge 0$, setting $s \equiv (n-r)/2$. Further,    
let $\Delta_G(\ux)$ denote the number of edges
$(i,j)\in E$ such that $x_i\neq x_j$ and $Z_r (G,\Delta)$ be the number of
configurations $\ux \in \Omega_r$ such that $\Delta_G(\ux) = \Delta$.
Since $|E|=m$ it follows that
$\sum_{(i,j)\in E} x_ix_j = m-2\Delta_G(\ux)$ and hence 
\begin{eqnarray*} 
Z_r (G) = e^{\beta m}\sum_{\Delta=0}^{m} Z_r (G,\Delta)\,
e^{-2\beta\Delta}\, .  
\end{eqnarray*} 
By the linearity of the expectation
and since the distribution of $G_n$ 
(chosen uniformly from $\graph(k,n)$)
is invariant under any permutation of the vertices, we have that 
\begin{align*} \E\{\,Z_r (G_n,\Delta)\}
&= \sum_{\ux \in \Omega_r} \prob\{\Delta_{G_n}(\ux)=\Delta\} =
\binom{n}{s}\prob\{\Delta_{G_n}(\ux^*)=\Delta\} \\ &= \binom{n}{s}
\, \frac{|\{G\in\graph(k,n)\mbox{ and } \Delta_G(\ux^*)=\Delta\}| }{
|\graph(k,n)|}\, , 
\end{align*} 
where $x^*_i=-1$ for $i \le s$ and
$x^*_i=1$ for $s < i \le n$.

The size of the ensemble $\graph(k,n)$ is precisely
the number of pairings of $nk$ objects, i.e.
\begin{eqnarray*}
|\graph(k,n)| = \Pair(nk) \equiv \frac{(nk)!}{(nk/2)!2^{nk/2}}\, .
\end{eqnarray*}
Similarly, the number of such pairings with exactly $\Delta$
edges of unequal end-points is 
\begin{eqnarray*}
|\{ G\in\graph(k,n)\mbox{ and } 
\Delta_G(\ux^*)=\Delta\}| =  
\binom{ks}{\Delta}
\binom{\hat{n}}{\Delta}
\; \Delta!\; 
\Pair(ks-\Delta)
\Pair(\hat{n}-\Delta)
\,,
\end{eqnarray*}
where $\hat{n} \equiv k(n-s)$.
Putting everything together we get that 
\begin{align}\label{eq:Zgn-star}
\E\{ Z_r & (G_n)\}=  \nonumber \\
& \frac{e^{\beta m}}{\Pair(2m)} \binom{n}{s}
 \sum_{\Delta=0}^{ks}\binom{ks}{\Delta}
\binom{\hat{n}}{\Delta}
\, \Delta!\, 
\Pair(ks-\Delta)
\Pair(\hat{n}-\Delta)
e^{-2\beta\Delta}\,.
\end{align}
Recall that for any  $q\in [0,1]$
\begin{eqnarray*}
n^{-1} \log \binom{n}{nq} = H(q) + o(1)\, ,\;\;\;\;\;\;\;\;\;
n^{-1} \log \Pair(n) = \frac{1}{2} \log \Big(\frac{n}{e}\Big) + o(1) \, ,
\end{eqnarray*}
where $H(x) \equiv -x\log x-(1-x)\log(1-x)$ denotes
the binary entropy function.

Setting $\Delta= \delta k n$, $s = u n$ and
$$
\psi_\beta(u,\delta) \equiv 
(u-\delta) \log (u-\delta) +(1-u-\delta) \log (1-u-\delta) + 
2 \delta \log \delta + 4\beta \delta \,, 
$$
we find upon substituting these estimates in the expression 
(\ref{eq:Zgn-star}) that  
\begin{eqnarray*}
n^{-1} \log \E\{Z_r (G_n)\} = \frac{\beta k}{2} + (1-k)H(u) 
- \frac{k}{2} \inf_{\delta \in [0,u]} \psi_\beta(u,\delta) +o(1)\,.
\end{eqnarray*}
Differentiating $\psi_\beta(u,\delta)$ in $\delta$ 
we deduce that the infimum 
in the preceding expression is achieved for the positive solution 
$\delta = \delta_* (\beta,u)$ of $(u-\delta)(1-u-\delta)=\delta^2 e^{4 \beta}$.
Using this value of $\delta$ we get that
$n^{-1} \log \E\{Z_r (G_n)\} = \eta_k(\beta,u) + o(1)$, where  
$$
\eta_k (\beta,u) \equiv 
\frac{\beta k}{2} + (1-k)H(u) 
- \frac{k}{2} 
\big\{ 
u \log (u-\delta_*(\beta,u)) +(1-u) \log (1-u-\delta_*(\beta,u))
\big\}  \,.
$$
Next, note that $\delta_*(\beta,1/2)=1/[2(1+e^{2\beta})]$
from which we obtain after some 
elementary algebraic manipulations 
that $\eta_k(\beta,1/2)=\varphi_k(\beta,0)$. Further,
as $\eta_k(\beta,u)$ is continuous in $u \ge 0$,
we conclude that
$$
\limsup_{n \to \infty}
n^{-1} \log \big\{ \sum_{|r| \le \epsilon n} 
\E Z_r(G_n) \big\} = \sup_{|2u-1| \le \epsilon} \eta_k(\beta,u)
\,,
$$
which for $\epsilon \to 0$ converges 
to $\eta_k(\beta,1/2)=\varphi_k(\beta,0)$, as claimed.
\end{proof}
%

%
%
\section{The Bethe-Peierls approximation}
\label{ch:Bethe}
\setcounter{equation}{0}

Bethe-Peierls approximation reduces the problem 
of computing partition functions and expectation values to the
one of solving a set of non-linear equations. While in general this
`reduction' involves an uncontrolled error, for 
mean-field models it
is expected to be asymptotically exact in the large 
system limit. In fact, in Section \ref{ch:IsingChapter} 
we saw such a result for 
the ferromagnetic Ising model on sparse tree-like graphs.

\emph{Bethe states}, namely those distributions 
that are well approximated within the Bethe-Peierls scheme
play for mean-field models the role that  
pure Gibbs states do on infinite lattices (for the latter
see \cite{Georgii}). For example, it is 
conjectured by physicists that a large class of models, 
including for instance the examples in Section 
\ref{ch:Intro}, decompose into convex combinations of Bethe states.

In the context of mean field spin glasses, the Bethe-Peierls 
method was significantly extended by M\'ezard, Parisi and Virasoro
to deal with proliferation of pure states \cite{SpinGlass}.
In the spin glass jargon, this phenomenon is referred to 
as `replica symmetry breaking,' and the whole approach is 
known as the `cavity method'. 
A closely related approach
is provided by the so-called TAP (Thouless-Anderson-Palmer) equations
\cite{SpinGlass}.

Section \ref{sec:BetheInformal} outlines the rationale behind 
the Bethe-Peierls approximation of local marginals, based on
the Bethe mean field equations (and the belief propagation
algorithm for iteratively solving them). Complementing it,
Section \ref{sec:BetheEntropy} introduces the
Bethe free entropy. In Section \ref{sec:Bethe-examples} we
explain how these ideas apply to 
the ferromagnetic Ising, the Curie-Weiss model, the 
Sherrington-Kirkpatrick model and the independent set model. 
Finally, in Section \ref{sec:BetheFormal} we define a 
notion of correlation decay which 
generalizes the so called `extremality condition'
in trees. We show that if the graphical model associated with 
a permissive graph-specification pair $(G,\upsi)$ 
satisfies such correlation decay condition then it 
is a Bethe state. Subject to  
a slightly stronger condition, \cite{bethe} validates also the
Bethe-Peierls approximation for its free entropy.

While in general extremality on the graph $G$ does not coincide with 
extremality on the associated tree model, in Section \ref{ch:Reco}
we shall provide a sufficient condition for this to happen
for models on random graphs.

%
\subsection{Messages, belief propagation and Bethe equations}
\label{sec:BetheInformal}

Given a variable domain $\cX$ and a simple finite graph 
$G\equiv (V, E)$ without double edges or self loops, let
$\vE\equiv \{i\to j: \; (i,j)\in E\}$
denote the induced set of directed edges. 
The Bethe-Peierls method provides an approximation for the marginal on 
$U \subset V$ of the probability measure $\mu \equiv \mu_{G,\upsi}$ 
cf. Eq.~(\ref{eq:Canonical}). The basic idea is to describe  
the influence of the factors outside $U$ via 
factorized boundary conditions. Such a boundary law
is fully specified by a collection of distributions
on $\cX$ indexed by the directed edges 
on the `internal' boundary 
$\dU=\{i \in U: \di \not\subseteq U\}$ of $U$ (where
as usual $\di$ is the set of neighbors of $i \in V$).
More precisely, this is described by appropriately
choosing a \emph{set of messages}.
\begin{definition}\label{def:smessage}
A \emph{set of messages} is a collection 
$\{\nu_{i\to j}(\, \cdot\, ):\; i\to j\in\vE\}$
of probability distributions over $\cX$
indexed by the directed edges in $G$.

A set of messages is \emph{permissive}
for a permissive graph-specification pair $(G,\upsi)$
if $\nu_{i \to j}(x_i^{\rm p})$ are positive and further
$\nu_{i \to j}(\, \cdot \,) = \psi_i(\, \cdot\, )/z_i$ 
whenever $\di=\{j\}$.
\end{definition}

As we shall soon see, in this context 
the natural candidate for the Bethe-Peierls approximation is 
the following \emph{standard message set}. 
\begin{definition}
The \emph{standard message set} for the canonical probability 
measure $\mu$
associated to a permissive graph-specification pair $(G,\upsi)$ 
is $\nu_{i\to j}^* (x_i) \equiv \mu^{(ij)}_i(x_i)$, 
that is, the marginal on $i$ of the probability measure
on $\cX^{V}$ 
\begin{eqnarray}
\mu^{(ij)}(\ux) = \frac{1}{Z_{ij}}\, 
\prod_{(k,l)\in E\setminus (i,j)}\,
\psi_{kl}(x_k,x_l) \prod_{k\in V} \psi_{k}(x_k)
\, ,
\label{eq:CavityMu}
\label{eq:EdgeMEasure}
\end{eqnarray}
obtained from equation (\ref{eq:Canonical}) 
upon `taking out' the contribution 
$\psi_{ij}(\cdot,\cdot)$ of edge $(i,j)$ 
(and with $Z_{ij}$ an appropriate normalization constant).
\end{definition}
\begin{remark}\label{rem:standard-permissive}
Since $\upsi$ is permissive, the measure
$\mu^{(ij)}(\cdot)$ is well defined and strictly positive
at $\ux=(x_1^{\rm p},\ldots,x_n^{\rm p})$. Further, 
the marginal on $i$ of $\mu^{(ij)}(\cdot)$ 
is precisely $\psi_i(x_i)/\sum_x \psi_i(x)$ whenever $\di = \{j\}$, so 
the collection $\{ \nu_{i\to j}^* (\cdot) \}$ 
is indeed a permissive set of messages (per Definition \ref{def:smessage}).  
\end{remark}

In order to justify the Bethe-Peierls method
let $\mu^{(i)}(\,\cdot\,)$ denote the probability measure
obtained from the canonical measure of (\ref{eq:Canonical}) 
when the vertex $i \in V$ and all edges incident on $i$ 
are removed from $G$. 
That is, 
\begin{eqnarray}
\mu^{(i)}(\ux) \equiv \frac{1}{Z_i}\, \prod_{(k,l)\in E,\; i\not\in(k,l)}\,
\psi_{kl}(x_k,x_l)
\prod_{k \in V, k \ne i} \psi_k(x_k)
\, .\label{eq:VertexMeasure}
\end{eqnarray}
For any $U\subseteq V$ we let $\mu_U$ 
(respectively, $\mu^{(ij)}_U$, $\mu^{(i)}_U$), denote the
marginal distribution of $\ux_U\equiv\{x_i:\, i\in U\}$ when $\ux$
is distributed according to $\mu$ (respectively $\mu^{(ij)}$, $\mu^{(i)}$).

Clearly, finding 
good approximations to the marginals
of the modified models $\mu^{(ij)}$, $\mu^{(i)}$ is essentially equivalent
to finding good approximations for the original model $\mu$.
Our first step consists of deriving an identity between certain marginals of 
$\mu^{(ij)}$ in terms of marginals of $\mu^{(i)}$.
Hereafter, we write 
$f(\cdot)\normeq g(\cdot)$ whenever two non-negative functions
$f$ and $g$ on the same domain
differ only by a positive normalization constant.  
By definition we then have that 
\begin{eqnarray}
\mu^{(ij)}_{ij}(x_i,x_j) \normeq
\psi_i(x_i) \sum_{\ux_{\di\setminus j}}\mu^{(i)}_{\di}(\ux_{\di})\,
\prod_{l\in\di\setminus j}\psi_{il}(x_i,x_l)\, .\label{eq:Exact}
\end{eqnarray}  
To proceed, we let $\me_{i\to j}^* (x_i) \equiv \mu^{(ij)}_i (x_i)$
and make the crucial approximate independence assumptions
\begin{eqnarray}
\mu_{ij}^{(ij)}(x_i,x_j) & =&  \me_{i\to j}^*(x_i)\me^*_{j\to i}(x_j)+ 
{\sf ERR}\, ,
\label{eq:BasicAssumption1}
\\
\mu^{(i)}_{\di} 
(\ux_{\di}) &=&\prod_{l\in\di} \me^*_{l\to i}(x_l) + {\sf ERR}\, ,
\label{eq:BasicAssumption2}
\end{eqnarray}
where the error terms {\sf ERR} are assumed to be small. Indeed, 
upon neglecting the error terms, plugging these expressions in 
equation (\ref{eq:Exact}), setting $x_j=x_j^{\rm p}$
and dividing by the positive 
common factor $\me^*_{j\to i}(x_j)$, we get the following
\emph{Bethe equations}.
\begin{definition}
\label{def:BetheEq}
Let $\M(\cX)$ denote the space of probability measures over $\cX$
and consider the Bethe (or \emph{belief propagation, BP}) mapping 
$\Tr$ of the space $\M(\cX)^{\vE}$ of possible message sets to itself,
whose value at $\nu$ is
\begin{eqnarray}\label{def:BetheMapping}
(\Tr\nu)_{i\to j}(x_i)\equiv \frac{\psi_i(x_i)}{z_{i\to j}}
\prod_{l\in\di\backslash{j}}\Big[ \sum_{x_l \in \cX}
\psi_{il}(x_i,x_l)\nu_{l\to i}(x_l) \Big] \, ,
\end{eqnarray}
where $z_{i\to j}$ is determined by the normalization condition 
$\sum_{x\in\cX} (\Tr\nu)_{i\to j}(x) = 1$. The 
\emph{Bethe equations} 
characterize fixed points of the BP mapping. That is, 
\begin{eqnarray}
\me_{i\to j}(x_i) \equiv 
(\Tr\nu)_{i\to j}(x_i) \,.
\label{eq:BPeq}
\end{eqnarray}
\end{definition}

\begin{remark}\label{rem:positivity}
The BP mapping $\Tr$ is well defined 
when the specification $\upsi$ is permissive. Indeed, 
in such a case there exists for each $i \to j \in \vE$ and
any message set $\nu$, a positive constant $z_{i\to j} 
\ge \psi_{\min}^{|\di|}$ 
for which $(\Tr\nu)_{i\to j}\in\M(\cX)$. 

Moreover, in this case 
by definition $(\Tr \nu)_{i \to j}(x)$
is positive at $x=x_i^{\rm p}$ 
and further, equals $\psi_i(x)/z_i$ 
whenever $\di = \{j\}$. 
In particular, any solution of the Bethe equations
is a permissive set of messages.
\end{remark}
%
%

These equations characterize the set of messages $\{ \me^*_{i\to j}(\cdot)\}$ 
to be used in the approximation. 
Bethe-Peierls method estimates marginals of the 
graphical model $\mu_{G,\upsi}$ in a manner similar to that 
expressed by 
(\ref{eq:BasicAssumption1}) and (\ref{eq:BasicAssumption2}).
For instance, $\mu_i(\cdot)$ is then approximated by  
\begin{eqnarray}
\mu_i(x_i) \normeq 
\psi_i(x_i) 
\prod_{j\in\di}\sum_{x_j}\psi_{ij}(x_i,x_j)\me^*_{j\to i} (x_j) \, .
\label{eq:LocalMarg1}
\end{eqnarray}
A more general expression will be provided in 
Section \ref{sec:BetheFormal}.

At this point the reader can verify 
that if $G$ is a (finite) tree then the error terms in
equations (\ref{eq:BasicAssumption1}) and
(\ref{eq:BasicAssumption2}) vanish, hence in this case
the Bethe equations have a unique solution, which is 
precisely the 
standard message set for the canonical measure $\mu$.
More generally, 
it is not hard to verify that in the framework of 
a (permissive) specification $\upsi$
for a factor graph $G=(V,F,E)$
the Bethe equations are then  
\begin{eqnarray*}
\nu_{a \to i} (x_i)&\normeq& \sum_{\ux_{\da \setminus i}} 
\psi_a (\ux_{\da}) \prod_{l \in \da \setminus i} \nu_{l \to a} (x_l) \,,
\\
\nu_{i \to a} (x_i) &\normeq& \prod_{b \in \di \setminus a} \nu_{b \to i} (x_i)
\end{eqnarray*}
and that when the factor graph is a (finite) 
tree these equations have a unique solution 
which is precisely 
the standard message set for the (canonical) measure 
$\mu_{G,\upsi}(\cdot)$ of (\ref{eq:FCanonical}). That is,
$\nu_{i \to a}(\cdot)$ and $\nu_{a \to i}(\cdot)$ are
then the marginals on variable $i$ for factor graphs in 
which factor $a$ and all factors in $\di \setminus a$ are
removed, respectively.

In view of the preceding, we expect such an approximation to be 
tight as soon as $G$ lacks short cycles or for a sequence 
of graphs that converges locally to a tree. 

%
%
\subsection{The Bethe free entropy}
\label{sec:BetheEntropy} 

Within the Bethe approximation all marginals are expressed in 
terms of the permissive messages $\{\nu_{i\to j}\}$ that 
solve the Bethe equations (\ref{eq:BPeq}). Not surprisingly, 
the free entropy $\log Z(G,\upsi)$ can also be 
approximated in terms as the \emph{Bethe free entropy} at this message set.
\begin{definition}\label{def:BetheFree}
The real valued function on the space of permissive message sets  
\begin{eqnarray}
\Phi_{G,\upsi} (\nu) &=& -\sum_{(i,j)\in E} \log \Big\{
\sum_{x_i,x_j}\psi_{ij}(x_i,x_j)\nu_{i\to j}(x_i)\nu_{j\to i}(x_j)\Big\}
\nonumber \\
&& + \sum_{i\in V}\log\Big\{\sum_{x_i}\psi_i(x_i) \prod_{j\in\di}\sum_{x_j}
\psi_{ij}(x_i,x_j)\nu_{j\to i}(x_j)\Big\} \,,
\label{eq:BetheFree}
\end{eqnarray}
is called the \emph{Bethe free entropy} associated with   
the given permissive graph-specification pair $(G,\upsi)$. 
In the following we shall often drop the subscripts and write
$\Phi(\nu)$ for the Bethe free entropy.
\end{definition}

In the spirit of the observations made at the end of Section 
\ref{sec:BetheInformal},
this approximation is exact whenever $G$ is a tree and 
the Bethe messages are used. 
\begin{propo}\label{prop:Bethe-free-tree} 
Suppose $G$ is a tree and let $\me^*$ denote the unique solution
of the Bethe equations (\ref{eq:BPeq}). 
Then, $\log Z(G,\upsi) = \Phi_{G,\upsi} (\nu^*)$.
\end{propo}
\begin{proof}
We progressively disconnect the tree $G$ in a recursive fashion. 
In doing so, note that if $f(x)=f_1(x)f_2(x)/f_3(x)$ and 
$f_a(x) \normeq p(x)$ 
for $a\in\{1,2,3\}$ and some probability distribution $p$, then
\begin{eqnarray}
\log\big\{\sum_x f(x) \big\} = 
\log\big\{\sum_xf_1(x)\big\}+
\log\big\{\sum_xf_2(x)\big\}-
\log\big\{\sum_xf_3(x)\big\}
\label{eq:RemarkBethe}
\end{eqnarray}
(adopting hereafter the convention that $0/0=0$).

Proceeding to describe the first step of the recursion,  
fix an edge $(i,j)\in E$. Without this edge the tree
$G$ breaks into disjoint subtrees 
$G^{(i)}$ and $G^{(j)}$ such that $i \in G^{(i)}$
and $j \in G^{(j)}$. Consequently, the measure 
$\mu^{(ij)}(\cdot)$ of (\ref{eq:CavityMu}) 
is then the product of two canonical measures, 
corresponding to the restriction of the specification
$\upsi$ to $G^{(i)}$ and to $G^{(j)}$, respectively. 
Let $Z_{i\to j}(x)$ denote the constrained 
partition function for the specification $\upsi$ 
restricted to the subtree $G^{(i)}$  
whereby we force the variable $x_i$ to take the value $x$.
With $Z_{j\to i}(x)$ defined similarly for
the subtree $G^{(j)}$, we obviously have that
\begin{eqnarray*}
Z(G,\upsi) 
= \sum_{x_i,x_j} Z_{i\to j}(x_i)\psi_{ij}(x_i,x_j) Z_{j\to i}(x_j)\,.
\end{eqnarray*}
Further, recall our earlier observation that for a tree $G$ 
the unique solution $\{\me^*_{i \to j}(\cdot)\}$ 
of (\ref{eq:BPeq}) is $\{\mu^{(ij)}_i(\cdot)\}$. Hence,
in this case $Z_{i\to j}(x_i) \normeq \me^*_{i\to j}(x_i)$,
$Z_{j\to i}(x_j) \normeq \me^*_{j\to i}(x_j)$ and 
$\me^*_{i\to j}(x_i)\psi_{ij}(x_i,x_j)\me^*_{j\to i}(x_j)
\normeq\mu_{ij}(x_i,x_j)$. Setting
$\psi^*_{i \to j}(x_i,x_j)
\equiv \me^*_{i \to j}(x_i) \psi_{ij}(x_i,x_j)$ we 
next apply 
the identity (\ref{eq:RemarkBethe}) for $x=(x_i,x_j)$, 
$f_1(x)=Z_{i\to j}(x_i)\psi^*_{j \to i}(x_i,x_j)$,
$f_2(x)=\psi^*_{i \to j}(x_i,x_j)Z_{j\to i}(x_j)$ 
and 
$f_3(x)=
\me^*_{i\to j}(x_i)\psi_{ij}(x_i,x_j)\me^*_{j\to i}(x_j)$
to get that 
\begin{eqnarray*}
\log Z(G,\upsi) = \log Z(G^{(i \to j)},\upsi^{(i \to j)}) +
\log Z(G^{(j \to i)},\upsi^{(j \to i)}) - \log \varphi(i,j) \,,
\end{eqnarray*}
where for each edge $(i,j)\in E$, 
\begin{eqnarray*}
\varphi(i,j) \equiv
\sum_{x_i,x_j}\me^*_{i\to j}(x_i)\psi_{ij}(x_i,x_j)
\me^*_{j\to i}(x_j)
\end{eqnarray*}
and the term 
\begin{eqnarray*}
Z(G^{(i \to j)},\upsi^{(i \to j)}) \equiv
\sum_{x_i,x_j}
Z_{i\to j}(x_i)\psi^*_{j \to i}(x_i,x_j) \,,
\end{eqnarray*}
is the partition function for the (reduced size) 
subtree $G^{(i \to j)}$ obtained when adding 
to $(G^{(i)},\upsi^{(i)})$ 
the edge $(i,j)$ and 
the vertex $j$ whose specification  
is now $\psi^*_{j} \equiv \me^*_{j \to i}$. We 
have the analogous representation for 
\begin{eqnarray*}
Z(G^{(j \to i)},\upsi^{(j \to i)}) \equiv
\sum_{x_i,x_j} \psi^*_{i\to j}(x_i,x_j) 
Z_{j\to i}(x_i) \,.
\end{eqnarray*}
It is not hard to verify that the unique 
solution of the Bethe equations (\ref{eq:BPeq}) 
for the graph-specification
$(G^{(i \to j)},\upsi^{(i \to j)})$ coincides 
with $\me^*(\cdot)$ at all directed edges
of $G^{(i \to j)}$.
Likewise, the unique
solution of the Bethe equations (\ref{eq:BPeq})
for the graph-specification
$(G^{(j \to i)},\upsi^{(j \to i)})$ 
coincides with $\me^*(\cdot)$  
at all directed edges 
of $G^{(j \to i)}$. Thus,  
recursively repeating this operation until 
we have dealt once with each edge of $G$, 
we find a contribution 
$-\log \varphi(k,l)$ from each $(k,l) \in E$,
the sum of which is precisely the first 
term in (\ref{eq:BetheFree}), evaluated 
at the permissive set of messages
$\me^*_{i \to j}(\cdot)$.
The residual graph remaining at this stage consists of 
disconnected `stars' centered at vertices of $G$,
with specification 
$\me^*_{l \to k} (x_l)$ at vertices 
$l \in \partial k$ for 
the `star' centered at $k \in V$ (and original specification 
at vertex $k$ and the edges $(k,l) \in E$). The 
log-partition function for such star is 
$\log \big\{ \sum_{x_k} \psi_k(x_k) \prod_{l\in\partial k}\sum_{x_l}
\psi^*_{l \to k}(x_l,x_k) \big\}$ so
the aggregate of these contributions over 
all vertices of $G$ is precisely the 
second term in (\ref{eq:BetheFree}), evaluated at
$\me^*_{i \to j}(\cdot)$.
\end{proof}

\begin{lemma}\label{lem:Bethe-stat}
Solutions of the Bethe equations (\ref{eq:BPeq}) 
for a given permissive graph-specification pair $(G,\upsi)$
are stationary points of the corresponding 
Bethe free entropy $\Phi_{G,\upsi} (\cdot)$. The converse holds when
the $|\cX|$-dimensional matrices $\{\psi_{ij}(x,y)\}$
are invertible for all $(i,j)\in E$.
\end{lemma}
\begin{proof} From the formula (\ref{eq:BetheFree}) and 
our definition (\ref{def:BetheMapping}) we find   
that for any $j \to i \in \vE$ and $x_j \in \cX$,
\begin{eqnarray*}
\frac{\partial\Phi(\nu)}{\partial \nu_{j\to i}(x_j)}=
&-&\frac{\sum_{x_i}\nu_{i\to j}(x_i)\psi_{ij}(x_i,x_j)}
{\sum_{x'_i,x'_j}\nu_{i\to j}(x'_i)\nu_{j\to i}(x'_j)\psi_{ij}(x'_i,x'_j)}\\
&+&
\frac{\sum_{x_i}(\Tr\nu)_{i\to j}(x_i)\psi_{ij}(x_i,x_j)}
{\sum_{x'_i,x'_j}(\Tr\nu)_{i\to j}(x'_i)\nu_{j\to i}(x'_j)\psi_{ij}(x'_i,x'_j)}
\;.
\end{eqnarray*}
Hence, if $\{\nu_{i \to j}(\cdot)\}$ satisfies  
the Bethe equations (\ref{eq:BPeq}), then
$\partial\Phi(\nu)/\partial \nu_{j\to i}(x) =0$ for
all $x \in \cX$ and any $j \to i \in \vE$, as claimed.

Conversely, given a permissive specification, 
if a permissive set of messages $\nu$ is a stationary point of 
$\Phi(\cdot)$, then by the preceding we have that for 
any $i \to j \in \vE$, some positive $c_{i \to j}$ and all $y \in \cX$,
$$
\sum_{x} \big[ (\Tr\nu)_{i\to j}(x) -c_{i \to j} \nu_{i\to j}(x) \big]
\psi_{ij}(x,y) = 0 \,.
$$
By assumption the matrices $\{\psi_{ij}(x,y)\}$
are invertible, hence $\nu_{i\to j}(x) \normeq
(\Tr\nu)_{i\to j} (x)$ for any $i \to j \in \vE$.
The probability measures $\nu_{i \to j}$ and 
$(\Tr\nu)_{i\to j}$ are thus identical, for 
each directed edge $i \to j$. That is, 
the set of messages $\nu$ satisfies the
Bethe equations for the given specification.
\end{proof}

%
%

%
%
\subsection{Examples: Bethe equations and free entropy}
\label{sec:Bethe-examples} 

In most of this section we consider the extension 
of the Ising measure (\ref{eq:IsingModelGen}) on 
$\{+1,-1\}^V$, of the form
\begin{eqnarray}
\mu_{\beta,\uB,\bJ}(\ux) = \frac{1}{Z(\beta,\uB,\bJ)} 
\, \exp\Big\{\beta\sum_{(i,j)\in E}J_{ij}x_ix_j + \sum_{i \in V} B_i x_i
\Big\}\,, \label{eq:GeneralIsing}
\end{eqnarray}
where 
$\bJ = \{J_{ij}, (i,j) \in E\}$ for generic `coupling constants'
$J_{ij} \in \reals$ as in the spin-glass example of (\ref{eq:spin-glass}).
This model corresponds to the permissive specification 
$\psi_{ij}(x_i,x_j)=\exp(\beta J_{ij}x_ix_j)$ and $\psi_i(x_i)=\exp(B_i x_i)$. 
Since $\cX=\{+1,-1\}$, any set of 
messages $\{ \me_{i\to j} \}$ is effectively encoded through the 
`cavity fields' 
\begin{eqnarray}\label{eq:def-cav-field}
h_{i\to j} \equiv \frac{1}{2}\, \log\frac{\me_{i\to j}(+1)}{\me_{i\to j}(-1)} 
\,.
\end{eqnarray}
Using these cavity fields, we find the following formulas.
\begin{propo}
The Bethe equations for the cavity fields and the measure  
$\mu_{\beta,\uB,\bJ}(\cdot)$ are
\begin{eqnarray}
h_{i\to j} = B_i + \sum_{l\in\di\setminus j}\atanh\left\{\theta_{il} 
\tanh (h_{l\to i})\right\} \,,
\label{eq:BPIsing}
\end{eqnarray}
where $\theta_{il} \equiv \tanh(\beta J_{il})$.
The expected magnetization $\< x_i \>$ for this  
measure is approximated (in terms 
of the Bethe cavity fields $h^*_{i \to j}$), as
\begin{eqnarray}
\<x_i\> = \tanh\Big\{B_i+ \sum_{l\in\di}\atanh\big\{\theta_{il} 
\tanh (h^*_{l\to i})\big\}\Big\} \,,
\label{eq:Bethe-mag}
\end{eqnarray}
and the Bethe free entropy of any permissive cavity field 
$\uh = \{ h_{i \to j} \}$ is 
\begin{align}\label{eq:BetheFreeIsing}
&\Phi_{G,\beta,\uB,\bJ} (\uh) = \frac{1}{2}  \sum_{i \in V} 
\sum_{j \in \di} \Big\{ \ggm (\theta_{ij}) 
 -\log \big[ 1 + \theta_{ij} \tanh(h_{i\to j})
\tanh(h_{j \to i}) \big] \Big\} 
\\
&+
\sum_{i\in V}\log\Big\{ e^{B_i} 
\prod_{j\in\di} [ 1 + \theta_{ij} \tanh(h_{j \to i}) ]
+ e^{-B_i} \prod_{j \in \di} [ 1- \theta_{ij} \tanh(h_{j \to i}) ]
\Big\} \,,
\nonumber
\end{align}
where $\ggm (u) \equiv -\frac{1}{2} \log (1-u^2)$.
\end{propo}
\begin{proof} Expressing the BP mapping for the Ising measure 
$\mu(\ux) \equiv \mu_{\beta,\uB,\bJ}(\ux)$
in terms of cavity fields we find that  
$$
(\Tr h)_{i\to j}(x_i)\equiv \frac{e^{B_i x_i}}{\tilde{z}_{i\to j}}
\prod_{l\in\di\backslash{j}} \cosh (h_{l \to i} + \beta J_{il} x_i) 
$$
for some positive normalization constants $\tilde{z}_{i\to j}$. 
Thus, the identity 
$$
\frac{1}{2} \, \log \frac{\cosh(a+b)}{\cosh(a-b)} = \atanh(\tanh(a)\tanh(b))
\,,
$$
leads to the formula (\ref{eq:BPIsing}) for 
the Bethe equations. 
The approximation (\ref{eq:LocalMarg1})
of local marginals then results with 
$\mu_i(x_i)\normeq e^{B_i x_i} \prod_{l \in \di}
\cosh(h^*_{l \to i} + \beta J_{il} x_i)$, out of which
we get the formula (\ref{eq:Bethe-mag}) for 
$\<x_i\>=\mu_i(+1)-\mu_i(-1)$ 
by the identity $\frac{1}{2}\log(a)=\atanh(\frac{a-1}{a+1})$. 
Next note that if $u=\tanh(b)$ then $\ggm (u)=\log \cosh (b)$
and recall that 
by definition, for any $(i,j) \in E$ and $x \in \cX$, 
\begin{equation}\label{eq:ident1}
\me_{j\to i}(x)=\frac{\exp(x h_{j\to i})}{2\cosh(h_{j\to i})} \,.
\end{equation}
Hence, using the identity
$$
\frac{1}{4} \sum_{x,y \in \{+1,-1\}} 
\frac{e^{axy} e^{bx} e^{cy}}{\cosh(a)\cosh(b)\cosh(c)} 
= 1 + \tanh(a) \tanh(b) \tanh(c) \,,
$$ 
the first term in the formula (\ref{eq:BetheFree}) of 
the Bethe free entropy $\Phi(\cdot)$ is in this case
$$
 -\sum_{(i,j)\in E} \ggm (\theta_{ij})
 -\sum_{(i,j)\in E} \log \big[ 1 + \theta_{ij} \tanh(h_{i\to j})
\tanh(h_{j \to i}) \big] \,.
$$
Similarly, using (\ref{eq:ident1}) and the identity
$$
\frac{1}{2} \sum_{x \in \{+1,-1\}} \, 
\frac{e^{axy} e^{bx}}{\cosh(a)\cosh(b)} 
= 1 + y \tanh(a) \tanh(b)  \,,
$$ 
for $y=x_i \in \{+1,-1\}$, we find that the second term in the 
formula (\ref{eq:BetheFree}) is in our case
\begin{align*}
\sum_{i\in V} &\sum_{j \in \di} \ggm (\theta_{ij}) 
\\
&+ 
\sum_{i\in V}\log\Big\{ e^{B_i} 
\prod_{j\in\di} [ 1 + \theta_{ij} \tanh(h_{j \to i}) ]
+ e^{-B_i} \prod_{j \in \di} [ 1- \theta_{ij} \tanh(h_{j \to i}) ]
\Big\} \,.
\end{align*}
Combining the preceding expressions for the two terms 
of (\ref{eq:BetheFree}) we arrive at the formula of 
(\ref{eq:BetheFreeIsing}). 
\end{proof}

We proceed with a few special models of interest.

\vspace{0.1cm}

\noindent{\bf The Curie-Weiss model.} This model, which we already considered
in Section \ref{sec:Curie-Weiss}, corresponds to $G=K_n$ (the complete
graph of $n$ vertices), with $B_i=B$ and 
$J_{ij} = 1/n$ for all $ 1 \le i \ne j \le n$. 
Since this graph-specification pair is invariant under re-labeling of
the vertices, the corresponding Bethe equations
(\ref{eq:BPIsing}) admit at least one constant solution 
$h^*_{i\to j} = h^*(n)$, possibly dependent on $n$, such that
$$
h^*(n)  = B + (n-1) \atanh \{ \tanh(\beta/n) \tanh(h^*(n) )\} \,.
$$
These cavity fields converge as $n \to \infty$ to  
solutions of the (limiting) equation 
$h^*  = B + \beta \tanh(h^*)$. Further, 
the Bethe approximations (\ref{eq:Bethe-mag}) for the magnetization 
$m(n)=\<x_i\>$ are of the form $m(n)=\tanh(h^*(n)) + O(1/n)$ 
and thus converge as $n \to \infty$ to solutions of the 
(limiting) equation $m= \tanh(B+\beta m)$. Indeed, we have already 
seen in Theorem \ref{thm:LargeDevMagn} 
that the Curie-Weiss magnetization (per spin) 
concentrates for large $n$ around the relevant solutions 
of the latter equation. 

\vspace{0.15cm}

\noindent{\bf Ising models on random $k$-regular graphs.}
By the same reasoning as for the Curie-Weiss model, 
in case of a $k$-regular graph of $n$ vertices 
with $J_{ij}= +1$, 
and $B_i=B$, the Bethe equations
admit a constant solution $h_{i\to j} = h^*$ such that
$$ 
h^* = B + (k-1)\atanh\{\theta \tanh(h^*)\} \,,
$$ 
for $\theta \equiv \tanh(\beta)$, with the corresponding 
magnetization approximation 
$m=\tanh\big( B + k\atanh\{\theta \tanh(h^*)\} \big)$
and Bethe free entropy 
\begin{eqnarray*}
n^{-1} \Phi_n(h^*)  
&\!\!\!\!=\!\!\!&  
 \frac{k}{2} \Big\{  
 \ggm (\theta) 
 -\log \big[ 1 + \theta \tanh^2 (h^*)
 \big] \Big\} \\
&\!\!\!\!+\!\!\!& \log\Big\{ e^{B} 
 [ 1 + \theta \tanh(h^*) ]^k 
+ e^{-B} [ 1-\theta \tanh(h^*) ]^k 
\Big\}\,. 
\end{eqnarray*}

\vspace{0.15cm}

\noindent{\bf Ising models on $k$-regular trees.}
It is instructive to contrast the above free entropy 
with the analogous result for  rooted $k$-regular trees $T_k(\ell)$.  From 
Proposition \ref{prop:Bethe-free-tree} we know 
that the free entropy $\log Z_\ell(B,\beta)$ for 
the Ising measure on the finite tree $T_k(\ell)$ 
is precisely the Bethe free entropy of 
(\ref{eq:BetheFreeIsing}) for the unique 
solution of the Bethe equations (\ref{eq:BPIsing})
with $J_{ij}=+1$ and $B_i=B$.

We 
denote 
by $n_{t}$ the number of vertices
at generation $t\in \{0,\dots, \ell\}$ (thus $n_0=1$ and
$n_t = k(k-1)^{t-1}$ for $t\ge 1$), and by
$$
n(\ell)=|T_k(\ell)|= k((k-1)^{\ell}-1)/(k-2)\,,
$$ 
the total number of vertices in $T_k(\ell)$. Due to symmetry 
of $T_k(\ell)$, the Bethe cavity field 
assumes the same value $h_r$ on all 
directed edges leading from a vertex at generation 
$\ell-r$ to one at generation $\ell-r-1$ of $T_k(\ell)$. 
Thus, we have   
\begin{equation}\label{eq:IsingTreeRec-n}
h_r  = B + (k-1)\atanh( \theta \tanh h_{r-1})  \,,   
\end{equation}
with initial condition $h_{-1} =0$. 
Similarly, we denote by $h_{r}^{\ell}$ 
of the Bethe cavity field on the $n_{\ell-r}$ directed edges 
leading from a vertex at generation $\ell-r-1$ to one at
generation $\ell-r$.
We then have
$$
h_r^{\ell} = B + (k-2)\atanh (\theta \tanh h_r) 
+ \atanh(\theta \tanh h_{r+1}^{\ell}) \,,
$$ 
for $r=\ell-1,\ell-2,\ldots,0$, with initial condition 
$h_\ell^{\ell}= 
h_{\ell-1}$.
The (Bethe) free 
entropy is in this case 
\begin{align*}
\log Z_\ell(B,\beta) &= (n(\ell)-1) \ggm (\theta) 
-\sum_{r=0}^{\ell-1} n_{\ell-r} 
\log \big[ 1 + \theta \tanh h_r \tanh h_r^{\ell} \big] \\
&+ \sum_{r=0}^{\ell} n_{\ell-r} \log \Big\{ e^{B} 
 [ 1 + \theta\tanh h_{r-1}  ]^{k-1} [ 1 + \theta \tanh h_r^{\ell}  ] 
\\
&\;\;\;\;\;\quad\qquad + e^{-B} 
 [ 1 - \theta \tanh h_{r-1}  ]^{k-1} [ 1 - \theta \tanh h_r^{\ell}  ] \Big\} \,.
\end{align*}
Using the relation (\ref{eq:IsingTreeRec-n}) you can verify 
that the preceding formula simplifies to
\begin{align*}
\log & Z_\ell(B,\beta) = (n(\ell)-1) \ggm (\theta) 
\\
&+ \log \Big\{ e^{B} 
 [ 1 + \theta\tanh h_{\ell-1}  ]^{k} 
+ e^{-B} 
 [ 1 - \theta \tanh h_{\ell-1}  ]^{k} \Big\}\\
&+\sum_{r=0}^{\ell-1} n_{\ell-r} \log \Big\{ e^{B} 
 [ 1 + \theta\tanh h_{r-1}  ]^{k-1} 
+ e^{-B} 
 [ 1 - \theta \tanh h_{r-1}  ]^{k-1} \Big\}
\,.
\end{align*}
The $\ell\to\infty$ limit can then be
expressed in terms of the $k$-canopy tree $\CTree_k$
(c.f. Lemma \ref{dfn:canopy}).
If $R$ denotes the random location of the root of 
$\CTree_k$, then we get
\begin{align*}
\lim_{\ell \to\infty} & \frac{1}{n(\ell)} \log Z_\ell(B,\beta) = \\
&\ggm (\theta) +  \E \log \Big\{ e^{B} 
 [ 1 + \theta\tanh h_{R-1}  ]^{k-1} 
+ e^{-B} 
 [ 1 - \theta \tanh h_{R-1}  ]^{k-1} \Big\} \,.
\end{align*}

\vspace{0.15cm}

\noindent
{\bf Locally tree-like graphs.} Recall Remark \ref{rem:local-trees1}, that 
$k$-regular graphs converge locally
to the Galton-Watson tree $\Tree(\node,\edge,\infty)$ with 
$\node_k=1$. More generally, consider 
the ferromagnetic Ising model $\mu_{\beta,B}(\ux)$ of 
(\ref{eq:IsingModel}), namely, with $J_{ij}=+1$ and $B_i=B$,
for a uniformly sparse 
graph sequence $\{G_n\}$ that converges locally
to the random rooted tree $\Tree$. Then, for any $n$ and
cavity field $\uh= \{h_{i \to j}\}$ we have from 
(\ref{eq:BetheFreeIsing}) that  
\begin{align*}
 n^{-1} \Phi_n &(\uh) = 
 \frac{1}{2} \E_n \Big[ \sum_{j \in \di}  
\big\{ \ggm (\theta) -
\log [ 1 + \theta \tanh(h_{i\to j})
\tanh(h_{j \to i}) ] \big\} \Big] \\
&+
\E_n \Big[ \log\big\{ e^{B} 
\prod_{j\in\di} [ 1 + \theta \tanh(h_{j \to i}) ]
+ e^{-B} \prod_{j \in \di} [ 1-\theta \tanh(h_{j \to i}) ]
\big\} \Big] \,,
\end{align*}
where $\E_n$ corresponds to expectations with 
respect to a uniformly chosen $i \in V_n$.
For $n \to \infty$, as shown in Lemma \ref{lemma:EdgeTree}
we have by local convergence and 
uniform sparsity that these expectations converge to the
corresponding expectations on the  
tree $\Tree$ rooted at $\root$. Consequently, we expect to have
\begin{align*}
\lim_{n \to \infty} n^{-1} & \Phi_n (\uh^*_n)  = 
 \frac{1}{2} \E \Big[ \sum_{j=1}^L \big\{ \ggm (\theta)  
 -  \log [ 1 + \theta \tanh(h^*_{\root \to j})
\tanh(h^*_{j \to \root}) ] \big\} \Big] \\
&+
\E \Big[ \log\big\{ e^{B} 
\prod_{j=1}^L [ 1 + \theta \tanh(h^*_{j \to \root}) ]
+ e^{-B} \prod_{j=1}^L [ 1-\theta \tanh(h^*_{j \to \root}) ]
\big\} \Big] \,,
\nonumber
\end{align*}
where $L=|\partial \root|$, the variables 
$\{\tanh(h_{j \to \root}^*)\}$ are the limit as $t \to \infty$ 
of the Ising magnetizations $\<x_j\>^{(t)}_{j}$ 
on the sub-trees of $j \in \partial \root$ and all its descendants 
(in $\Tree(t)$, either with free or plus boundary conditions),
and for $j=1,\ldots,L$,
$$
h_{\root \to j}^* = B 
+ \sum_{k=1, k \ne j}^L \atanh\{ \theta \tanh(h^*_{k \to \root}) \} \,.
$$
Indeed, this 
is precisely
the prediction (\ref{eqn:phi}) for the free entropy density
of ferromagnetic Ising models on such graphs (which is 
proved in \cite{ising} to hold in case
$\Tree$ is a Galton-Watson tree).

\vspace{0.15cm}

\noindent 
{\bf The Sherrington-Kirkpatrick model.}
The Sherrington-Kirkpatrick spin-glass model 
corresponds to the complete graph $G_n=K_n$
with the scaling 
$\beta\to\beta/\sqrt{n}$, constant 
$B_i=B$ and 
$J_{ij}$ which are i.i.d. standard normal random variables.
Expanding the corresponding Bethe equations (\ref{eq:BPIsing}), we 
find that for large $n$ and any $i,j$, 
\begin{eqnarray}
h_{i\to j} = B+ \frac{\beta}{\sqrt{n}} 
\sum^n_{l=1, l \ne i,j} J_{il}
\tanh (h_{l\to i}) + o(\frac{1}{\sqrt{n}})\, .
\label{eq:CavitySK}
\end{eqnarray}
Similarly, expanding the formula (\ref{eq:Bethe-mag}), we get 
for the local magnetizations $m_i \equiv \<x_i\>$ and large $n$ that
\begin{eqnarray*}
\atanh (m_i) = h_{i \to j} +
\frac{\beta J_{ij}}{\sqrt{n}} \tanh (h_{j\to i}) + o(\frac{1}{\sqrt{n}}) = 
h_{i\to j} +\frac{\beta J_{ij}}{\sqrt{n}}\, m_j+o(\frac{1}{\sqrt{n}})\, .
\end{eqnarray*}
Substituting this in both sides of equation (\ref{eq:CavitySK}), 
and neglecting terms of $O(n^{-1/2})$ yields the so-called 
TAP equations
\begin{eqnarray}
\atanh (m_i) = B + \frac{\beta}{\sqrt{n}} \sum^n_{l=1, l \ne i} 
J_{il}m_l-m_i \frac{\beta^2}{n} \sum^n_{l=1, l \ne i} J_{il}^2 (1-m_l^2)\, .
\end{eqnarray}

\vspace{0.15cm}

\noindent
{\bf The independent set model.}
In this model, which is not within the framework of 
(\ref{eq:GeneralIsing}), we consider the measure
\begin{eqnarray}
\mu_{G,\lambda} (\ux) = \frac{1}{Z(G,\lambda)}\, 
\lambda^{|\ux|}\prod_{(i,j)\in E}\,
\ind((x_i,x_j) \neq (1,1))\, ,
\label{eq:IndependentSet}
\end{eqnarray}
where $|\ux|$ denotes the number of non-zero entries in 
the vector $\ux \in \{0,1\}^V$. It corresponds to the permissive 
specification $\psi_{ij}(x,y) = \ind\big((x,y)\neq (1,1)\big)$,
and $\psi_i(x)=\lambda^x$, having $x_i^{\rm p} = 0$ for all 
$i \in V$. In this case the Bethe equations are 
$$
\me_{i \to j} = \frac{1}{1+\lambda \prod_{l \in \di \setminus j} 
\me_{l \to i}} \,,
$$
for $\me_{i \to j} \equiv \me_{i \to j}(0)$ and their solution
$\{\me^*_{i \to j}\}$ provides the approximate densities 
$$
\mu(x_i=1) = 
\frac{\lambda \prod_{j \in \di} \me^*_{j \to i}}
{1+\lambda \prod_{j \in \di} \me^*_{j \to i}} \,,
$$
and the approximate free entropy 
$$
\Phi(\me^*) 
= \sum_{i \in V} \log \Big\{1+\lambda \prod_{j \in \di} \me^*_{j\to i}\Big\}
- \sum_{(i,j) \in E} \log [ \me^*_{i\to j} + \me^*_{j \to i}
- \me^*_{i \to j} \me^*_{j \to i} ] \,.
$$

%
%
\subsection{Extremality, Bethe states and Bethe-Peierls approximation}
\label{sec:BetheFormal}

Following upon Section \ref{sec:BetheInformal}
we next define the Bethe-Peierls approximation of 
local marginals in terms of a given set of messages.
To this end, recall that each 
subset 
$U\subseteq V$ 
has a (possibly infinite) diameter 
$\diam(U) = \max \{ d(i,j) : i,j\in U \}$
(where $d(i,j)$ is the number of edges traversed in the 
shortest path on $G$ from $i \in V$ to $j \in V$), and it induces 
the subgraph $G_U = (U,E_U)$ such that $E_U = \{ (i,j)\in E : i,j \in U \}$.
%

\begin{definition}
Let $\cU$ denote the collection of $U \subseteq V$ 
for which $G_U=(U,E_U)$ is a tree and 
each $i \in \dU$ is a leaf of $G_U$ 
(i.e. $|\di\cap U|=1$ whenever $i \in \dU$).
A set of messages $\{\nu_{i\to j}\}$ 
induces on each $U \in \cU$ the probability measure 
\begin{eqnarray}
\nu_{U}(\ux_U) = \frac{1}{Z_U}\, \prod_{i\in U} \psi^*_i(x_i)
\prod_{(ij)\in E_U}\psi_{ij}(x_i,x_j) \, ,
\label{eq:PatchDefinition}
\end{eqnarray}
where $\psi^*_i(\cdot)=\psi_i(\cdot)$ except for $i \in \dU$ 
in which case $\psi_i(\cdot)=\nu_{i\to u(i)}(\cdot)$ with 
$\{u(i)\} = \di \cap U$.

A probability measure $\rho(\ux)$ on $\cX^V$ is 
$(\ve,r)$-\emph{Bethe approximated} by a set of
messages $\{\nu_{i\to j}\}$ if
\begin{eqnarray}
\sup_{U \in \cU, \diam(U) \le 2r} \, ||\rho_U-\nu_{U} ||_{\sTV}\le \ve\, ,
\label{eq:DistanceCondition}
\end{eqnarray}
where $\rho_U(\,\cdot\,)$ denotes the marginal distribution of $\ux_U$
under $\rho(\cdot)$. We call any such $\rho(\cdot)$ an 
\emph{$(\ve,r)$-Bethe state} 
for the graph-specification pair $(G,\upsi)$.
\end{definition}
\begin{remark}
Note that if $i \notin \dU$ is a leaf of an 
induced tree $G_U$ then $\di=\{u(i)\}$ and if
$\{\nu_{i \to j}\}$ is a permissive set of messages then  
$\nu_{i\to u(i)}(\cdot) \normeq \psi_i(\cdot)$.
Consequently, in (\ref{eq:PatchDefinition}) we
may and shall not distinguish between $\dU$ 
and the collection of all leaves of $G_U$.
\end{remark}

We phrase our error terms and correlation properties 
in terms of valid rate functions, and consider 
graphs that are \emph{locally tree-like}. Namely,
\begin{definition}
A \emph{valid rate function} is a monotonically non-increasing 
function $\delta:\naturals\to [0,1]$ that decays to zero
as $r\to\infty$. By (eventually) increasing $\delta(r)$,
we assume, without loss of generality,
that $\delta(r+1)\ge\delta_* \delta(r)$ for some positive
$\delta_*$ and all $r \in \naturals$.

Given an integer $R\ge 0$ we say that $G$ is $R$-\emph{tree like}
if its girth exceeds $2R+1$ (i.e. 
$\Ball_i(R)$ is a tree for every $i \in V$).
\end{definition}

We show in the sequel 
that the Bethe approximation 
holds when the canonical measure on a tree like graph
satisfies the following correlation decay hypotheses. 
\begin{definition}\label{def:Extremal}
A probability measure $\rho$ on $\cX^{V}$ 
is \emph{extremal} for $G$ with valid rate function 
$\delta(\cdot)$ if for any $A,B\subseteq V$, 
\begin{eqnarray}
||\rho_{A,B}(\,\cdot\, ,\,\cdot\,)-\rho_A(\,\cdot\,)
\rho_B(\,\cdot\,)||_{\sTV}\le \delta(d(A,B))\, ,
\end{eqnarray}
where $d(A,B)=\min \{ d(i,j): i \in A, j \in B \}$ is the length of
the shortest path in $G$ between $A \subseteq V$ and $B \subseteq V$. 
\end{definition}

We consider the notions of Bethe measure and extremality for
general probability distributions over $\cX^V$ (and not only
for the canonical measure $\mu_{G,\upsi}(\,\cdot\, )$). The key (unproven)
assumption of statistical physics approaches is that the canonical
measure (which is ultimately, the object of interest),
can be decomposed as a unique convex combination of extremal measures,
up to small error terms. This motivates the name `extremal'.
Further, supposedly each element of this decomposition can 
then be treated accurately within its Bethe approximation.

Here is the first step in verifying this broad conjecture, dealing with 
the case where the canonical measure $\mu_{G,\upsi}(\,\cdot\,)$ 
is itself extremal.
\begin{thm}\label{thm:Bethe}
Let $\upsi$ be a permissive specification
for an $R$-tree like graph $G$ and $\delta(\, \cdot\, )$
a valid rate function. If $\mu_{G,\upsi}(\cdot)$  
is extremal with  rate $\delta(\, \cdot\, )$ then it is 
$(\ve,r)$-Bethe approximated by its standard message set 
for $\ve = \exp(c^r) \delta(R-r)$ and all $r<R-1$,
where the (universal) constant $c$ depends only 
on $|\cX|$, $\delta_*$, $\kappa$ and 
the maximal degree $\Delta \ge 2$ of $G$.
In particular, $\mu_{G,\upsi}(\,\cdot\,)$ is then
an $(\ve,r)$-Bethe state for this graph-specification pair.
%
\end{thm}

To prove the theorem, recall first that for any
probability measures $\rho_a$ on a discrete set $\cZ$ 
and $f: \cZ \mapsto [0,f_{\max}]$ we have the elementary bound 
\begin{eqnarray}
||\rh_1-\rh_2||_{\sTV}\le \frac{3f_{\max}}{2 \langle \rho_1,f \rangle} 
||\rho_1-\rho_2||_{\sTV} \,,
\label{eq:SimpleIneq}
\end{eqnarray}
where $\rh_a(z)\equiv\rho_a(z)f(z)/ \langle \rho_a,f \rangle$ and
$\langle \rho_a,f \rangle \equiv \sum_{z\in\cZ} \rho_a(z)f(z)$
(c.f. \cite[Lemma 3.3]{ising}). Further, it is easy to check 
that if $\mu(\cdot)=\mu_{G,\upsi}(\cdot)$ and 
$(G,\upsi)$ is a permissive graph-specification pair, then 
for any $C\subseteq V$, 
\begin{eqnarray}
\label{eq:LowerBound}
\mu_{C} (\ux_C^{\rm p}) &\ge& \cX^{-|C|} \kappa^{\Delta |C|}\,.  
\end{eqnarray}
In addition, as shown in \cite[Section 3]{bethe}, for such 
$\mu(\cdot)$, if $G_{U'}$ is a tree, $(i,j) \in E_{U'}$
and $j \notin A \supseteq \dU'$, then 
\begin{eqnarray}\label{eq:LinkDistribution}
||\mu^{(ij)}_{i|A}(\,\cdot\,|\ux_{A})-
\mu^{(ij)}_{i|A}(\,\cdot\,|\uy_{A})||_{\sTV}\le
b
 ||\mu_{ij|A}(\,\cdot\,|\ux_{A})-
\mu_{ij|A}(\,\cdot\,|\uy_{A})||_{\sTV}\,, 
\end{eqnarray}
for $b \equiv 2 |\cX| \kappa^{-(\Delta+1)}$ and 
all $\ux, \uy\in\cX^V$. Finally, the following lemma 
is also needed for our proof of the theorem.
\begin{lemma}\label{lem:andrea}
If the canonical measure $\mu$
for $2$-tree like graph and a permissive specification
is extremal of valid rate function $\delta(\cdot)$ then 
for some finite $K=K(|\cX|,\kappa,\Delta)$ and any $A \subseteq V$ 
$$
||\mu^{(ij)}_{A}-\mu_{A} ||_{\sTV} \le K \delta\big(d(\{i,j\},A)\big) \,.
$$
\end{lemma}
\begin{proof} Set $B=\di\cup\dj\setminus \{i,j\}$ and 
$C=\bigcup_{l \in B} \, \partial l$ noting that $|B| \le 2(\Delta-1)$,
$|C| \le 2\Delta(\Delta-1)$ and since $G$ is $2$-tree like,
necessarily the induced subgraph $G_B$ has no edges. Hence,
\begin{align*}
\mu_{B}(\ux_B)\ge& 
\mu_C(\ux_C^{\rm p})\mu_{B|C}(\ux_B|\ux^{\rm p}_C)\\
\ge &\mu_C(\ux_C^{\rm p})\prod_{l\in B} \Big( \frac{\psi_l(x_l) 
\prod_{k\in\partial l}\psi_{lk}(x_l,x_k^{\rm p})}{
\sum_{x'_l}
\psi_l(x_l') \prod_{k\in\partial l}\psi_{lk}(x_l',x_k^{\rm p})}\Big)\ge
\mu_C(\ux_C^{\rm p})\, \kappa^{2(\Delta^2-1)}\, 
\end{align*}
so by the bound (\ref{eq:LowerBound}) we deduce that 
$\mu_B(\ux_B)\ge c_0$ for all $\ux_B$ and some positive
$c_0=c_0(|\cX|,\kappa,\Delta)$. 
Next assume, without loss of generality,
that $A\cap B = \emptyset$. Then 
\begin{align*}
||\mu^{(ij)}_{A}-\mu_{A} ||_{\sTV}&= \frac{1}{2} 
\sum_{\ux_A}
\Big|\sum_{\ux_B}\mu^{(ij)}_{B}(\ux_B)\mu_{A|B}(\ux_A|\ux_B)-
\sum_{\ux'_B}\mu_{B}(\ux'_B)\mu_{A|B}(\ux_A|\ux'_B)\Big|\\
&\le
\sup_{\ux_B,\ux'_B}||\mu_{A|B}(\,\cdot\,|\ux_B)-
\mu_{A|B}(\,\cdot\,|\ux'_B)||_{\sTV} \\
&\le 
\frac{1}{c_0^2}
\E \Big\{ ||\mu_{A|B}(\,\cdot\,|\uX^{(1)}_B)-
\mu_{A|B}(\,\cdot\,|\uX^{(2)}_B)||_{\sTV} \Big\}\, ,
\end{align*}
where $\uX^{(1)}$ and $\uX^{(2)}$ are independent random 
configurations, each of distribution $\mu$. Next,
from the extremality of $\mu(\cdot)$ we deduce that 
\begin{eqnarray*}
\E\big\{ ||\mu_{A|B}(\,\cdot\,|\uX^{(1)}_B)-
\mu_{A|B}(\,\cdot\,|\uX^{(2)}_B)||_{\sTV} \big\} 
\le 2 \delta\big(d(A,B)\big)\,, 
\end{eqnarray*}
so taking $K=2/c_0^2$ we arrive at our thesis.
\end{proof}

\begin{proof}[Proof of Theorem \ref{thm:Bethe}] 
Fixing $r<R-1$, a permissive 
graph-specification pair $(G,\upsi)$  
that is extremal for $R$-tree like graph
$G$ with valid rate function $\delta(\, \cdot\, )$ and
$U \in \cU$ with $\diam(U)\le 2r$, 
let $\oU_{R'}=\{k \in V: d(k,U) \ge R'\}$
for $R'=R-r>1$. 
Note that 
\begin{align}\label{eq:triangtv}
||\mu_U (\cdot)  - \nu_U (\cdot) ||_{\sTV} \, &\le \E||\mu_U(\,\cdot\,)
-\mu_{U|\overline{U}_{R'}}(\,\cdot\,|\wuX_{\overline{U}_{R'}})
||_{\sTV} \nonumber \\
&+ \E|| \mu_{U|\overline{U}_{R'}}(\,\cdot\,|
\wuX_{\overline{U}_{R'}})-\nu_U(\,\cdot\,)||_{\sTV}\, ,
\end{align}
where $\nu_U$ corresponds to the standard message set 
(i.e. $\nu_{i\to j} = \mu^{(ij)}_i$ for the measure 
$\mu^{(ij)}(\cdot)$ of (\ref{eq:CavityMu})), and  
the expectation is with respect to the random configuration 
$\wuX$ of distribution $\mu$. The first term on the right side 
is precisely 
$||\mu_{U,\overline{U}_{R'}}(\,\cdot\, , \,\cdot\,)
-\mu_U(\,\cdot\,)\mu_{\overline{U}_{R'}}(\,\cdot\,)||_{\sTV}$ 
which for $\mu(\cdot)$ extremal of valid rate function 
$\delta(\cdot)$ is bounded by $\delta(d(U,\overline{U}_{R'}))=\delta(R-r)$.
Turning to the second term, 
consider the permissive set of messages 
\begin{eqnarray*}
\wnu_{i\to j} (x_i) = 
\mu^{(ij)}_{i|\cBall_i(R')}(x_i|\wuX_{\cBall_i(R')}) \, ,
\end{eqnarray*}
where $\cBall_i(t)$ denotes the  
collection of vertices of distance at least $t$ from $i$.
Since $\diam(U) \le 2r$ there exists
$i_o \in V$ such that $U \subseteq \Ball_{i_o}(r)$ and as 
$\Ball_{i_o}(R)$ is a tree, 
the canonical measure for $\Ball_{i_o}(R) \setminus G_U$ is the 
product of the corresponding measures for the subtrees rooted 
at $i\in\dU$. Noting that $V \setminus \Ball_{i_o}(R) \subseteq
\overline{U}_{R'}$, 
it is thus not hard to verify that we have the representation 
\begin{eqnarray}
\mu_{U|\overline{U}_{R'}}(\ux_U|\wuX_{\overline{U}_{R'}}) = 
\frac{1}{\widetilde{Z}_U}
\, \prod_{i\in U} \wpsi^*_i(x_i)
\prod_{(ij)\in E_U}\psi_{ij}(x_i,x_j) \, ,
\label{eq:LocalCond}
\end{eqnarray}
as in (\ref{eq:PatchDefinition}), corresponding to the 
messages $\{\wnu_{i \to j}\}$ (i.e. with 
$\wpsi^*_i(\cdot)=\psi_i(\cdot)$ except for $i \in \dU$ 
in which case $\wpsi_i(\cdot)=\wnu_{i\to u(i)}(\cdot)$).
Consequently, we proceed to bound $|| \wnu_U -\nu_U||_{\sTV}$ 
by applying the inequality (\ref{eq:SimpleIneq}) for  
the function 
$$
f(\ux_U) = \prod_{i \in U \setminus \dU} \psi_i(x_i) 
\prod_{(ij)\in E_U}\psi_{ij}(x_i,x_j)
$$ 
on $\cZ=\cX^{U}$ and probability measures $\rho_a$ that are uniform 
on $\cX^{U\setminus\dU}$ with 
$\rho_1(\ux_{\dU}) = \prod_{i\in\dU}\nu_{i\to u(i)}(x_i)$ and 
$\rho_2(\ux_{\dU}) = \prod_{i\in\dU}\wnu_{i\to u(i)}(x_i)$.  
To this end, recall that $f(\ux_U) \le  
f_{\max}= \psi_{\max}^{M}$ for $M=|U|-|\dU|+|E_U|$.
Further, since $G_U$ is a tree (hence $|E_U| \le |U|$), 
and $\upsi$ is a permissive specification 
(also when $(i,j)$ is removed from $E$), upon 
applying (\ref{eq:LowerBound}) for $|C|=1$, we have that 
\begin{align*}
\langle \rho_1 , f \rangle \ge& 
\prod_{i \in U \setminus \dU} \frac{\psi_i(x^{\rm p}_i)}{|\cX|} 
\prod_{(ij)\in E_U}
\psi_{ij}(x^{\rm p}_i,x^{\rm p}_j)\prod_{i\in\dU}
\nu_{i\to u(i)}(x_i^{\rm p}) \\
\ge& 
f_{\max} |\cX|^{-|U|} \kappa^{M+\Delta |\dU|} 
\ge 
f_{\max} c_1^{-|U|}\, ,
\end{align*}
where $c_1=|\cX| \kappa^{-(\Delta+1)}$ is a finite constant. 
Consequently, we deduce upon applying (\ref{eq:SimpleIneq}) that   
\begin{align}\label{eq:ineq1}
||\mu_{U|\overline{U}_R}(\,\cdot\,|\wuX_{\overline{U}_R})-\nu_U(\,\cdot\,)
||_{\sTV} &=
||\rh_2-\rh_1||_{\sTV} \le 2 c_1^{|U|} ||\rho_1-\rho_2||_{\sTV} 
\nonumber \\
&\le 2 c_1^{|U|} \sum_{i\in\dU} 
|| \nu_{i\to u(i)}-\wnu_{i\to u(i)}||_{\sTV}\, .
\end{align}
Following \cite{bethe} we show in the sequel that 
\begin{eqnarray}\label{eq:amir-new}
\E \big\{ || \nu_{i\to u(i)}-\wnu_{i\to u(i)}||_{\sTV} \big\} 
\le c_2 \delta(R-r)\, ,
\end{eqnarray}
for some finite $c_2=c_2(|\cX|,\Delta,\kappa,\delta_{*})$ 
and all $i \in \dU$.
As $|\dU| \le |U|\le |\Ball_{i_o}(r)| \le \Delta^{r+1}$, we can 
choose $c=c(|\cX|,\Delta,\kappa,\delta_*)$ finite such that 
$1 + 2 c_1^{|U|} |\dU| c_2 \le \exp(c^r)$. Then, combining 
the inequalities (\ref{eq:triangtv}), (\ref{eq:ineq1}) and
(\ref{eq:amir-new}) results with 
$$
||\mu_U-\nu_U||_{\sTV} \le \exp(c^r) \, \delta(R-r)\, ,
$$
for every $U \in \cU$ of $\diam(U) \le 2r$ and $r<R-1$, which is 
the thesis of Theorem \ref{thm:Bethe}.

As for the proof of (\ref{eq:amir-new}), fixing $i \in \dU$ 
let $A=\cBall_i(R')$ and 
$\nu'_{i \to j}=\mu^{(ij)}_{i|A}(\cdot|\uX'_{A})$ 
where $\uX'$ of distribution $\mu^{(ij)}$ is independent of $\wuX$.
Then,
\begin{align}\label{eq:ineq1amir}
\E \big\{ || \nu_{i\to j}-\wnu_{i\to j}||_{\sTV} \big\} &=
\E \big\{ || \, \E \nu'_{i\to j}-\wnu_{i\to j}||_{\sTV} \}
\nonumber\\ 
&\le \E \big\{ || \nu'_{i\to j}-\wnu_{i\to j}||_{\sTV} \} \,.
\end{align}
Further, setting $U'=\Ball_i(R')$ note that 
$G_{U'}$ is a tree (since $G$ is $R$-tree like),
such that $\dU' \subseteq A$ (while $\di$ and $A$ are disjoint).
Thus, from (\ref{eq:LinkDistribution}) we have that 
for any $j \in \di$, 
\begin{align}\label{eq:ineq2}
|| \nu'_{i\to j}-\wnu_{i\to j}||_{\sTV}
&= 
||\mu^{(ij)}_{i|A}(\,\cdot\,|\uX'_{A})-
\mu^{(ij)}_{i|A}(\,\cdot\,|\wuX_{A})||_{\sTV}
\nonumber \\
&\le
b
\, ||\mu_{ij|A}(\,\cdot\,|\uX'_{A})-
\mu_{ij|A}(\,\cdot\,|\wuX_{A})||_{\sTV}\, .
\end{align}
Taking the expectation with respect to the independent random 
configurations $\uX'$ (of law $\mu^{(ij)}$) and $\wuX$ (of law
$\mu$), leads to  
\begin{align*}
\E \big\{ ||
\mu_{ij|A}(\,\cdot\,|\uX'_{A})
&- \mu_{ij|A}(\,\cdot\,|\wuX_{A})
||_{\sTV} \big\} \\
&\le 2 ||\mu_{\{ij\},A} -
\mu_{\{ij\}} \mu_{A} 
||_{\sTV} 
+ ||\mu^{(ij)}_{A}-\mu_{A} ||_{\sTV} \,.
\end{align*}
For $\mu$ extremal of valid rate function $\delta(\cdot)$
the latter expression is, due to Lemma \ref{lem:andrea}, bounded  
by $(2 + K) \delta(R'-1) \le (2+K) \delta(R-r)/\delta_*$, 
which together with (\ref{eq:ineq1amir}) and 
(\ref{eq:ineq2}) results with (\ref{eq:amir-new}).
\end{proof}

%
%
\section{Colorings of random graphs}
\label{chap:Coloring}
\setcounter{equation}{0}

Given a graph $G=(V,E)$, recall that a proper $q$-coloring of $G$
is an assignment of colors to the vertices of $G$ such that no edge has both
end-points of the same color.
Deciding whether a graph is $q$-colorable is a classical NP-complete 
constraint satisfaction problem. Here we shall study this problem when 
$G$ is sparse and random.  More precisely,  
we shall consider the uniform measure $\mu_G(\,\cdot\,)$ over 
proper $q$-colorings of $G$, with $q\ge 3$.

As the average degree of $G$ increases, the measure $\mu_G(\,\cdot\,)$
undergoes several phase transitions and exhibits
coexistence when the average degree is within a certain interval. 
Eventually, for any $q$, if the average degree is large enough,
a random graph becomes, with high probability, non $q$-colorable.
Statistical physicists have put forward a series of exact 
conjectures on these phase transitions 
\cite{FlorentLenka,OurPNAS,FirstColoring},
but as of now most of it can not be rigorously verified (c.f. 
\cite{AchlioptasCoja,AchlioptasNaor,AchlioptasNaorPeres}
for what has been proved so far).

We begin in Section \ref{sec:Broad} with an
overview of the 
various phase transitions as they emerge from the statistical mechanics picture.
Some 
bounds on the $q$-colorability of a random graph are
proved in Section \ref{sec:COL-UNCOL}. Finally, Section 
\ref{sec:Clust} 
explores the nature of the coexistence threshold 
for $q$-coloring, in particular, connecting it with the question 
of information reconstruction, to which Section \ref{ch:Reco} is devoted.

\subsection{The phase diagram: a broad picture}
\label{sec:Broad}

Let $\ux = \{x_i:\, i\in V\}$ denote a $q$-coloring of the
graph $G=(V,E)$  (i.e. for each vertex $i$, 
let $x_i\in\{1,\dots,q\} \equiv \cX_q$). 
Assuming that the graph $G$ admits a proper $q$-coloring, the 
uniform measure over the set of proper $q$-colorings of $G$ is
\begin{eqnarray}
\mu_G(\ux) = \frac{1}{Z_G}\, \prod_{(i,j)\in E}\,\ind(x_i\neq x_j)\, ,
\label{eq:CanonicalCol}
\end{eqnarray}
with $Z_G$ denoting the number of proper $q$-colorings of $G$.
We shall consider the following two examples of a random graph
$G=G_n$ over the vertex set $V=[n]$:
\begin{enumerate}
\item[(a).] $G=G_{n,\alpha}$ is uniformly chosen from the
Erd\"os-Renyi ensemble $\graph(\alpha,n)$ of 
graphs of $m=\lfloor n\alpha \rfloor$ edges (hence of average 
degree $2\alpha$).
\item[(b).] $G=G_{n,k}$ is a uniformly chosen random $k$-regular graph.
\end{enumerate}

Heuristic statistical mechanics studies suggest a rich phase transition
structure for the measure $\mu_G(\, \cdot\,)$. For 
any
$q\ge 4$, different regimes are separated by three distinct
critical values of the average degree:
$0<\alpha_{\rm d}(q)<\alpha_{\rm c}(q)<\alpha_{\rm s}(q)$
(the case $q=3$ is special in that $\alpha_{\rm d}(q)=
\alpha_{\rm c}(q)$, whereas 
$q=2$ is rather trivial, as $2$-colorability 
is equivalent to having no odd cycles, in which case 
each connected component of $G$ admits two proper
colorings, independently of the coloring of the rest of $G$).
In order to characterize such phase transitions we will use two notions
(apart from colorability), namely coexistence and \emph{sphericity}.
To define the latter notion we recall that the 
joint type of two color assignments $\ux= \{x_i:\, i\in V\}$ and
$\uy = \{y_i:\, i\in V\}$ is a $q\times q$ matrix whose $x,y$ entry 
(for $x,y\in\{1,\dots,q\}$) is the fraction of vertices 
with color $x$ in the first assignment and color $y$ in the second.
\begin{definition}\label{def:Spherical}
Let $\nu = \{\nu(x,y)\}_{x,y\in [q]}$ be the joint type of
two independent color assignments, 
each distributed according to $\mu_G(\,\cdot\,)$, 
with $\onu(x,y) = 1/q^{2}$ denoting the uniform joint type. 
We say that $\mu_G$ is $(\ve,\delta)$-\emph{spherical} 
if $||\nu-\onu||_2
 \le \ve$ with probability at least $1-\delta$.
\end{definition}

The various regimes of $\mu_G(\cdot)$ are characterized as follows
(where all statements are to hold with respect to the uniform choice of 
$G \in \graph(\alpha,n)$ 
with probability approaching one as $n \to \infty$):
\begin{itemize}
\item[I.] For $\alpha<\alpha_{\rm d}(q)$
the set of proper $q$-colorings forms a unique compact lump:
there is no coexistence. Further, $\mu_G(\,\cdot\,)$ 
is with high probability $(\ve,\delta)$-spherical for any $\ve,\delta>0$.
\item[II.]
For $\alpha_{\rm d}(q)<\alpha<\alpha_{\rm c}(q)$ the measure $\mu_G$ 
exhibits coexistence in the sense of Section
\ref{sec:GeneralDef}. More precisely, there exist $\epsilon>0$, $C>0$
and for each $n$ a partition of the space of configurations $\cX_q^n$ 
into $\cN=\cN_n$ sets $\{\Omega_{\ell,n}\}$ such that for any 
$n$ and $1 \le \ell \le \cN$, 
\begin{eqnarray*}
\frac{\mu_G(\partial_{\epsilon}\Omega_{\ell,n})}
{\mu_G(\Omega_{\ell,n})}\le e^{-C\, n}\, .
\end{eqnarray*}
Furthermore, 
there exists $\Sigma =\Sigma(\alpha)>0$, called 
\emph{complexity} or \emph{configurational entropy}
and a subfamily $\Typ=\Typ_n$ 
of the partition $\{\Omega_{\ell,n}\}_{\ell \in \Typ}$ such that 
$$
\sum_{\ell\in\Typ} \mu_G(\Omega_{\ell,n})\ge 1-e^{-C'n} \,,
$$
for some $C'>0$ independent of $n$ and
$$
e^{-n\Sigma-o(n)}\le \inf_{\ell \in \Typ} \mu_G(\Omega_{\ell,n}) 
\le \sup_{\ell \in \Typ} \mu_G(\Omega_{\ell,n}) \le e^{-n\Sigma+o(n)} 
$$ 
so in particular, $|\Typ_n| = e^{n\Sigma+o(n)}$. 
\item[III.] For $\alpha_{\rm c}(q)<\alpha<\alpha_{\rm s}(q)$ the situation is analogous to 
the last one, but now $\cN_n$ is sub-exponential in $n$. More precisely,
for any $\delta>0$, a fraction $1-\delta$ of the measure $\mu_G$
is comprised of $\cN(\delta)$ elements of the partition, whereby $\cN(\delta)$ converges 
as $n \to \infty$ to a finite random variable. Furthermore, $\mu_G(\,\cdot\,)$ is no longer spherical.
\item[IV.] For $\alpha_{\rm s}(q)<\alpha$ the random graph $G_n$ is, 
with high probability, uncolorable (i.e. non $q$-colorable).
\end{itemize}

Statistical mechanics methods provide semi-explicit expressions
for the threshold 
values $\alpha_{\rm d}(q)$, $\alpha_{\rm c}(q)$ and $\alpha_{\rm s}(q)$
in terms of the solution of a certain identity whose argument 
is a probability measure on the $(q-1)$-dimensional simplex.
%
%
\subsection{The COL-UNCOL transition}
\label{sec:COL-UNCOL}

Though the existence of a colorable-uncolorable transition is not 
yet established, $q$-colorability is a monotone graph property
(i.e. if $G$ is $q$-colorable, so is any subgraph of $G$). As such,
Friedgut's theory \cite{AchlioptasFriedgut,Friedgut}
provides the first step in this direction. Namely, 
\begin{thm}
Suppose the random graph $G_{n,\alpha}$ is uniformly chosen 
from the Erd\"os-Renyi graph ensemble $\graph(\alpha,n)$.
Then, for any $q\ge 3$ there 
exists $\alpha_{\rm s}(q;n)$ such that 
for any $\delta>0$,
\begin{eqnarray} \lim_{n\to\infty}\prob\{G_{n,\alpha_{\rm
s}(q;n)(1-\delta)}\, \mbox{is $q$-colorable}\} = 1\, ,\\
\lim_{n\to\infty}\prob\{G_{n,\alpha_{\rm s}(q;n)(1+\delta)}\, \mbox{is
$q$-colorable}\} = 0\, .  
\end{eqnarray} 
\end{thm} 
We start with a simple upper bound on the COL-UNCOL transition threshold.
\begin{propo}\label{prop:ubd-qcol}
The COL-UNCOL threshold is upper bounded as
\begin{eqnarray}
\alpha_{\rm s}(q;n)\le
\overline{\alpha}_{\rm s}(q) \equiv \frac{\log q}{\log(1-1/q)}\, .
\end{eqnarray}
\end{propo}
\begin{proof} 
A $q$-coloring is a partition of the vertex set $[n]$
into $q$ subsets of sizes $n_x$, $x \in \cX_q$. Given a 
$q$-coloring, the probability that a uniformly chosen
edge has both end-points of the same color is 
$$
\sum_{x \in \cX_q} 
\binom{n_x}{2}/\binom{n}{2} \ge \frac{1}{q} - \frac{2}{n-1} \,.
$$
Consequently, 
choosing first the $q$-coloring and then choosing 
uniformly the $m$ edges to be included in $G=G_{n,\alpha}$ 
we find that the expected number of proper $q$-colorings 
for our graph ensemble is bounded by 
\begin{eqnarray*}
\E\{Z_G\} \le q^{n}\Big(\frac{n+1}{n-1}-\frac{1}{q}\Big)^m\, .
\end{eqnarray*}
Since $\E\{Z_{G}\}\to 0$ for $\alpha>\oalpha_{\rm s}(q)$ our 
thesis follows from Markov's inequality.
\end{proof}

Notice that $\oalpha_{\rm s}(q) = q\log q [1+o(1)]$ as $q\to\infty$.
This asymptotic behavior is known to be tight, for it is 
shown in \cite{AchlioptasNaor} that 
\begin{thm}\label{thm:lbd-qcol}
The COL-UNCOL threshold is lower bounded as
\begin{eqnarray}
\alpha_{\rm s}(q;n)\ge
\underline{\alpha}_{\rm s}(q) \equiv  (q-1)\log(q-1)\, .
\end{eqnarray}
\end{thm}
\begin{proof}[Sketch of proof] 
Let $Z$ denote the number of \emph{balanced}
$q$-colorings, namely $q$-colorings having exactly $n/q$ 
vertices of each color. A computation similar to the one 
we used when proving Proposition \ref{prop:ubd-qcol} yields
the value of $\E Z$. It captures enough of $\E Z_G$
to potentially yield a tight lower bound on $\alpha_{\rm s}(q)$
by the second moment method, namely, using the bound 
$\prob(Z_G>0) \ge \prob(Z>0) \ge (\E Z)^2/\E Z^2$. 
The crux of the matter is of course to control the second
moment of $Z$, for which we defer to \cite{AchlioptasNaor}.
\end{proof}

The proof of Theorem \ref{thm:lbd-qcol} is non-constructive.
In particular, it does not suggest a way of efficiently 
finding a $q$-coloring when $\alpha$ is near $\alpha_{\rm s}(q;n)$
(and as of now, it is not even clear if this is possible).
In contrast, we provide next a simple, `algorithmic' 
(though sub-optimal), lower bound on $\alpha_{\rm s}(q;n)$.
To this end,
recall that the $k$-core of a graph $G$ is the 
largest induced subgraph of $G$ having minimal degree at least $k$.
\begin{propo}
If $G$ does not have a non-empty $q$-core then it is $q$-colorable.
\end{propo}
\begin{proof}
Given a graph $G$ and a vertex $i$, denote by $G\setminus\{i\}$ the
graph obtained by removing vertex $i$ and all edges incident to it.
If $G$ does not contain a $q$-core, then 
we can sequentially remove vertices of degree less than $q$ 
(and the edges incident to them), one at a time, until we have 
decimated the whole graph. This simple `peeling algorithm' 
provides an
ordering $i(1),i(2),\dots$, $i(n)$ of the vertices, such that
setting $G_0=G$ and $G_t= G_{t-1}\setminus \{i(t)\}$,
we have that for any $t\le n$, the degree of $i(t)$ 
in $G_{t-1}$ is smaller than $q$. 
Our thesis follows from the observation that if $G\setminus\{i\}$
is $q$-colorable, and $i$ has degree smaller than $q$, then $G$
is $q$-colorable as well. 
\end{proof}

As mentioned before, this proof outlines an efficient 
algorithm for constructing a $q$-coloring 
for any graph $G$ whose $q$-core is empty, 
and in principle, also for enumerating in this 
case the number of $q$-colorings of $G$.
The threshold for the appearance of a $q$-core in a 
random Erd\"os-Renyi graph chosen uniformly from $\graph(\alpha,n)$ 
was first determined in \cite{PittelEtAl}.
\begin{propo}
Let $h_\alpha (u) = \prob\{\poisson (2\alpha u)\ge q-1\}$,
and define (for $q\ge 3$)
\begin{eqnarray}
\alpha_{\rm core}(q) = \sup\{\alpha \ge 0 :\, h_\alpha (u)\le u
\quad \forall u \in [0,1]\}\, .
\end{eqnarray}
Then, with high probability, a uniformly random graph $G$ 
from $\graph(\alpha,n)$ has a $q$-core if $\alpha>\alpha_{\rm core}(q)$,
and does not have one if $\alpha<\alpha_{\rm core}(q)$. 
\end{propo}
\begin{proof}[Sketch of proof] 
Starting the peeling algorithm at 
such graph $G_0=G_{n,\alpha}$ yields an inhomogeneous 
Markov chain $t \mapsto G_t$ which is well approximated
by a chain of reduced state space $\integers_+^q$ and
smooth transition kernel. The asymptotic behavior of 
such chains is in turn governed by the solution of 
a corresponding ODE, out of which we thus deduce the 
stated asymptotic of the probability that a uniformly 
random graph $G$ from $\graph(\alpha,n)$ has a $q$-core. 
We shall not detail this approach here,
as we do so in Section \ref{sec:ODE} for the
closely related problem of finding the threshold
for the appearance of a $2$-core in a
uniformly random hypergraph.
\end{proof}

We note in passing that the value of $\alpha_{\rm core}(q)$
can be a-priori predicted by the following 
elegant heuristic `cavity' argument. For a vertex $i\in V$
we call `$q$-core induced by $i$' the largest induced subgraph having 
minimum degree at least $q$ except possibly at $i$.
We denote by $u$ the probability that for a uniformly chosen
random edge $(i,j)$, its end-point $i$ belongs to the $q$-core 
induced by $j$. 
Recall that for large $n$ the degree $\Delta$ 
of the uniformly chosen vertex $i$ of $G_{n,\alpha}$, excluding 
the distinguished edge $(i,j)$, is approximately 
a $\poisson(2\alpha)$ random variable. We expect each of 
these $\Delta$ edges to connect $i$ to a vertex from the 
$q$-core induced by $j$ with probability $u$ and 
following the Bethe ansatz, these events should be 
approximately independent of each other. Hence, under 
these assumptions the vertex $i$ is in the $q$-core 
induced by $j$ with probability $h_\alpha(u)$, leading 
to the self-consistency equation $u=h_\alpha(u)$. 
The threshold $\alpha_{\rm core}(q)$ then corresponds to the 
appearance of a positive solution of this equation.

\subsection{Coexistence and clustering: the physicist's approach}
\label{sec:Clust}

For $\alpha<\alpha_{\rm s}(q)$, the measure $\mu_G(\,\cdot\,)$ 
is well defined but can have a highly non-trivial structure,
as discussed in Section \ref{sec:Broad}.
We describe next the physicists conjecture 
for the corresponding threshold $\alpha_{\rm d}(q)$ and 
the associated complexity function $\Sigma(\alpha)$.
For the sake of simplicity, we shall write the explicit formulae 
in case of random $(k+1)$-regular ensembles instead of the
Erd\"os-Renyi ensembles $\graph(\alpha,n)$ we use in our overview. 

\subsubsection{Clustering and reconstruction thresholds: a conjecture}
\label{sec:Clust=Rec}

Following \cite{MezardMontanariRec},
the conjectured value for $\alpha_{\rm d}(q)$ has a particularly
elegant interpretation in terms of a phase transition for a model 
on the rooted Galton-Watson tree $\Tree=\Tree(P,\infty)$ with 
offspring distribution $P=\poisson(2\alpha)$. 
With an abuse of notation, let $\mu$ also denote the 
free boundary 
Gibbs 
measure over proper $q$-colorings of $\Tree$
(recall that every tree is $2$-colorable). More explicitly, a 
proper $q$-coloring $\ux = \{x_i \in \cX_q :i\in \Tree\}$ is sampled from 
$\mu$ as follows. First sample the root color uniformly at random. Then, 
recursively, 
for each colored node $i$, sample the colors of its 
offspring
uniformly at random among the colors that are different from $x_i$.

We denote by $\root$ the root of $\Tree$ and by $\cBall_\root(t)$ 
the set of vertices of $\Tree$ whose distance from the root is at least $t$.
Finally, for any subset of vertices $U$, we let $\mu_U(\,\cdot\,)$
be the marginal law of the corresponding color assignments.

For small $\alpha$ the color at the root de-correlates from colors 
in $\cBall_\root(t)$ when $t$ is large, whereas at large $\alpha$ they 
remain correlated at any distance $t$.
The `reconstruction threshold' separates these two regimes.
\begin{definition}\label{def:Rec}
The reconstruction threshold $\alpha_{\rm r}(q)$ is the maximal value of
$\alpha$ such that
\begin{eqnarray}
\lim_{t\to\infty} \E \{ \, ||\mu_{\root,\cBall_\root(t)}-
\mu_\root \times \mu_{\cBall_\root(t)}||_{\sTV} \, \} = 0
\end{eqnarray}
(where the expectation is over the random tree $\Tree$).
If the limit on the left-hand side is positive, we say that the
reconstruction problem is \emph{solvable}.
\end{definition}

It is conjectured that the coexistence threshold 
$\alpha_{\rm d}(q)$ for locally tree like random graphs 
coincides with the reconstruction threshold 
$\alpha_{\rm r}(q)$ for the corresponding random trees. 
We next 
present a statistical physics argument in favor 
of this conjecture. There are various non-equivalent
versions of this argument, all 
predicting the same location for the threshold.
The argument that we will reproduce was first developed in
\cite{CarParPatSou,FraPar_pot,MezPar_glasses},
to explore the physics of glasses and spin glasses.

Note that the major difficulty 
in trying to identify the existence of `lumps' 
is that we do not know, a priori, where these lumps are in the
space of configurations. However,
if $\uX^*$ is a configuration sampled from $\mu(\,\cdot\,)$, it
will fall inside one such lump so the idea is to study how a second
configuration $\ux$ behaves when tilted towards the first one.
Specifically, 
fix $\ux^*= \{x^*_i \in \cX_q :\, i\in V\}$ and consider the tilted 
measures 
\begin{eqnarray*}
\mu^*_{G,\ux^*,\epsilon} 
(\ux) = \frac{1}{Z_\epsilon} \prod_{(i,j)\in E}\ind(x_i\neq x_j)\prod_{i\in V}
\psi_{\epsilon}(x_i^*,x_i)\, ,
\end{eqnarray*}
where $\psi_{\epsilon}(x,y)$ 
is a tilting function depending continuously on $\epsilon$,   
such that $\psi_0(x,y)=1$ (so 
$\mu^*_0$ reduces to the uniform measure over proper colorings), and 
which favors $x=y$ when $\epsilon>0$. For instance, we might take
\begin{eqnarray*}
\psi_{\epsilon}(x,y) = \exp\Big\{\epsilon \, \ind (x=y)\Big\}\, .
\end{eqnarray*}

While the study of the measure
$\mu^*_{G,\ux^*,\epsilon}$ is beyond our current means, 
we gain valuable insight from examining its Bethe approximation. 
Specifically, in this setting messages depend in addition to the graph also 
on $\ux^*$ and $\epsilon$, and the Bethe equations 
of Definition \ref{def:BetheEq} are 
\begin{eqnarray}\label{eq:Bethe-color}
\nu_{i\to j}(x_i) =  z_{i \to j}^{-1} \psi_{\epsilon}(x_i^*,x_i)\,
\prod_{l\in\di\setminus j}
\big( 1 -\nu_{l\to i}(x_i) \big) \, ,
\end{eqnarray}
with $z_{i \to j}$ a normalization constant.
In shorthand we write this equation as
\begin{eqnarray*}
\nu_{i\to j} = \F_{\epsilon}\{\nu_{l\to i}:\, l\in\di\setminus j\}\, .
\end{eqnarray*}

Let us now assume that $G$ is a regular graph of degree $k+1$ 
and that $\uX^*$ is a uniformly random proper $q$-coloring of $G$. Then,
the message $\nu_{i\to j}$ is itself a random variable, taking values in 
the $(q-1)$-dimensional probability simplex $\M(\cX_q)$. For
each $x\in \cX_q$ we denote by $Q_x$ (which also depends on 
$\epsilon$), the conditional law of
$\nu_{i\to j}$ given that $X^*_i = x$. In formulae, for any 
Borel measurable 
subset $A$ of $\M(\cX_q)$, we have
\begin{eqnarray*}
Q_x(A) \equiv\prob\left\{\nu_{i\to j}(\,\cdot\,)\in A\big|X^*_i = x\right\}\, .
\end{eqnarray*}

Assume that, conditionally on the reference coloring $\uX^*$,
the messages $\nu_{l\to i}$ for 
$l \in \di \setminus j$ are asymptotically independent, 
and have the laws $Q_{X_i^*}$. We then obtain the following recursion for 
$\{Q_x\}$,
\begin{eqnarray*}
Q_x(A) = \sum_{x_1\dots x_k}\mu(x_1,\dots,x_k|x)\int
\ind(\F_{\epsilon}(\nu_1,\dots,\nu_k)\in A)\,\prod_{i=1}^kQ_{x_i}(\de\nu_i)\,,
\end{eqnarray*}
where $(x_1,\ldots,x_k)$ denote the values of $(X_l^*$,
$l \in \di \setminus j)$ and $\mu(x_1,\ldots,x_k|x)$ the 
corresponding conditional marginal of $\mu=\mu^*_0$ 
given $X_i^*=x$. 
Assuming further that for a random regular graph $G=G_{n,k+1}$ 
the measure $\mu(x_1,\ldots,x_k|x)$ converges as $n \to \infty$
to the analogous conditional law for 
the regular $k$-ary tree, we obtain the fixed point equation 
\begin{eqnarray}
\label{eq:fixpt}
Q_x(A) = \frac{1}{(q-1)^k}\sum_{x_1\dots x_k\neq x}\int
\ind(\F_{\epsilon}(\nu_1,\dots,\nu_k)\in A)\,\prod_{i=1}^kQ_{x_i}(\de\nu_i)\,.
\end{eqnarray}
In the limit $\eps = 0$ this equation admits 
a trivial degenerate solution, whereby 
$Q_x =\delta_{\onu}$ is concentrated on one point, 
the uniform vector $\onu(x) = 1/q$ for all $x \in \cX_q$. 
The interpretation of this solution is that, 
as $\eps\downarrow 0$, a random coloring 
from the tilted measure $\mu^*_{G,\uX^*,\eps}$, 
becomes uncorrelated from the reference coloring $\uX^*$. 

It is not hard to verify that this is the only degenerate solution 
(namely, where each measure $Q_x$ is supported on one point),
of (\ref{eq:fixpt}) at $\epsilon=0$. 
A second scenario is however possible. 
It might be that, as $\eps\downarrow 0$
(and, in particular, for $\eps=0$), Eq.~(\ref{eq:fixpt})
admits also a non-trivial solution, whereby at least one
of the measures $Q_x$ is not supported on the uniform 
vector $\onu$. 
This is interpreted by physicists as implying coexistence: 
the coloring sampled
from the tilted measure $\mu^*_{G,\uX^*,\eps}$ remains trapped
in the same `state' (i.e. in the same subset of configurations
$\Omega_{\ell,n}$), as $\uX^*$. 

Let us summarize the statistical physics conjecture:
the uniform measure $\mu_G(\,\cdot\, )$
over proper $q$-colorings of a random $(k+1)$-regular graph 
exhibits coexistence if and only if Eq.~(\ref{eq:fixpt})
admits a non-trivial solution for $\eps =0$.
In the next subsection we show that this happens if and only
if $k\ge k_{\rm r}(q)$, with $k_{\rm r}(q)$ the reconstructibility
threshold on $k$-ary trees (which is defined analogously to
the Poisson tree threshold $\alpha_{\rm r}(q)$, see Definition 
\ref{def:Rec}).

\subsubsection{The reconstruction threshold for $k$-ary trees}

We say that a probability measure on $\M(\cX_q)$ is 
\emph{color-symmetric} if it is invariant under the action of
color permutations on its argument $\nu \in \M(\cX_q)$. 
Following \cite[Proposition 1]{MezardMontanariRec} we proceed to 
show that the existence of certain non-trivial solutions 
$\{Q_x\}$ of (\ref{eq:fixpt}) at $\epsilon =0$
is equivalent to solvability of the 
corresponding reconstruction problem for $k$-ary trees. 
\begin{propo}\label{prob:k-tree-reco}
The reconstruction problem is solvable on $k$-ary trees 
if and only if Eq.~(\ref{eq:fixpt}) admits 
at $\epsilon=0$ a solution $\{Q_x, x \in \cX_q\}$ such that 
each $Q_x$ has the Radon-Nikodym density $q \nu(x)$ 
with respect to the same color-symmetric, non-degenerate 
probability measure $Q$.  
\end{propo}
\begin{proof} 
First notice that $\{Q_x, x \in \cX_q\}$ is a 
solution of (\ref{eq:fixpt}) at $\epsilon=0$ if and only if 
for any $x \in \cX_q$ and bounded Borel function $g$, 
\begin{eqnarray}\label{eq:conditionalRec}
\int g(\nu) q^{-1} Q_x(\de\nu) = c_{q,k} \int g(\F_0(\nu_1,\dots,\nu_k))
\, \prod_{i=1}^k [Q_* - q^{-1} Q_x] (\de\nu_i)\,,
\end{eqnarray}
where $Q_* = q^{-1} \sum_{x=1}^q Q_x$ and $c_{q,k}=q^{k-1}(q-1)^{-k}$.
If this solution is of the stated form, then $Q_*=Q$ and upon   
plugging $Q_x(\de\nu)=q \nu(x) Q(\de\nu)$ in the 
identity (\ref{eq:conditionalRec}) we see that for 
any bounded Borel function $h$,
\begin{eqnarray}
\int h(\nu) Q(\de \nu) = 
\int\; \left[\frac{z(\nu_1,\dots,\nu_k)}{z(\onu,\dots,\onu)}\right]\;
h( \F_0(\nu_1,\dots,\nu_k) )\,\prod_{i=1}^k Q(\de\nu_i)\,,
\label{eq:UnconditionalRec}
\end{eqnarray}
where $z(\nu_1,\dots,\nu_k) = \sum_{x=1}^q\prod_{i=1}^k(1-\nu_i(x))$
is the normalization constant of the mapping $\F_0(\cdot)$
(so $c_{q,k}=1/z(\onu,\dots,\onu)$). 
Conversely, for any color-symmetric probability measure $Q$ on 
$\M(\cX_q)$ the value of $\int \nu(x) Q(\de\nu)$ is independent of
$x \in \cX_q$, hence $Q_x(\de\nu)=q \nu(x) Q(\de\nu)$ are then 
also probability measures on $\M(\cX_q)$ and such that $Q=Q_*$. 
Further, recall that for any $x \in \cX_q$ and $\nu_i \in \M(\cX_q)$,
$$
z(\nu_1,\dots,\nu_k)\F_0(\nu_1,\dots,\nu_k)(x) = \prod_{i=1}^k (1-\nu_i(x))\,,
$$
so if such $Q$ satisfies (\ref{eq:UnconditionalRec}), then 
considering there $h(\nu)=g(\nu)\nu(x)$ leads to 
$\{Q_x\}$ satisfying (\ref{eq:conditionalRec}). 

If a solution $Q$ of (\ref{eq:UnconditionalRec}) is degenerate, 
i.e. supported on one point $\nu$, then 
$\nu=\F_0(\nu,\dots,\nu)$, hence $\nu=\onu$. That is, any
non-trivial solution $Q \neq \delta_{\onu}$ is also non-degenerate.
We thus proceed to show that solvability 
of the reconstruction problem on $k$-ary trees is equivalent
to having color-symmetric solution $Q \neq \delta_{\onu}$ of 
(\ref{eq:UnconditionalRec}). To this end, consider a 
proper $q$-coloring $X=\{ X_v : v \in \Tree\}$ 
of the $k$-ary tree, sampled at
random according to the free boundary Gibbs measure $\mu$.
Let $\nu^{(t)}$ denote the marginal distribution of the root color given 
the colors at generation $t$. In formulae, this is the 
$\M(\cX_q)$-valued random variable
such that for $x\in\{1,\dots,q\}$,
\begin{eqnarray*}
\nu^{(t)}(x) = 
\mu_{\root|\cBall_\root(t)}(x|X_{\cBall_\root(t)})
= \prob\{X_\root=x|X_{\cBall_{\root}(t)}\}\, .
\end{eqnarray*}
Denote by $Q^{(t)}_x$ the conditional law of $\nu^{(t)}$ given the 
root value $X_\root=x$. 
The $k$-ary tree of $(t+1)$ generations is
the merging at the root of $k$ disjoint $k$-ary trees, 
each of which has $t$ generations. Thus, conditioning 
on the colors $x_1,\ldots,x_k$ of the root's offspring, 
one finds 
that the probability measures $\{Q^{(t)}_x\}$ satisfy for any 
$x$ and any bounded Borel function $h(\cdot)$ the recursion
$$
\int h(\nu) Q^{(t+1)}_x(\de\nu) = \frac{1}{(q-1)^k}\sum_{x_1,\dots,x_k\neq x}
\int h(\F_0 (\nu_1,\dots,\nu_k))\,
\prod_{i=1}^k Q_{x_i}^{(t)}(\de \nu_i)\, ,
$$
starting at $Q^{(0)}_x=\delta_{\nu_x}$,
where $\nu_x$ denotes the probability vector that puts
weight one on the color $x$. 

Let $Q^{(t)}$ denote the unconditional law of $\nu^{(t)}$. That is,
$Q^{(t)}=q^{-1} \sum_{x=1}^q Q_x^{(t)}$. By the tower 
property of the conditional expectation, for any 
$x \in \cX_q$ and bounded measurable function $h$ on $\M(\cX_q)$,
\begin{align*}
 \int h(\nu) Q^{(t)}_x (\de\nu) &= q \E [ h(\nu^{(t)}) \ind(X_\root=x) ]
\\ &= q \E[ h(\nu^{(t)}) \nu^{(t)} (x) ] 
= q \int \nu(x) h(\nu) Q^{(t)} (\de\nu) \,.
\end{align*}
Consequently, $Q^{(t)}_x$
has the 
Radon-Nikodym derivative
$q \nu(x)$ with respect to $Q^{(t)}$. Plugging this into the recursion 
for $Q^{(t)}_x$ we find that $Q^{(t)}$ satisfies the 
recursion relation
\begin{equation}\label{eq:recur2}
\int h(\nu) Q^{(t+1)}(\de \nu) = \int
\; \left[\frac{z(\nu_1,\dots,\nu_k)}{z(\onu,\dots,\onu)}\right]\;
h( \F_0(\nu_1,\dots,\nu_k) )\,\prod_{i=1}^k Q^{(t)} (\de\nu_i)\,,
\end{equation}
starting at $Q^{(0)}=q^{-1} \sum_{x=1}^q \delta_{\nu_x}$.  

Note that for each $x \in \cX_q$, 
the sequence $\{\nu^{(t)} (x) \}$ is a reversed martingale 
with respect to the filtration 
${\mathcal F}_{-t} =\sigma(X_{\cBall_{\root}(t)})$, $t \ge 0$,
hence by L\'evy's downward theorem, it has an almost sure limit. 
Consequently, the probability measures $\{Q^{(t)}\}$ 
converge weakly to a limit $Q^{(\infty)}$. 

As $Q^{(0)}$ is color-symmetric and the recursion (\ref{eq:recur2})
transfers the color-symmetry of $Q^{(t)}$ to that of $Q^{(t+1)}$, we
deduce that $Q^{(\infty)}$ is also color-symmetric. 
Further, with the function $\F_0 :\M(\cX_q)^k \to \M(\cX_q)$ continuous 
at any point $(\nu_1,\dots,\nu_k)$ for which 
$z(\nu_1,\dots,\nu_k) >0$, it follows from the recursion 
(\ref{eq:recur2}) that $Q^{(\infty)}$ satisfies 
(\ref{eq:UnconditionalRec}) for any continuous $h$,
hence for any bounded Borel function $h$. By definition,
\begin{eqnarray*}
|| \mu_{\root,\cBall_\root(t)}- \mu_\root \times
\mu_{\cBall_\root(t)}||_{\sTV} = 
\int ||\nu-\onu||_{\sTV}\, Q^{(t)} (\de\nu)\,,
\end{eqnarray*}
and with $\cX_q$ finite, the function $\nu \mapsto ||\nu-\onu||_{\sTV}$ is 
continuous. Hence, the reconstruction problem is solvable if and 
only if $Q^{(\infty)} \neq \delta_{\onu}$. That is, as claimed,
solvability implies the existence of a non-trivial color-symmetric 
solution $Q^{(\infty)}$ of  (\ref{eq:UnconditionalRec}).

To prove the converse assume there exists a color-symmetric solution 
$Q \neq \delta_{\onu}$ of Eq.~(\ref{eq:UnconditionalRec}).  Recall 
that in this case $Q_x(\de\nu)=q \nu(x) Q(\de\nu)$ are probability measures
such that $Q=q^{-1} \sum_{x=1}^q Q_x$. Further, if a random variable $Y(t)$ is 
conditionally independent of $X_\root$ given $X_{\cBall_\root(t)}$ then  
\begin{align*}
|| \mu_{\root,Y(t)} - \mu_{\root} \times \mu_{Y(t)}||_{\sTV}
&\le 
|| \mu_{\root,Y(t),\cBall_{\root(t)}} - 
\mu_{\root} \times \mu_{Y(t),\cBall_{\root(t)}}||_{\sTV}
\\
&=
|| \mu_{\root,\cBall_{\root(t)}} - 
\mu_{\root} \times \mu_{\cBall_{\root(t)}}||_{\sTV} 
\end{align*}
(where $\mu_{\root,Y(t)}$ denotes the joint law of $X_\root$ and $Y(t)$).
Turning to construct such a random variable $Y(t) \in \M(\cX_q)$, 
let $\dBall_{\root}(t)$ denote the vertices of the tree at 
distance $t$ from $\root$ and set $\nu_i \in \M(\cX_q)$ 
for $i\in \dBall_{\root}(t)\}$ to be 
conditionally independent given $X_{\cBall_\root(t)}$, 
with $\nu_i$ distributed according to the random 
measure $Q_{X_i}(\,\cdot\,)$. Then, define recursively 
$\nu_v \equiv \F_0(\nu_{u_1},\ldots,\nu_{u_k})$
for $v \in \dBall_{\root}(s)$, $s=t-1,t-2,\ldots,0$,
where $u_1,\ldots,u_k$ denote the offspring of $v$ in $\Tree$. 
Finally, set $Y(t)=\nu_\root$. 

Under this construction, the law 
$P_{v,x}$ of $\nu_v$ conditional upon $X_v=x$
is $Q_{X_v}$, for any $v \in \dBall_{\root}(s)$, $s = t,\ldots, 0$.
Indeed, clearly this is the case for $s=t$ and proceeding recursively, 
assume it applies at levels $t,\ldots,s+1$. Then, as $\{Q_x, x \in \cX_q\}$ 
satisfy (\ref{eq:conditionalRec}), we see that 
for $v \in \dBall_{\root}(s)$ of offspring $u_1,\ldots,u_k$, 
any $x \in \cX_q$ and bounded Borel function $g(\cdot)$, 
\begin{align*} 
\int g(\nu) P_{v,x}(\de\nu) &=
\E [ \int g(\F_0 (\nu_1,\dots,\nu_k))\,
\prod_{i=1}^k Q_{X_{u_i}} (\de \nu_i) | X_v=x ] \\
&= q c_{q,k} \int g(\F_0 (\nu_1,\dots,\nu_k))\,
\prod_{i=1}^k [Q- q^{-1} Q_x] (\de \nu_i)  
\\
&= \int g(\nu) Q_x(\de\nu) \,.
\end{align*}
That is, $P_{v,x}=Q_x$, as claimed.
In particular, $\mu_{Y(t)|\root}=Q_{X_\root}$,
$\mu_{Y(t)}=Q$ and with $Q_x(\de \nu)=q\nu(x) Q(\de \nu)$,
it follows that 
\begin{eqnarray*}
|| \mu_{\root,Y(t)}- \mu_\root \times
\mu_{Y(t)}||_{\sTV} = 
\frac{1}{q}\sum_{x=1}^q ||Q_x-Q||_{\sTV} 
= \int ||\nu-\onu||_{\sTV}\, Q(\de\nu) \, ,
\end{eqnarray*}
which is independent of $t$ and strictly positive (since 
$Q \ne \delta_{\onu}$). By the preceding inequality, this 
is a sufficient condition for reconstructibility.
\end{proof}

\subsubsection{Complexity: exponential growth of the number of clusters}

We provide next a heuristic derivation of the 
predicted value of the complexity parameter $\Sigma=\Sigma(k)$ 
for proper $q$-colorings of
a uniformly chosen random regular graph $G=G_{n,k+1}$, 
as defined in Section \ref{sec:Broad}, regime II,
namely, when $k_d(q) < k < k_c(q)$.
This parameter is interpreted as  the exponential growth rate 
of the number of `typical' lumps or `clusters' to which the
uniform measure $\mu_G(\, \cdot\,)$ decomposes.
Remarkably, we obtain an expression for $\Sigma(k)$ 
in terms of the non-degenerate solution of (\ref{eq:fixpt})
at $\epsilon=0$.

Recall Definition \ref{def:BetheFree} that the Bethe free entropy 
for proper $q$-colorings of $G$ and a given (permissive)
message set $\{\nu_{i\to j}\}$ is 
\begin{align}
\Phi\{\nu_{i\to j}\} = &-\sum_{(i,j)\in E} \log\big\{ 1 -
\sum_{x=1}^q \nu_{i\to j}(x)\nu_{j\to i}(x)\big\}
\nonumber\\
&+ \sum_{i\in V}\log\Big\{
\sum_{x=1}^q \prod_{j\in\di}
\big( 1 - \nu_{j\to i}(x) \big) \Big\} \, .
\end{align}
According to the Bethe-Peierls approximation, the logarithm of the
number $Z_n$ of proper $q$-colorings for $G=G_{n,k+1}$ is  
approximated for large $n$ by the value of $\Phi\{\nu_{i\to j}\}$
for a message set $\{\nu_{i\to j}\}$ which solves 
the Bethe-Peierls equations (\ref{eq:Bethe-color}) at $\epsilon=0$. 
One trivial solution of these equations is $\nu_{i\to j}=\onu$
(the uniform distribution over $\{1,\dots,q\}$), and 
for $G=G_{n,k+1}$ the corresponding Bethe free entropy is 
\begin{align}
\Phi(\onu) & =  n\Big\{-\frac{k+1}{2}\log\big\{
1 - \sum_{x=1}^q \onu(x)^2 \big\}+
\log\big\{\sum_{x=1}^q (1- \onu(x))^{k+1} \big\}\Big\} 
\nonumber \\
& =  n [ \log q + \frac{k+1}{2} \log (1-1/q) ] \, .
\end{align}

As explained before, when $k_{\rm d}(q) < k < k_{\rm s}(q)$, upon 
fixing $n$ large enough, a regular graph $G_n$ of degree $k+1$ and 
a reference proper $q$-coloring $\ux^*$ of its vertices, 
we expect Eq.~(\ref{eq:Bethe-color}) 
to admit a second solution $\{\nu^*_{i \to j}\}$ for all 
$\epsilon>0$ small enough. In the limit $\epsilon \downarrow 0$,
this solution is conjectured to describe the uniform measure 
over proper $q$-colorings in the cluster $\Omega_{\ell,n}$ 
containing $\ux^*$. In other words, the restricted measure
\begin{eqnarray}
\mu_{\ell,n}(\ux) = \mu_{G_n} (\ux|\Omega_{\ell,n}) = 
\frac{1}{Z_{\ell,n}}\, \prod_{(i,j)\in E} \ind(x_i\neq x_j)\, 
\ind(\ux\in\Omega_{\ell,n})\, ,
\end{eqnarray}
is conjectured to be Bethe approximated by such message set 
$\{\nu^*_{i \to j}\}$. One naturally expects the corresponding
free entropy approximation to hold as well. That is, to have
$$ 
\log Z_{\ell,n} = \Phi\{\nu^*_{i \to j}\} + o(n) \,.
$$
As discussed in Section \ref{sec:Broad}, in regime II,
namely, for $k_{\rm d}(q) < k < k_{\rm c}(q)$,
it is conjectured that for uniformly chosen 
proper $q$-coloring $\uX^*$, the value of 
$n^{-1} \log Z_{\ell,n}$ 
(for the cluster $\Omega_{\ell,n}$ 
containing $\uX^*$), concentrates in probability 
as $n \to \infty$, around a non-random value.
Recall that $\log Z_n = \Phi(\onu) + o(n)$, so with 
most of the $Z_n$ proper $q$-colorings of $G_n$ 
comprised within the $e^{n\Sigma+o(n)}$ 'typical' 
clusters $\Omega_{\ell,n}$, $\ell \in \Typ_n$, each having  
$Z_{\ell,n}$ proper $q$-colorings, we conclude that  
\begin{align}
\Phi(\onu) = \log Z_n +o(n) &= 
\log\Big\{\sum_{\ell=1}^{|\Typ_n|} Z_{\ell,n}
\Big\}+o(n)\nonumber\\
&= n \Sigma +\E [\Phi\{\nu^*_{i\to j}\}]+o(n)\, ,
\end{align}
where the latter expectation is with respect to both the random
graph $G_n$ and the reference configuration $\uX^*$ 
(which together determine the message set $\{\nu^*_{i\to j}\}$).

This argument provides a way to compute
the exponential growth rate $\Sigma(k)$ of the number of clusters, as
\begin{eqnarray*}
\Sigma(k) = \lim_{n\to\infty}
n^{-1}\Big\{ \Phi(\overline{\nu})-\E [\Phi\{\nu^*_{i\to j}\}] 
\Big\}\, .
\end{eqnarray*}
For a uniformly chosen random proper $q$-coloring $\uX^*$, the   
distribution of $\{\nu^*_{i \to j}\}$ can be expressed in 
the $n \to \infty$ limit in terms of the corresponding 
solution $\{Q_x\}$ of the fixed 
point equation (\ref{eq:fixpt}) at $\epsilon=0$. 
Specifically, following the Bethe ansatz, we
expect that for uniformly chosen $i \in [n]$, 
the law of $\{\nu^*_{j\to i}, j \in \di\}$ 
conditional on $\{ X^*_i, X^*_j, j \in \di\}$ 
converges as $n \to \infty$ to 
the product measure $\prod_{j=1}^{k+1} Q_{X^*_j}$ and
the law of $\{\nu^*_{i\to j}, \nu^*_{j \to i}\}$ 
conditional on $X^*_i$ and $X^*_j$ converges 
to the product measure $Q_{X^*_i} \times Q_{X^*_j}$. 
By the invariance of the uniform measure over 
proper $q$-colorings to permutations of the colors,
for any edge $(i,j)$ of $G_{n,k+1}$, the
pair $(X^*_i,X^*_j)$ is uniformly distributed over
the $q(q-1)$ choices of $x_i \ne x_j$ in $\cX_q^2$. Moreover, 
for large $n$ the Bethe approximation predicts that 
$\{ X^*_i, X^*_j, j \in \di\}$ 
is nearly uniformly distributed over the $q(q-1)^{k+1}$ 
choices of $x_j \in \cX_q$, all of which 
are different from $x_i \in \cX_q$. We thus conclude that 
\begin{align}
\Sigma(k) &= -\frac{k+1}{2}\frac{1}{q(q-1)}\sum_{x_1\neq x_2}
\int W_{\rm e}(\nu_1,\nu_2)\, Q_{x_1}(\de\nu_1)\, Q_{x_2}(\de\nu_2)
\nonumber\\
&+\frac{1}{q(q-1)^{k+1}} \sum_{x=1}^q \sum_{x_j \neq x}
\int W_{\rm v}(\nu_1,\dots,\nu_{k+1})\, \prod_{j=1}^{k+1}
Q_{x_j}(\de\nu_j)\, ,
\end{align}
where
\begin{align}
W_{\rm e}(\nu_1,\nu_2) &=  
\log \Big\{ \frac{1-\sum_{x=1}^q \nu_1(x) \nu_2(x)}{1-1/q}\Big\} \, ,\\
W_{\rm v}(\nu_1,\dots,\nu_{k+1}) &= 
\log \Big\{ \frac{1}{q} \sum_{x=1}^q \prod_{j=1}^{k+1} 
\frac{1-\nu_j(x)}{1-1/q} \Big\} \, .
\end{align}
%
%
\section{Reconstruction and extremality}
\label{ch:Reco}  
\setcounter{equation}{0}

As shown in Section \ref{sec:BetheFormal}, the Bethe-Peierls approximation
applies for permissive graph-specification pairs $(G,\upsi)$ such that:
\begin{enumerate}
\item[(a).] The graph $G=(V,E)$ has large girth (and it often suffices
for $G$ to merely have a 
large girth in the neighborhood of most vertices).
\item[(b).] The dependence between the random vectors $\ux_A$ and
$\ux_B$ is weak for subsets $A$ and $B$ which are far apart on $G$ 
(indeed, we argued there that `extremality' is the appropriate 
notion for this property).
\end{enumerate}
While these conditions suffice for Bethe-Peierls
approximation to hold on general graphs with bounded degree,
one wishes to verify them for specific models on sparse 
random graphs. For condition (a) this can be done by
standard random graph techniques (c.f. Section \ref{sec:loc-conv}),
but checking condition (b) is quite an intricate task. Thus,
largely based on \cite{GerschenMon1}, we explore here the 
extremality condition in the context of random sparse graphs. 

Beyond the relevance of extremality for the 
Bethe-Peierls approximation, it is interesting \emph{per se} 
and can be rephrased in terms of the \emph{reconstruction problem}. 
In Section \ref{sec:Clust=Rec} 
we considered the latter in case of proper 
$q$-colorings, where it amounts to estimating 
the color of a distinguished (root) vertex $\root\in V$ 
for a uniformly chosen proper coloring $\uX$ of the given 
graph $G=(V,E)$, when the colors $\{X_j, j\in U\}$ on a
subset $U$ of vertices are revealed. In particular, 
we want to understand whether revealing the colors 
at large distance $t$ from the root, induces a 
non-negligible bias on the distribution of $X_\root$.

It turns out that, for a random Erd\"os-Renyi graph 
chosen uniformly from the ensemble $\graph(\alpha,n)$, 
there exists a critical value $\alpha_{\rm r}(q)$, 
such that reconstruction is possible (in the sense of 
Definition \ref{def:Rec}), when
the number of edges per vertex $\alpha>\alpha_{\rm r}(q)$,
and impossible when $\alpha<\alpha_{\rm r}(q)$. Recall from 
Section \ref{sec:Clust}, that the reconstruction threshold 
$\alpha_{\rm r}(q)$ is conjectured to coincide with the 
so-called `clustering' threshold $\alpha_{\rm d}(q)$. 
That is, the uniform measure over proper $q$-colorings 
of these random graphs should exhibit co-existence 
if and only if $\alpha_{\rm r}(q) = \alpha_{\rm d}(q)\le 
\alpha< \alpha_{\rm c}(q)$. As we will show, this relation 
provides a precise determination of the clustering threshold.

More generally, consider a graph-specification pair 
$(G,\upsi)$, with a distinguished marked vertex $\root\in V$ 
(which we call hereafter the `root' of $G$), and a
sample $\uX$ from the associated  
graphical model $\mu_{G,\upsi}(\ux)$ of (\ref{eq:Canonical}).
The reconstructibility question asks whether `far away' 
variables $\uX_{\cBall_{\root}(t)}$
provide non-negligible information about $X_\root$
(here  
$\cBall_{\root}(t)$ denotes the subset of vertices $i \in V$
at distance $d(\root,i)\ge t$ from the root).
This is quantified by the following definition, where
as usual, for $U \subseteq V$ we  
denote the corresponding marginal distribution
of $\uX_U=\{X_j:\, j\in U\}$ by $\mu_U(\ux_U)$.
\begin{definition}\label{defn:recon}
The reconstruction problem is $(t,\ve)$-\emph{solvable} 
(also called, $(t,\ve)$-\emph{reconstructible}), 
for the graphical model associated with $(G,\upsi)$ 
and rooted at $\root\in V$, if  
\begin{align}
\Vert \mu_{\root,\cBall_{\root}(t)}-
\mu_{\root}\times \mu_{\cBall_{\root}(t)}\Vert_{\sTV} 
\ge\ve\, .
\label{eq:GraphReconstruction}
\end{align}
We say that the reconstruction problem is \emph{solvable} 
(\emph{reconstructible}), 
for a given sequence $\{G_n\}$ of random graphs (and specified 
joint distributions of the graph 
$G_n$, the specification $\upsi$ on it,
and the choice of $\root \in V_n$), 
if for some $\ve>0$ and all $t \ge 0$, the events $A_n(t)$ that the 
reconstruction problem is $(t,\ve)$-solvable on $G_n$, 
occur with positive probability. That is, when
$\inf_t \limsup_{n \to \infty} \prob\{A_n(t)\} > 0$.
\end{definition}
\begin{remark} The inequality (\ref{eq:GraphReconstruction}) 
fails when the connected component of $\root$ in $G$,  
has diameter less than $t$. Hence, for sparse random graphs $G_n$, 
the sequence $n \mapsto \prob\{A_n(t)\}$ is often bounded 
away from one (on account of $\root$ possibly being in a small 
connected component). 
\end{remark}

The rationale for this definition is that 
the total variation distance on the left hand side 
of Eq.~(\ref{eq:GraphReconstruction}) 
measures the information 
about $X_{\root}$
that the variables in $\cBall_{\root}(t)$ provide.
For instance, it is proportional to the difference 
between the probability
of correctly guessing $X_{\root}$ when 
knowing $X_{\cBall_{\root}(t)}$, 
and the a-priori probability of doing so without 
knowing $X_{\cBall_{\root}(t)}$.

Note that non-reconstructibility is slightly 
weaker than the extremality condition
of Section \ref{sec:BetheFormal}. Indeed, we 
require here a decay of the correlations 
between a vertex $\root$ and an arbitrary
subset of vertices at distance $t$ from it, 
whereas in Definition \ref{def:Extremal}, 
we require such decay for \emph{arbitrary} 
subsets of vertices $A$ and $B$ of distance $t$
apart. However, it is not hard to check that 
when proving Theorem \ref{thm:Bethe}
we only consider the extremality condition 
in cases where the size of the subset $B$ 
does not grow with $R$ (or with the size of $G$)
and for graph sequences that converge 
locally to trees, this is in turn 
implied by a non-reconstructibility type
condition, where $B=\{\root\}$ is a single vertex.

Recall Section \ref{sec:loc-conv} that for a uniformly chosen root 
$\root$ and locally tree-like sparse random graphs 
$G_n$, for any $t \ge 0$ fixed, the finite neighborhood 
$\Ball_\root(t)$ converges in distribution to a (typically
random) tree. We expect that with high probability the vertices 
on the corresponding boundary set
$\dBall_{\root}(t)=\{ i \in \Ball_\root(t) : 
\di \not\subseteq \Ball_\root(t)\}$, are `far apart' 
from each other in the complementary subgraph 
$\cBall_{\root}(t)$. This suggests that for the graphical model on 
$\cBall_\root(t)$, the variables $\{X_j, j \in \dBall_\root(t)\}$ 
are then weakly dependent, and so approximating 
$G_n$ by its limiting tree structure 
might be a good way to resolve the reconstruction problem.
In other words, one should expect reconstructibility on $G_n$ 
to be determined by reconstructibility on the associated 
limiting random tree. 

Beware that the preceding argument is circular, 
for we assumed that variables on `far apart' vertices 
(with respect to the residual graph $\cBall_{\root}(t)$), 
are weakly dependent, in order to deduce the same 
for variables on vertices that are `far apart' in $G_n$.
Indeed, its conclusion fails for many
graphical models. For example, \cite{GerschenMon1} shows
that the tree and graph reconstruction thresholds 
do not coincide in the simplest example one can think of, 
namely, ferromagnetic Ising models.

On the positive side, we show in the sequel that 
the tree and graph reconstruction problems 
are equivalent under the sphericity condition
of Definition \ref{def:Spherical}
(we phrased this definition in terms proper colorings,
but it applies verbatim to general graphical models).
More precisely, if for any $\ve,\delta>0$,
the canonical measure $\mu(\,\cdot\,)$ is $(\ve,\delta)$-spherical 
with high probability (with respect to the graph distribution), 
then the graph and tree reconstructions do coincide. 
It can indeed be shown that, under the sphericity condition, 
sampling $\uX$ according to the graphical model 
on the residual graph $\cBall_{\root}(t)$,
results with $\{X_j, j \in \dBall_\root(t)\}$ 
which are approximately independent. 

This sufficient condition was applied in \cite{GerschenMon1}
to the Ising spin glass (where sphericity
can be shown to hold as a consequence of a recent 
result by Guerra and Toninelli \cite{Guerra}).
More recently, \cite{MonResTet} deals with 
proper colorings of random graphs
(building on the work of Achlioptas and Naor, in \cite{AchlioptasNaor}). 
For a family of graphical models parametrized by their
average degree, it is natural to expect reconstructibility to 
hold at large average degrees (as the graph is `more connected'),
but not at small average degrees (since the graph `falls' apart into 
disconnected components). We are indeed able to 
establish a threshold behavior (i.e. a critical degree value above 
which reconstruction is solvable) both for spin glasses 
and for proper colorings.
%
%
\subsection{Applications and related work}

Beyond its relation with the Bethe-Peierls 
approximation, the reconstruction problem 
is connected to a number of other 
interesting problems, two of which we 
briefly survey next. 

\vspace{0.05cm}

{\bf Markov Chain Monte Carlo} (MCMC) algorithms provide a well established
way of approximating marginals of the distribution 
$\mu=\mu_{G,\upsi}$ of (\ref{eq:GeneralGraphModel}). The 
idea is to define an (irreducible and
aperiodic) Markov chain whose unique stationary distribution is
$\mu(\cdot)$, so if this chain 
converges rapidly to its stationary state (i.e., 
its mixing time is small), then it can be effectively 
used to generate a sample $\uX$ from $\mu(\cdot)$.

In many interesting cases, the chain is reversible and consists of
local updates (i.e. consecutive states differ in only few  
variables, with transition probabilities determined by 
the restriction of the state to a neighborhood in $G$ of 
the latter set). Under these conditions, the mixing time is 
known to be related to the correlation decay properties
of the stationary distribution $\mu(\,\cdot\,)$ 
(see, \cite{StaticDynamics2,StaticDynamics1}). 
With 
\begin{eqnarray}
\Delta(t;\ux) \equiv 
\Vert \mu_{\root|\cBall_{\root}(t)}(\,\cdot\, |\ux_{\cBall_{\root}(t)})
- \mu_{\root}(\,\cdot\,) \Vert_{\sTV}\, ,
\label{eq:UniformDecay}
\end{eqnarray}
one usually requires in this context
that the dependence between 
$x_{\root}$ and $\ux_{\cBall_{\root}(t)}$ 
decays uniformly, i.e. 
$\sup_{\ux}\Delta(t;\ux) \to 0$ as $t\to\infty$. 
On graphs with sub-exponential growth, a fast enough (uniform) 
decay is necessary and sufficient for fast mixing. However, 
for more general graphs, this \emph{uniform} decay is 
often a too strong requirement, which one might opt 
to replace by the weaker assumption of 
non-reconstructibility (indeed, the 
inequality (\ref{eq:GraphReconstruction}) 
can be re-written as $\E [\Delta(t;\uX)] \ge \ve$, where  
the expectation is with respect to the random sample $\uX$).

In this direction, it was shown in \cite{BergerEtAl} 
that non-reconstructibility is a necessary condition 
for fast mixing. Though the converse may in general fail,
non-reconstructibility is sufficient 
for rapid decay of the variance of local functions (which 
in physics is often regarded as the criterion for fast dynamics, 
see \cite{MonSem}). Further, for certain graphical 
models on trees, \cite{BergerEtAl} shows that non-reconstructibility  
is equivalent to polynomial spectral gap, a result
that is sharpened in \cite{Martinelli} to the equivalence 
between non-reconstructibility and fast mixing 
(for these models on trees). 

\vspace{0.05cm}

{\bf Random constraint satisfaction problems.} 
Given an instance of a constraint satisfaction problem (CSP),
consider the uniform distribution over its solutions. 
As we have seen in Section \ref{sec:random-CSP}, it
takes the form (\ref{eq:FCanonical}), which is an 
immediate generalization of (\ref{eq:GeneralGraphModel}).

Computing the marginal $\mu_{\root}(x_{\root})$ is useful both for
finding a solution and the number of solutions of such a CSP.
Suppose we can generate \emph{only one} uniformly random solution 
$\uX$. In general this is not enough for approximating the
law of $X_\root$ in a meaningful way, but one can try the following: 
First, fix all variables 
`far from $\root$' to take the same value as in the sampled configuration,
namely $X_{\cBall_{\root}(t)}$. Then, compute the conditional distribution 
at $\root$ (which for locally tree-like graphs can be  
done efficiently via dynamic programming). While the resulting 
distribution is in general not a good approximation of 
$\mu_\root(\cdot)$, non-reconstructibility implies that it is, 
with high probability within total variation distance $\ve$ 
of $\mu_\root(\cdot)$. That is, non-reconstructibility yields a
good approximation of $\mu_{\root}(x_{\root})$ based on 
a single sample (namely, a single uniformly random solution $\uX$).
The situation is even simpler under the assumptions of 
our main theorem (Theorem \ref{thm:graph<->tree}), where 
the boundary condition $\uX_{\cBall_{\root}(t)}$ may
be replaced by an i.i.d. uniform boundary condition.

We have explained in Section \ref{chap:Coloring} why 
for a typical sparse random graph of large average degree
one should expect the set of proper colorings 
to form well-separated `clusters'. 
The same rationale should apply, at high constraint density, 
for the solutions of a typical instance of a CSP based on 
large, sparse random graphs 
(c.f. \cite{AchlioRicci,Mora,MarcGiorgioRiccardo}). This in
turn increases the computational 
complexity of sampling even one uniformly random solution.

Suppose the set of solutions partitions into clusters 
and any two solutions that differ on at most $n\ve$ 
vertices, are in the same cluster. Then, knowing the value of all 
`far away' variables $\uX_{\cBall_{\root}(t)}$ 
determines the cluster to which the sample $\uX$ belongs,
which in turn provides some information on $X_{\root}$. 
The preceding heuristic argument connects reconstructibility 
to the appearance of well-separated solution clusters, a 
connection that has been studied for example in 
\cite{OurPNAS,MezardMontanari}.

\vspace{0.05cm}

Reconstruction problems also emerge in a variety of other contexts:
$(i)$ Phylogeny (where given some evolved genomes,
one aims at reconstructing the genome of their common ancestor,
c.f. \cite{Phylogeny}); 
$(ii)$ Network tomography 
(where given end-to-end delays in a computer network,
one aims to infer the link delays in its interior,
c.f. \cite{Tomo}); 
$(iii)$ Gibbs measures theory (c.f. \cite{BleherRuizZagrebnov, Georgii}).

%
%
\noindent{\bf Reconstruction on trees: A brief survey.}

The reconstruction problem is relatively well understood
in case the graph is a tree (see \cite{TreeRec}).
The fundamental reason for this is that then the canonical
measure $\mu(\ux)$ admits a simple description. More precisely, 
to sample $\uX$ from $\mu(\cdot)$, first sample the value
of $X_{\root}$ according to the marginal law $\mu_{\root}(x_{\root})$,
then recursively for each node $j$, sample its children $\{X_\ell\}$
\emph{independently} conditional on their parent value. 

Because of this Markov structure, one can derive a recursive
distributional equation for the conditional marginal at the root
$\nu^{(t)} (\,\cdot\,)\equiv
\mu_{\root|\cBall_{\root}(t)}(\,\cdot\,|\uX_{\cBall_{\root}(t)})$
given the variable values at generation $t$ (just as we have
done in the course of proving Proposition \ref{prob:k-tree-reco}).
Note that $\nu^{(t)}(\,\cdot\,)$ is
a random quantity even for a deterministic graph $G_n$
(because $\uX_{\cBall_{\root}(t)}$ 
is itself drawn randomly from the distribution 
$\mu(\,\cdot\,)$). Further, it contains all the 
information in the boundary about $X_{\root}$
(i.e. it is a `sufficient statistic'), so  
the standard approach to tree reconstruction is to study 
the asymptotic behavior of the distributional recursion 
for $\nu^{(t)}(\,\cdot\,)$.

Indeed, following this approach,
reconstructibility has been thoroughly characterized
for zero magnetic field Ising models on generic trees 
(c.f. \cite{BleherRuizZagrebnov,Asymmetric,EvaKenPerSch}).
More precisely, for such model on 
an infinite tree $\Tree$ of branching number br$(\Tree)$,
the reconstruction problem is solvable 
if and only if br$(\Tree)(\tanh\beta)^2>1$.
For the cases we treat in the sequel, 
br$(\Tree)$ coincides with the mean 
offspring
number of any vertex, hence this result 
establishes a sharp reconstruction threshold 
in terms of the average degree (or in terms of 
the inverse temperature parameter $\beta$), 
that we shall generalize here to random graphs. 

Reconstruction on general graphs poses new challenges,
since it lacks such recursive description of
sampling from the measure $\mu(\,\cdot\,)$.
The result of \cite{BergerEtAl} allows for deducing 
non-reconstructibility from fast mixing of certain 
reversible Markov chains with local updates. However,
proving such fast mixing is far from being an easy task,
and in general the converse does not hold 
(i.e. one can have slow mixing and non-reconstructibility).

A threshold $\lambda_{\rm r}$ for fast mixing 
has been established in \cite{MWW} 
for the independent set model of (\ref{eq:IndependentSet}), 
in case $G_n$ are random bipartite graphs.
Arguing as in \cite{GerschenMon1}, it can be shown 
that this is also the graph reconstruction threshold. 
An analogous result was proved in \cite{GerschenMon1}
for the ferromagnetic Ising model and random regular graphs
(and it extends also to Poisson random graphs, see
\cite{DemboMonUnpublished}).
In all of these cases, the graph 
reconstruction threshold does not coincide with the 
tree reconstruction threshold, but coincides
instead with the tree `uniqueness threshold' 
(i.e. the critical parameter 
such that the uniform decorrelation condition 
$\sup_{\ux}\Delta(t;\ux)\to 0$ holds).
%
%
\subsection{Reconstruction on graphs: sphericity and tree-solvability}

For the sake of clarity, we focus hereafter
on \emph{Poisson} graphical models.  
Specifying such an ensemble requires
an alphabet $\cX$, a \emph{density parameter} $\gamma\ge 0$, 
a finite collection of non-negative, symmetric functionals 
$\psi_a(\,\cdot\, ,\,\cdot\,)$ on $\cX \times \cX$, indexed
by $a\in\cC$, and a probability distribution $\{p(a):\, a\in \cC\}$ on $\cC$.
In the random multi-graph $G_n$ the multiplicities of edges 
between pairs of vertices $i \ne j \in [n]$ are 
independent Poisson$(2\gamma/n)$ 
random variables, and $G_n$ has additional independent Poisson$(\gamma/n)$ 
self-loops at each vertex $i \in [n]$. For each occurrence of 
an edge $e=\{e_1,e_2\}$ in $G_n$ (including its self-loops), 
we draw an independent random variable $A_e \in \cC$ 
according to the distribution $\{p(\,\cdot\,)\}$ and 
consider the graphical model of specification 
$\upsi \equiv \{\psi_{A_e}(x_{e_1},x_{e_2}): e\in 
G_n\}$. Finally, the root $\root$ is uniformly chosen in $[n]$,
independently of the graph-specification pair $(G_n,\upsi)$.

For example, the uniform measure over proper $q$-colorings 
fits this framework (simply take 
$\cX = \cX_q$ and $|\cC|=1$ with $\psi(x,y) = \ind(x \ne y))$.

It is easy to couple the multi-graph $G_n$ of the Poisson model 
and the Erd\"os-Renyi random graph from the ensemble 
$\graph(\gamma,n)$ such that the two graphs differ 
in at most 
$\Delta_n = \sum_{1 \le i \le j \le n} Y_{\{i,j\}}$ edges, 
where the independent variables $Y_{\{i,j\}}$ have the
Poisson($\gamma/n$) distribution when $i=j$ and that of
$($Poisson($2\gamma/n$)-1$)_+$ when $i \ne j$.
It is not hard to check that $\Delta_n/(\log n)$ is almost surely
uniformly bounded, and hence by 
Proposition \ref{prop:local-trees2},
almost surely the Poisson multi-graphs $\{G_n\}$ are
uniformly sparse and converge locally to the rooted at 
$\root$,
Galton-Watson tree $\Tree$ of Poisson$(2\gamma)$ 
offspring distribution. Let $\Tree(\ell)$, $\ell\ge 0$ denote
the graph-specification pair on the  
first $\ell$ generations of $\Tree$, where each edge carries the 
specification $\psi_a(\,\cdot\, ,\,\cdot\,)$ with probability $p(a)$,
independently of all other edges and of the realization of $\Tree$.

It is then natural to ask whether reconstructibility 
of the original graphical models is related to reconstructibility
of the graphical models $\mu^{\Tree(\ell)}(\ux)$ 
per Eq.~(\ref{eq:GeneralGraphModel}) for 
$G=\Tree(\ell)$ and the same specification $\upsi$. 
\begin{definition}
Consider a sequence of random graphical models $\{G_n\}$ converging 
locally to the random rooted tree $\Tree$.
We say that the reconstruction problem is \emph{tree-solvable} for 
the sequence $\{G_n\}$ if 
it is solvable for $\{\Tree(\ell)\}$. That is, 
there exists $\ve>0$ such that, as $\ell\to\infty$,
for any $t\ge 0$,
\begin{eqnarray}
\Vert \mu_{\root,\cBall_{\root}(t)}^{\Tree(\ell)}-
 \mu^{\Tree(\ell)}_{\root}\times
\mu^{\Tree(\ell)}_{\cBall_{\root}(t)}\Vert_{\sTV} 
\ge
\ve\, ,
\label{eq:TreeReconstruction}
\end{eqnarray}
with positive probability. 
\end{definition}
This definition could have been expressed directly in terms of
the free boundary Gibbs measure $\mu^{\Tree}$ on the 
infinite rooted tree $\Tree$. Indeed, the reconstruction 
problem is tree-solvable if and only if with positive probability
$$
\liminf_{t\to\infty} \Vert \mu_{\root,\cBall_{\root} (t)}^{\Tree}- 
\mu^{\Tree}_{\root}\times
\mu^{\Tree}_{\cBall_{\root}(t)}\Vert_{\sTV} > 0 \;.
$$ 
While Eqs.~(\ref{eq:TreeReconstruction}) and 
(\ref{eq:GraphReconstruction}) are similar, as
explained before, passing from the original graph
to the tree is a significant simplification (due to the recursive 
description of sampling from $\mu^{\Tree(\ell)}(\,\cdot\,)$).

%
%
We proceed with a sufficient condition 
for graph-reconstruction to be equivalent to tree reconstruction.
To this end, we introduce the concept of `two-replicas type' as follows.
Consider a graphical model 
$G$ and two i.i.d. samples $\uX^{(1)}$, $\uX^{(2)}$ 
from the corresponding canonical measure 
$\mu(\,\cdot\,)=\mu_{(G,\upsi)}(\cdot)$
(we will call them \emph{replicas} following the spin glass
terminology). The \emph{two replica type} is a matrix
$\{\nu(x,y):\, x,y\in\cX\}$ where $\nu(x,y)$ counts the fraction 
of vertices $j$ such that $X^{(1)}_j=x$ and $X^{(2)}_j=y$.
We denote by $\Rep$ the set of distributions $\nu$ 
on $\cX\times\cX$ and by $\Rep_n$ the subset of 
\emph{valid two-replicas types}, that is, distributions
$\nu$ with $n\nu(x,y)\in \naturals$ for all $x,y \in \cX$.

The matrix 
$\nu=\nu_n$ 
is a random variable, because the 
graph $G_n$ is random, and the two replicas $\uX^{(1)}$, $\uX^{(2)}$ 
are i.i.d. conditional on $G_n$. If $\mu(\,\cdot\,)$ 
was 
the uniform distribution, then 
$\nu_n$
would concentrate 
(for large $n$),
around $\onu(x,y)\equiv 1/|\cX|^2$. 
Our sufficient condition requires this to be approximately true.
\begin{thm}\label{thm:graph<->tree}
Consider a sequence of random Poisson graphical models $\{G_n\}$.
Let $\nu_n
(\,\cdot\, ,\,\cdot\,)$ be the type of  two i.i.d.
replicas $\uX^{(1)}$, $\uX^{(2)}$, and 
$\Delta\nu_n (x,y)\equiv\nu_n(x,y)-\onu(x,y)$. 
Assume that, for any $x\in\cX$,
\begin{eqnarray}
\lim_{n \to \infty} \E\Big\{\big[\Delta\nu_n(x,x)-2|\cX|^{-1}\sum_{x'}
\Delta\nu_n(x,x')\big]^2\Big\} 
= 0\, . 
\label{eqn:graph<->tree}
\end{eqnarray}
Then,
the reconstruction problem for $\{G_n\}$ is solvable if and only if
it is tree-solvable.
\end{thm}
\begin{remark} 
The expectation in Eq.~(\ref{eqn:graph<->tree})
is with respect to the two replicas $\uX^{(1)}$, $\uX^{(2)}$ 
(which the type $
\nu_n(\,\cdot\,,\,\cdot\,)$ is a function of), 
conditional on $G_n$, as well as with respect to $G_n$. Explicitly, 
\begin{align}
\E\{F(\uX^{(1)},\uX^{(2)})\} &= 
\E\Big\{\sum_{\ux^{(1)},\ux^{(2)}}\!\mu_{G_n}(\ux^{(1)})
\mu_{G_n}(\ux^{(2)})\;
F(\ux^{(1)},\ux^{(2)})\Big\} \,.
\end{align}
\end{remark}
\begin{remark}\label{rem:sphericity} 
It is easy to see that the sphericity condition 
of Definition \ref{def:Spherical} implies 
Eq.~(\ref{eqn:graph<->tree}). 
That is, (\ref{eqn:graph<->tree}) holds if $\mu_{G_n}$ are
$(\ve,\delta_n)$-spherical for any $\ve>0$ and some $\delta_n(\ve) \to 0$.
\end{remark}

\begin{remark} 
In fact, as is hinted by the proof, 
condition (\ref{eqn:graph<->tree}) can be weakened, e.g.
$\onu(\,\cdot\,\,\cdot\,)$ can be chosen more generally than the uniform
matrix. Such a generalization amounts to assuming that
`replica symmetry is not broken' 
(in the spin glass terminology, see \cite{MezardMontanari}).
For the sake of simplicity we omit such  generalizations.
\end{remark}

Condition (\ref{eqn:graph<->tree})
emerges naturally in a variety of contexts, a notable one 
being second moment method applied to random constraint 
satisfaction problems.
As an example, consider proper colorings of random graphs, cf. 
Section \ref{chap:Coloring}.
The second moment method was used in   \cite{AchlioptasNaorPeres}
to bound from below the colorability threshold.
The reconstruction threshold on trees was estimated in 
\cite{Bha-Ver-Vig,sly}. Building on these results,
and as outlined at the end of Section \ref{sec:pfs-main-result}
the following statement is obtained in \cite{MonResTet}.
\begin{thm}\label{thm:Colorings}
For proper $q$-colorings of a Poisson random graph of density
$\gamma$, the reconstruction problem is solvable if and only if
$\gamma>\gamma_{\rm r}(q)$, where for large $q$,
\begin{eqnarray}\label{eq:sly}
\gamma_{\rm r}(q)
= \frac{1}{2}\, q\, [\log q+\log\log q+o(1)]\, .
\end{eqnarray}
\end{thm}

In general the graph and tree reconstruction thresholds do 
not coincide. For example, as mentioned before, 
zero magnetic field ferromagnetic Ising models 
on the Galton-Watson tree $\Tree(\node,\edge,\infty)$
(of Section \ref{ch:IsingModel}), 
are solvable if and only if $\aedeg (\tanh(\beta))^2>1$. 
The situation changes dramatically for graphs, 
as shown in \cite{DemboMonUnpublished,GerschenMon1}.
\begin{thm}\label{thm:ferromagnet}
For both Poisson random graphs and random regular graphs, 
reconstruction is solvable for zero magnetic field, ferromagnetic 
Ising models, if and only if $\aedeg \tanh(\beta)>1$.
\end{thm}
In physicists' language, the ferromagnetic phase transition occurring at 
$\aedeg\tanh(\beta)=1$, cf. Section \ref{ch:IsingModel},
`drives' the reconstruction threshold.
The proof of reconstructibility for $\aedeg \tanh(\beta)>1$ essentially
amounts to finding a bottleneck in Glauber dynamics.
As a consequence it immediately implies that the mixing time is exponential
in this regime. We expect this to be a tight estimate of the threshold
for exponential mixing.

On the other hand, for 
a zero magnetic field, 
Ising spin-glass, the tree and graph thresholds do coincide.
In fact, 
for such a model 
on a Galton-Watson tree with 
Poisson$(2\gamma)$ offspring distribution, reconstruction 
is solvable if and only if 
$2\gamma (\tanh(\beta))^2>1$ (see, \cite{EvaKenPerSch}). 
The corresponding graph result is: 
\begin{thm}\label{thm:spinglass}
Reconstruction is solvable for Ising spin-glasses of zero magnetic
field, on Poisson random graph of density parameter $\gamma$,
provided $2\gamma(\tanh(\beta))^2>1$, and it is unsolvable if 
$2\gamma(\tanh(\beta))^2<1$.
\end{thm}
%
%
\subsection{Proof of main results}\label{sec:pfs-main-result}

Hereafter, let 
$\Ball_i(t) = \{ j \in [n] : d(i,j)\le t \}$, 
$\cBall_i(t) = \{ j \in [n] : d(i,j)\ge t \}$
and $\Edge_i(t) \equiv \Ball_i(t)\cap\cBall_i(t)$ 
(i.e. the set of vertices of distance $t$ from $i$).
Further, partition the edges of $G_n$ between the
subgraphs $\Ball_i(t)$ and $\cBall_i(t)$ so 
edges between two vertices from $\Edge_i(t)$ 
are all in $\cBall_i(t)$, and excluded from $\Ball_i(t)$.

Beyond the almost sure convergence of the law of $\Ball_{\root}(t)$ 
to the corresponding Galton-Watson tree of depth-$t$, 
rooted at $\root$ (which as explained before, is a consequence
of Proposition \ref{prop:local-trees2}),
the proof of Theorem \ref{thm:graph<->tree} relies 
on the following form of independence between $\Ball_{\root}(t)$ 
and $\cBall_{\root}(t)$ for Poisson random graphs. 
\begin{propo}\label{prop:poisson_independance}
Let $G_n$ be a Poisson random graph on vertex set $[n]$ and 
density parameter $\gamma$. 
Then, conditional on $\Ball_{\root}(t)$,
$\cBall_{\root}(t)$ is a Poisson random graph on vertex set 
$[n]\setminus \Ball_{\root}(t-1)$ with same edge 
distribution as $G_n$.
\end{propo}
\begin{proof}
Condition on $\Ball_{\root}(t) = \Graph(t)$, and let 
$\Graph(t-1)=\Ball_{\root}(t-1)$
(notice that this is uniquely determined from $\Graph(t)$).
This is equivalent to conditioning on a given edge realization 
between the vertices 
$k$, $l$ such that $k\in \Graph(t-1)$ and $l\in \Graph(t)$.

The 
graph $\cBall_{\root}(t)$ has as vertices the set   
$[n]\setminus\Graph(t)$ and its edges are those $(k,l)\in G_n$ such that
$k,l\not\in \Graph(t-1)$. Since the latter set of edges is disjoint
from the one we are conditioning upon, the claim follows by the 
independence of the choice of edges taken into $G_n$.
\end{proof}
%

%

We also need to bound the tail of the distribution of the number of vertices
in the depth-$t$ neighborhood of $\root$. This can be done by comparison with
a Galton-Watson process.
\begin{propo}\label{prop:BoundNeighborhood}
Let 
$\Vert \Ball_{\root}(t) \Vert$ 
denote the number of edges (counting their multiplicities), 
in depth-$t$ neighborhood of the root in a Poisson random graph $G_n$ 
of density $\gamma$. 
Then, for any $\lambda>0$ there exists 
finite $g_t(\lambda,\gamma)$ such that, for any $n$, $M\ge 0$
\begin{eqnarray}\label{eq:bd-vert}
\prob\{ \Vert\Ball_{\root}(t)\Vert\ge M\}
\le g_t(\lambda,\gamma)\, \lambda^{-M}\, .
\end{eqnarray}
\end{propo}
\begin{proof}
Notice that, because of the symmetry of the graph distribution under 
permutation of the vertices, we can 
and shall fix $\root$ to be a deterministic vertex. 
Starting at $\root$ we explore $G_n$ in breadth-first fashion
and consider the sequence of random 
variables $E_t=\Vert \Ball_{\root}(t) \Vert$. Then,
for each $t \ge 0$, the value of $E_{t+1}-E_t$ is,
conditional on $\Ball_{\root}(t)$, 
upper bounded by the sum of 
$|\Edge_{\root}(t)| \times |\cBall_{\root}(t)|$ 
i.i.d. Poisson($2\gamma/n$) random variables.
Since $|\cBall_{\root}(t)| \le n$ and 
$|\Edge_{\root}(t)| \le E_t-E_{t-1}$ for $t \ge 1$
(with $|\Edge_{\root}(0)|=1$), it follows that
$E_t$ is stochastically dominated by
$|\Tree(t)|$, where $\Tree(t)$ is a depth-$t$ Galton-Watson tree 
with Poisson$(2\gamma)$ offspring distribution. By Markov's inequality,
$$
\prob\{\Vert \Ball_{\root}(t) \Vert \ge M\}\le 
\E\{\lambda^{|\Tree(t)|}\}\, \lambda^{-M} \,.
$$
To complete the proof, recall that 
$g_t(\lambda,\gamma)\equiv\E\{\lambda^{|\Tree(t)|}\}$ is the finite
solution of the recursion 
$g_{t+1}(\lambda,\gamma) = \lambda \xi (g_t(\lambda,\gamma),\gamma)$
for $\xi(\lambda,\gamma) = e^{2\gamma(\lambda-1)}$ and  
$g_0(\lambda,\gamma) = \lambda$. 
\end{proof}
%

%

In order to prove Theorem \ref{thm:graph<->tree} we will
first establish that, under condition (\ref{eqn:graph<->tree}),
any (fixed) subset of the variables $\{X_1,\dots,X_n\}$
is (approximately) uniformly distributed. This is, at first sight, 
a surprising fact. 
Indeed, 
the condition (\ref{eqn:graph<->tree})
only provides direct control on two-variables correlations. 
It turns out that two-variables correlations
control $k$-variable correlations for any bounded $k$ because
of the symmetry among  $X_1,\dots,X_n$. To clarify this 
point, 
it 
is convenient to take a more general point of view. 
\begin{definition}
For any distribution $\mu(\,\cdot\,)$ over $\cX^n$ (where $\cX$ is
a generic measure space), and any permutation $\pi$ over the set
$\{1\,\dots,n\}$ let $\mu^{\pi}(\,\cdot\,)$ denote the distribution 
obtained acting with $\pi$ on $\cX\times\cdots\times \cX$.

Let $\mu(\,\cdot\,)$ be a random probability distribution 
over  $\cX\times\cdots\times \cX$. We say that $\mu$
is \emph{stochastically exchangeable} if $\mu$ is distributed
as $\mu^{\pi}$ for any permutation $\pi$.
\end{definition}

\begin{propo}\label{degenerescence}
Suppose (\ref{eqn:graph<->tree}) holds for a finite set $\cX$ and
the type $\nu_n$ of two i.i.d. replicas $\uX^{(1)}$, $\uX^{(2)}$
from a sequence of stochastically exchangeable random measures 
$\mu^{(n)}$ on $\cX^n$. Then, for any fixed set of vertices 
$i(1),\dots,i(k)\subseteq [n]$ and any $\xi_1,\dots, \xi_k\in\cX$,  
as $n \to \infty$,
\begin{eqnarray}
\E \Big\{ \big|\mu^{(n)}_{i(1),\dots,i(k)}(\xi_1,\dots, \xi_k) - 
|\cX|^{-k} \big|^2 \Big\} \to 0 \,.\label{eqn:degenerescence}
\end{eqnarray}
\end{propo}

\begin{proof}
Per given replicas $\uX^{(1)}$, $\uX^{(2)}$, we define, for any $\xi\in\cX$ 
and $i \in [n]$,
\begin{eqnarray*}
\sQ_i(\xi) = \Big\{
\ind(X^{(1)}_i=\xi)-\frac{1}{|\cX|}\Big\}\Big\{
\ind(X_i^{(2)}=\xi)-\frac{1}{|\cX|}\Big\} 
\end{eqnarray*}
and let $\sQ(\xi) = n^{-1} \sum_{i=1}^n \sQ_i(\xi)$ denote the  
average of $\sQ_i(\xi)$ over a uniformly random $i\in [n]$. Since 
$$
\sQ(\xi) = \Delta\nu_n(\xi,\xi)-
|\cX|^{-1} \sum_{x'}\Delta\nu_n(\xi,x')
- |\cX|^{-1} \sum_{x'}\Delta\nu_n(x',\xi),
$$
it follows from (\ref{eqn:graph<->tree}) 
and the triangle inequality, that $\E\{\sQ(\xi)^2\}\to 0$ as $n \to \infty$. 
Further, 
$|Q(\xi)|\le 1$, so by the Cauchy-Schwarz inequality we deduce that 
for any fixed, non-empty $U \subseteq [n]$, $b \in U$ and $\xi_a \in \cX$,
\begin{eqnarray*}
\Big|\E\big\{\prod_{a \in U} \sQ(\xi_a)\big\}\Big| \le
\E|\sQ(\xi_b)|\to  0\, .
\end{eqnarray*}

Next, fixing $i(1),i(2),\ldots,i(k)$ and $U \subseteq [k]$, let
\begin{eqnarray*}
Y_U \equiv 
\E\Big\{ \prod_{a\in U} (\ind(X_{i(a)}=\xi_a)-|\cX|^{-1})\,
\big|\mu \Big\} \,,
\end{eqnarray*}
where $\E\{\,\cdot\,|\mu\}$ denotes the
expectation with respect to the measure $\mu(\,\cdot\,)$
of the replicas $\uX^{(1)}$, $\uX^{(2)}$, i.e. at fixed 
realization of $\mu=\mu^{(n)}$.
Note that by the stochastic exchangeability of $\mu$, and 
since $\sup_\xi |\sQ(\xi)|\le 1$, we have that for any
non-empty $U \subseteq [k]$,
\begin{eqnarray*}
\E \{ Y_U^2 \} = \E\big\{\prod_{a \in U} \sQ_{i(a)}(\xi_a)\big\} =
\E\big\{\prod_{a \in U} \sQ(\xi_a)\big\} +\Delta_{U,n}
\, ,
\end{eqnarray*}
where $|\Delta_{U,n}|$ 
is upper bounded by the probability that $|U| \le k$
independent uniform in $[n]$ random variables 
are not distinct, which is $O(1/n)$.
Thus, $\E \{ Y_U^2 \} \to 0$ as $n\to\infty$, for 
any fixed, non-empty $U \subseteq [k]$.

The proof of the proposition is completed by noting that 
$Y_\emptyset = 1$ and 
\begin{eqnarray*}
\mu_{i(1),\dots,i(k)}(\xi_1,\dots, \xi_k) 
= \sum_{U \subseteq [k]} |\cX|^{|U|-k} Y_U \,,
\end{eqnarray*}
hence by the Cauchy-Schwarz inequality,
\begin{eqnarray*}
\E \Big\{ \big| \mu_{i(1),\dots,i(k)}(\xi_1,\dots, \xi_k) - 
|\cX|^{-k} \big|^2 \big\} \le
\sum_{\emptyset \ne U,V \subseteq [k]} \E |Y_U Y_V| 
\le 2^k \sum_{\emptyset \ne U \subseteq [k]} \E \{ Y_U^2 \}
\end{eqnarray*}
goes to zero as $n \to \infty$.
\end{proof}

The following lemma is the key for relating the 
solvability of the reconstruction problem to its tree-solvability. 
\begin{lemma}\label{lem:Simple2}
For any graphical model $\mu=\mu_{G_n,\upsi}$, 
any vertex $\root \in [n]$, and all $t\le \ell$, 
\begin{align}
\Big| \Vert\mu_{\root,\cBall_{\root}(t)}-
\mu_\root \times \mu_{\cBall_{\root}(t)}\Vert_{\sTV}
&-
\Vert \mu^<_{\root,\cBall_{\root}(t)}
 -\mu^<_{\root} \times \mu^<_{\cBall_{\root}(t)}\Vert_{\sTV}\Big| \nonumber\\
&\le 5 |\cX|^{|\Ball_{\root}(\ell)|}\, \Vert\mu^{>}_{\Edge_{\root}(\ell)}
-\rho_{\Edge_{\root}(\ell)}\Vert_{\sTV}\, ,\label{eq:treeBoundary}
\end{align}
where for any $U \subseteq [n]$, we let $\rho_U(\ux_U)=1/|\cX|^{|U|}$
denote the uniform distribution of $\ux_U$, with $\mu^{<}_U$ denoting 
the marginal law of $\ux_U$ for the graphical
model in which the edges of $\cBall_{\root}(\ell)$ are omitted,
whereas $\mu^{>}_U$ denotes such marginal law 
in case all edges of $\Ball_{\root}(\ell)$ are omitted.
\end{lemma}
\begin{proof} Adopting hereafter the shorthands
$\Ball(t)$, $\cBall(t)$ and $\Edge(t)$ for
$\Ball_{\root}(t)$, $\cBall_{\root}(t)$
and $\Edge_{\root}(t)$, respectively, recall that 
by the definition of these sets there are no edges 
in $G_n$ between $\Ball(t)$ and $\cBall(t) \setminus \Edge(t)$.
Hence, $\Delta(t,\ux)$ of Eqn.~(\ref{eq:UniformDecay}) 
depends only on $\ux_{\Edge(t)}$ and consequently,
\begin{align*}
\Vert \mu_{\root,\cBall(t)}-
\mu_\root \times \mu_{\cBall(t)} \Vert_{\sTV}
&= \sum_{\ux} \mu_{\cBall(t)}(\ux_{\cBall(t)}) \Delta(t,\ux) \\
&= \sum_{\ux} \mu_{\Edge(t)}(\ux_{\Edge(t)}) 
\Vert \mu_{\root|\Edge(t)}(\,\cdot\, |\ux_{\Edge(t)})
- \mu_{\root}(\,\cdot\,) \Vert_{\sTV} \,.
\end{align*}
By the same reasoning also 
\begin{eqnarray*}
\Vert \mu^<_{\root,\cBall(t)}
 -\mu^<_{\root} \times \mu^<_{\cBall(t)}\Vert_{\sTV} =
\sum_{\ux} \mu^<_{\Edge(t)}(\ux_{\Edge(t)}) 
\Vert \mu^<_{\root|\Edge(t)}(\,\cdot\, |\ux_{\Edge(t)})
- \mu^<_{\root}(\,\cdot\,) \Vert_{\sTV} \,.
\end{eqnarray*}
Since $\root \in \Ball(t) \subseteq \Ball(\ell)$, 
the conditional law of $x_\root$ given $\ux_{\Edge(t)}$ is
the same under the graphical model for $G_n$ and the one 
in which all edges of $\cBall(\ell)$ are omitted.
Further, by definition of the total variation distance,
the value of $\Vert\mu_{U}-\mu^<_{U}\Vert_{\sTV}$ is non-decreasing 
in $U\subseteq \Ball(\ell)$. With the total variation distance 
bounded by one, it thus follows
from the preceding identities and the triangle inequality 
that the left hand side of Eq.~(\ref{eq:treeBoundary}) 
is bounded above by 
\begin{eqnarray*}
\Vert \mu_{\root}-\mu^<_{\root}\Vert_{\sTV} +
2 \Vert \mu_{\Edge(t)}-\mu^<_{\Edge(t)}\Vert_{\sTV}
\le 
3 \Vert\mu_{\Ball(\ell)}-\mu^<_{\Ball(\ell)} \Vert_{\sTV}\,.
\end{eqnarray*}

Next, considering the distribution $\mu_{\Ball(\ell)}(z)$ on the 
discrete set $\cZ=\cX^{\Ball(\ell)}$, notice that, 
as a consequence of Eq.~(\ref{eq:GeneralGraphModel}) 
and of the fact that $\Ball(\ell)$ and $\cBall(\ell)$ are edge disjoint,
\begin{eqnarray}\label{eq:basic-rep}
\mu_{\Ball(\ell)}(z) = \frac{f(z) \rho_2(z)}
{\sum_{z' \in \cZ} f(z') \rho_2(z')} \;,
\end{eqnarray}
for the $[0,1]$-valued function $f= \mu^{<}_{\Ball(\ell)}$ on $\cZ$,
and the distribution $\rho_2= \mu^{>}_{\Ball(\ell)}$ on this set.
Clearly, replacing $\rho_2$ in the right hand side of (\ref{eq:basic-rep})
by the uniform distribution $\rho_1=\rho$ on $\cZ$,
results with $\sum_{z' \in \cZ} f(z') \rho_1 (z') =1/|\cZ|$ 
and in the notations of (\ref{eq:SimpleIneq}),
also with $\rh_1=f$. We thus deduce from the latter bound that
$$
\Vert\mu_{\Ball(\ell)}-\mu^<_{\Ball(\ell)} \Vert_{\sTV}
\le \frac{3}{2} |\cZ| \Vert \mu^{>}_{\Ball(\ell)} - \rho_{\Ball(\ell)}
\Vert_{\sTV} \,,
$$
and the proof of the lemma is complete upon noting that 
$\mu^{>}_{\Ball(\ell)}$ deviates from the uniform distribution
only in terms of its marginal on $\Edge(\ell)$. 
\end{proof}

\begin{proof}[Proof of Theorem \ref{thm:graph<->tree}] 
Fixing $t \le \ell$, let
$\Delta_n$ denote  the left hand side of  
Eq.~(\ref{eq:treeBoundary}). We claim that its expectation 
with respect to the Poisson random model $G_n$ 
vanishes as $n\to\infty$. 
First, with $\Delta_n \le 1$ and 
$\sup_n \prob(\Vert \Ball_{\root}(\ell) \Vert \ge M) \to 0$ as 
$M \to \infty$, see Proposition \ref{prop:BoundNeighborhood}, 
it suffices to prove that for any finite $M$, as $n \to \infty$,
\begin{eqnarray*}
\E \{ \Delta_n \ind(\Vert \Ball_{\root}(\ell) \Vert < M) \} \to 0 \,.
\end{eqnarray*}
Decomposing this expectation according to the finitely many
events $\{\Ball_{\root}(\ell)={\sf H}\}$, indexed by rooted, 
connected, multi-graphs ${\sf H}$ 
of less than $M$ edges (counting multiplicities), 
we have by (\ref{eq:treeBoundary}) that 
\begin{align*}
\E \{ \Delta_n \ind(\Vert \Ball_{\root}(\ell) \Vert < M) \} \le 
5 |\cX|^M \sum_{\Vert {\sf H}\Vert < M} 
\E\{ ||\mu^{>}_{\Edge_{\root}(\ell)}-\rho_{\Edge_{\root}(\ell)}||_{\sTV}
\, \big|\, \Ball_{\root}(\ell)={\sf H}\}\, ,
\end{align*}
and it is enough to show that each of the terms on the right hand 
side vanishes as $n\to\infty$. 

Recall Proposition \ref{prop:poisson_independance} that 
each term in the sum is the expectation, with 
respect to a Poisson graphical model of density $\gamma$
over the collection  $[n] \setminus \Ball_{\root}(\ell-1)$
of at least $n-M$ vertices. The event $\{\Ball_{\root}(\ell)={\sf H}\}$
fixes the set $\Edge = \Edge_{\root}(\ell)$ whose finite size
depends only on the rooted multi-graph ${\sf H}$. 
By Proposition \ref{degenerescence} we thus deduce 
that conditional on this event, the expected value of 
\begin{eqnarray*}
\Vert \mu^{>}_{\Edge}-\rho_{\Edge} \Vert_{\sTV}=
\frac{1}{2}\sum_{\ux_{\Edge}}\left|\mu^{>}_{\Edge}(\ux_{\Edge})
-|\cX|^{-|\Edge|}\right|\, ,
\end{eqnarray*}
vanishes as $n \to \infty$. To recap, we have shown that 
for any $t \le \ell$, the expected value of the left hand side of  
Eq.~(\ref{eq:treeBoundary}) vanishes as $n \to \infty$. 

In view of Definition \ref{defn:recon}, this implies 
that the reconstruction problem is solvable for $\{G_n\}$ 
if and only if $\inf_t \limsup_{\ell \to \infty} 
\limsup_{n \to \infty} \prob\{A_n(t,\ell,\ve) \} > 0$
for some $\ve>0$, where $A_n (t,\ell,\ve)$ denotes the event
$$
\Vert \mu^<_{\root,\cBall_{\root}(t)}
 -\mu^<_{\root} \times \mu^<_{\cBall_{\root}(t)}\Vert_{\sTV} \ge \ve \,.
$$
Recall that $\mu^{<}(\,\cdot\,)$ is the canonical measure for 
the edge-independent random specification on the random graph 
$\Ball_{\root}(\ell)$ and that almost surely the uniformly sparse 
Poisson random graphs $\{G_n\}$ converge locally to the Galton-Watson 
tree $\Tree$ of Poisson($2\gamma$) offspring distribution. 
Applying Lemma \ref{lemma:EdgeTree} for the uniformly bounded function 
$\ind(A_n(t,\ell,\ve))$ of $\Ball_{\root}(\ell)$ and averaging first 
under our uniform choice of $\root$ in $[n]$, we deduce that 
$\prob\{ A_n(t,\ell,\ve) \} \to \prob \{ A_\infty (t,\ell,\ve) \}$, 
where $A_\infty(t,\ell,\ve)$ denotes the event on 
the left hand side of (\ref{eq:TreeReconstruction}). 
That is, $\{G_n\}$ is solvable if and only if
$\inf_t \limsup_{\ell \to \infty} \prob\{A_\infty (t,\ell,\ve) \} > 0$
for some $\ve>0$, which is precisely the definition of 
tree-solvability.  
\end{proof}

\begin{proof}[Proof of Theorem \ref{thm:Colorings}]
Following \cite{MonResTet}, this proof consists of four steps: 

\vspace{0.15cm}
\noindent$(1)$ It is shown in \cite{sly} that for 
regular trees of degree $2\gamma$
the reconstruction threshold $\gamma_{\rm r, tree}(q)$ 
for proper $q$-colorings grows with $q \to \infty$ 
as in (\ref{eq:sly}).
In the large $\gamma$ limit considered here, a Poisson$(2\gamma)$
random variable is tightly concentrated around its mean. Hence,
as noted in \cite{sly}, the result (\ref{eq:sly}) extends 
straightforwardly to the case of random Galton-Watson trees 
with offspring distribution Poisson$(2\gamma)$.

\vspace{0.15cm}
\noindent $(2)$ Given two balanced proper $q$-colorings 
$\ux^{(1)}$, $\ux^{(2)}$ of $G_n$
(a $q$-coloring is balanced if it has exactly $n/q$ vertices of each color), 
recall that 
their \emph{joint type} is the $q$-dimensional matrix $\nu(\cdot,\cdot)$ 
such that $\nu(x,y)$ counts the fraction of vertices $i \in [n]$ 
with $x^{(1)}_i=x$ and $x^{(2)}_i=y$. 
Let $Z_{\rm b}(\nu)$ denote the number of balanced pairs of 
proper $q$-colorings $\ux^{(1)}$, $\ux^{(2)}$ of $G_n$ with
the given joint type $\nu$. For $\gamma \le q\log q-O(1)$, 
while proving Theorem \ref{thm:lbd-qcol}
it is shown in \cite{AchlioptasNaor} that  
$\E\, Z_{\rm b}(\nu)/ \E\, Z_{\rm b}(\overline{\nu}) \to 0$ 
exponentially in $n$ (where
$\onu(x,y) = 1/q^2$ denotes the uniform joint type).

\vspace{0.15cm}
\noindent $(3)$ The preceding result implies that, for any $\ve>0$ 
and some non-random $\delta_n(\ve) \to 0$,
the uniform measure over proper $q$-colorings of an instance of the 
random Poisson multi-graph $G_n$ is with high probability 
$(\ve,\delta_n)$-spherical (see Definition \ref{def:Spherical}).
Notice that this implication is not straightforward as it requires
bounding the expected ratio of $Z_{\rm b}(\nu)$ to the total number of
pairs of proper $q$-colorings. We refer to \cite{MonResTet} for this 
part of the argument.

\vspace{0.15cm}
\noindent $(4)$ As mentioned in 
Remark \ref{rem:sphericity}, by Theorem \ref{thm:graph<->tree} 
the latter sphericity condition yields that with high probability
the $q$-colorings reconstruction problem is solvable if and only if 
it is tree-solvable. Therefore, the result of step (1) about 
the tree-reconstruction threshold $\gamma_{\rm r, tree}(q)$  
completes the proof.
\end{proof}

\section{XORSAT and finite-size scaling}
\label{ch:XORSAT}

XORSAT is a special constraint satisfaction problem (CSP)
first introduced in  \cite{CrDa}.
An instance of XORSAT  is defined by a pair 
$\cF =(\H,\ub)$, where $\H$ is an $m\times n$ 
binary matrix and $\ub$ is a binary vector of length $m$.
A solution of this instance is just a solution of
the linear system 
\begin{eqnarray}
\H\, \ux = \ub\, \;\;\;\;\mod 2\, .
\end{eqnarray}
In this section we shall focus on the $l$-XORSAT problem 
that is defined by requiring that $\H$ has $l$
non zero entries per row. Throughout this section we assume $l\ge 3$.

It is quite natural to associate to an 
instance $\cF = (\H,\ub)$ the uniform measure over 
its solutions, $\mu(\ux)$. If the positions
of the non-vanishing entries of the $a$-th row 
of $\H$ are denoted by $i_1(a),\dots, i_l(a)$, the latter
takes the form
\begin{eqnarray}
\mu(\ux) = \frac{1}{Z_{\H,\ub}}\,\prod_{a=1}^m
\ind(x_{i_1(a)}\oplus\cdots\oplus x_{i_l(a)}=b_a)\, ,
\end{eqnarray}
where $\oplus$ denotes sum modulo $2$, and $Z_{\H,\ub}$
is the number of solutions of the linear system. A factor graph 
representation is naturally associated to this measure
analogously to what we did in Section \ref{sec:other-models}
for the satisfiability problem.

In the following we study the set of solutions
(equivalently, the uniform measure $\mu(\ux)$ over such solutions), 
of random $l$-XORSAT instances, distributed according to
various ensembles. 
The relevant control parameter is the number of equations per variable 
$\alpha= m/n$. For the ensembles discussed here there exists a 
critical value $\alpha_{\rm s}(l)$ such that,
if $\alpha<\alpha_{\rm s}(l)$, a random instance
has solutions with high probability. Vice-versa, if 
$\alpha>\alpha_{\rm s}(l)$, a random instance typically does not 
have solutions.

In the regime in which random instances have solutions,
the structure of the solution set changes dramatically 
as $\alpha$ crosses a critical value $\alpha_{\rm d}(l)$.
The two regimes are characterized as follows (where
all statements should be understood as holding 
with high probability).
\begin{itemize}
\item[I.] $\alpha<\alpha_{\rm d}(l)$. The set of solutions 
of the linear system $\H\ux = \ub$ forms a `well connected lump.'
More precisely, there exist $c=c(\eps)<\infty$ such that
if a set $\Omega\subseteq\{0,1\}^n$ contains at least 
one solution and at most half of the solutions, then 
for all $\eps>0$ and $n$,
\begin{eqnarray}
\frac{\mu(\partial_\eps\Omega)}{\mu(\Omega)}\ge n^{-c} \,.
\label{eq:NoClustXOR}
\end{eqnarray}
%
\item[II.] $\alpha_{\rm d}(l)<\alpha<\alpha_{\rm s}(l)$.
The set of solutions is `clustered.' There exists 
a partition of the 
hypercube $\{0,1\}^n$ into sets $\Omega_1,\dots,\Omega_{\cN}$ 
such that $\mu(\Omega_\ell)>0$ and
\begin{eqnarray}
\mu(\partial_{\eps}\Omega_{\ell}) = 0\, ,
\end{eqnarray}
for some $\eps>0$ and all $\ell$. Further $\cN = e^{n\Sigma+o(n)}$
for some $\Sigma>0$, and each subset $\Omega_{\ell}$
contains the same number of solutions 
$|\{\ux\in\Omega_{\ell}:\H\ux=\ub\}|= e^{ns+o(n)}$, for some $s\ge 0$. 
Finally, the uniform measure over solutions
in $\Omega_{\ell}$, namely 
$\mu_{\ell}(\,\cdot\,) = \mu(\,\cdot \,|\Omega_{\ell})$, satisfies 
the condition  (\ref{eq:NoClustXOR}).
\end{itemize}

The fact that the set of solutions of XORSAT forms 
an affine space makes it a much simpler problem, 
both computationally and in terms of analysis.
Nevertheless, XORSAT shares many common features with 
other CSPs.
For example, regimes I and II are analogous to 
the first two regimes introduced in Section \ref{sec:Broad}
for the coloring problem. In particular, the measure
$\mu(\,\cdot\,)$ exhibits coexistence in the second
regime, but not in the first. This 
phenomena is seen in many random CSPs ensembles
following the framework of Section \ref{sec:random-CSP},
well beyond the cases of coloring and XORSAT 
(see \cite{MezardMontanari} for further details
and references). Some rigorous results in this direction 
are derived in \cite{AchlioRicci,Clust2}, but 
the precise parameter range for the various regimes 
remains a conjecture, yet to be proved. 
Even for XORSAT, where the critical 
values $\alpha_{\rm d}(l)$, $\alpha_{\rm s}(l)$ have 
been determined rigorously (see \cite{cocco,mezard}), the 
picture we described is not yet completely proved. Indeed, 
as of now, we neither have a proof of 
absence of coexistence for $\alpha<\alpha_{\rm d}(l)$, 
i.e. the estimate Eq.~(\ref{eq:NoClustXOR}), nor a 
proof of its analog for the uniform measure 
$\mu_{\ell}(\,\cdot\, )$  when $\alpha>\alpha_{\rm d}(l)$.  

In Section \ref{sec-xorsat} we focus on the case of random 
regular graphs, moving in Section \ref{sec:RandomCore} 
to uniformly random ensembles and their $2$-cores, for 
which Sections \ref{sec:MarkovKernel} to \ref{sec:FSS}
explore the dynamical (or `clustering') phase transition
and derive its precise behavior at moderate values of $n$.
%
\subsection{XORSAT on random regular graphs}\label{sec-xorsat}

We begin with some general 
facts about the XORSAT problem.
Given a XORSAT instance  $\cF =(\H,\ub)$, we denote by $r(\H)$ 
the rank of $\H$ over the finite
field $\GF$, and by $Z_{\H} \ge 1$ the number
of solutions $\ux$ of $\H \ux = \uu{0} \mod 2$. 
From linear algebra we know that $\cF$ is satisfiable
(i.e. there exists a solution 
for the system $\H \ux = \ub \mod 2$) 
if and only if $\ub$ is in the image
of $\H$. This occurs for precisely
$2^{r(\H)}$ of the $2^m$ possible binary 
vectors $\ub$ and in particular it occurs
for $\ub = \uu{0}$. If $\cF$ 
is satisfiable then the set of all 
solutions is an affine space of 
dimension $n-r(\H)$ over $\GF$, hence 
of size $Z_{\H} = 2^{n-r(\H)}$. Further, 
$r(\H) \le m$ with $r(\H)=m$ if 
and only if the rows of $\H$
are linearly independent (over $\GF$),
or equivalently iff $Z_{\H^T} =1$ (where $\H^T$ denotes
the transpose of the matrix $\H$).

Let us now consider a generic distribution over
instances $\cF=(\H,\ub)$ 
such that $\ub$ is chosen uniformly from $\{0,1\}^m$
independently of the matrix $\H$.
It follows from the 
preceding that
\begin{equation}\label{eq:XOR-sat-pr}
\prob( \cF \mbox{ is satisfiable}) = 2^{n-m} \E [ 1/ Z_{\H} ] 
\ge \prob( Z_{\H^T} = 1 ) \,.
\end{equation}
By these considerations also
$$
\prob( \cF \mbox{ is satisfiable}) 
 \le \frac{1}{2}+ \frac{1}{2} \prob( Z_{\H^T}=1) \, .
$$ 
As mentioned already, the probability that a random instance 
is satisfiable,  
$\prob( \cF \mbox{ is satisfiable})$, 
abruptly drops from near one to near zero
when the number of equations per variable crosses a 
threshold $\alpha_{\rm s}(l)$.
As a consequence of our bounds, $\prob( Z_{\H^T} = 1 )$ 
abruptly drops from near one to near zero at the same threshold.

Suppose that $\H$ is the \emph{parity matrix} of 
a factor graph $G=(V,F,E)$ which may have multiple edges. 
That is, each entry $H_{ai}$ of the binary matrix $\H$ is 
just the parity of the multiplicity of edge $(a,i)$ in $G$. 
We 
associate to the instance $\cF$ an energy function 
$\eham_{\H,\ub}(\ux)$ given by the number of unsatisfied equations
in the linear system $\H\ux = \ub$, with the corresponding
partition function 
\begin{eqnarray}
Z_{\H,\ub}(\beta) = \sum_{\ux} \exp\{-2\beta \eham_{\H,\ub}(\ux)\}\, .
\end{eqnarray}
In particular, $Z_{\H} = \lim_{\beta \to\infty } Z_{\H,\ub}(\beta)$
whenever the instance is satisfiable. 
Moreover, it is easy to show that $Z_{\H,\ub}(\beta)$ is 
independent of $\ub$ whenever the instance is satisfiable. 

We thus proceed to apply the general approach of  
high-temperature expansion on $Z_{\H,\uu{0}}(\beta)$,
which here yields important implications for all $\beta>0$ 
and in particular, also on $Z_{\H}$.
For doing so, it is convenient to map the variable domain 
$\{0,1\}$ to $\{+1,-1\}$ and rewrite $Z_{\H,\uu{0}}(\beta)$ as 
the partition function of a generalized (ferromagnetic) Ising 
model of the form (\ref{eq:extIsing}). That is, 
\begin{eqnarray}
Z_{\H,\uu{0}}(\beta) = e^{-\beta|F|}\sum_{\underline{x}\in\{+1,-1\}^V}
\exp\Big\{\beta\sum_{a\in F} x_a\Big\} \equiv  e^{-\beta|F|}Z_G(\beta) \,,
\end{eqnarray} 
where $x_a \equiv \prod_{i\in\partial a} x_i$ for each $a \in F$.
We also  
introduce the notion of a \emph{hyper-loop}
in a factor graph $G=(V,F,E)$, which is a subset 
$F'$ of function nodes such that every variable node $i \in V$ 
has an even degree in the induced subgraph $G'=(V,F',E')$.
\begin{lemma}\label{lem:ht-exp}
Set $N_G(0) \equiv 1$ and 
$N_G(\ell)$ denote the number of hyper-loops of size 
$\ell \ge 1$ in a factor graph $G=(V,F,E)$.
Then, for any $\beta \in \reals$,
\begin{eqnarray}\label{eq:ht-XORSAT}
Z_G(\beta) = 2^{|V|} (\cosh\beta)^{|F|}\sum_{\ell= 0}^{|F|} N_G(\ell)\,
(\tanh\beta)^\ell \, . 
\end{eqnarray}
Further, if $Z_{\H^T}=1$ with $\H$ the parity matrix of $G$ 
then 
$$
Z_G(\beta)=2^{|V|} (\cosh\beta)^{|F|} \,.
$$
\end{lemma}
\begin{proof}
Observe that $e^{\beta x_a} = \cosh(\beta) [ 1 + x_a (\tanh \beta)]$ for 
any function node $a \in F$ and any $\ux \in \{+1,-1\}^V$. Thus, 
setting $F=[m]$, we have the following   
`high-temperature expansion' of $Z_G(\beta)$ as a polynomial 
in $(\tanh \beta)$,
\begin{align*}
Z_G(\beta)&= (\cosh \beta)^{m} \sum_{\ux} 
\prod_{a \in F} [1 + x_a (\tanh \beta) ]\\
&= (\cosh \beta)^{m} \sum_{F' \subseteq [m]} 
(\tanh \beta)^{|F'|} \sum_{\ux} \prod_{a \in F'} x_a \,.
\end{align*}
Since $x_i^2 =1$ for each $i \in V$ we see that 
$\prod_{a \in F'} x_a$ is simply the product 
of $x_i$ over all variable nodes $i$ that have 
an odd degree in the induced subgraph $G'$. The sum of
this quantity over all $\ux \in \{+1,-1\}^V$ is by symmetry 
zero, unless $F'$ is either 
an empty set or a hyper-loop of $G$, in which case this 
sum is just the number of such binary vectors $\ux$, that is
$2^{|V|}$. Therefore, upon collecting together all 
hyper-loops $F'$ in $G$ of the same size we get
the stated formula (\ref{eq:ht-XORSAT}). To complete the 
proof of the lemma note that the sum of columns of the 
transpose $\H^T$ of the parity matrix of $G$ corresponding 
to the function nodes in a hyper-loop $F'$ in $G$ must be
the zero vector over the field $\GF$. Hence, the existence of a hyper-loop
in $G$ provides a non-zero solution of $\H^T \uy = \uu{0} \mod 2$
(in addition to the trivial zero solution). 
Consequently, if $Z_{\H^T}=1$ then $N_G(\ell)=0$ for all $\ell \ge 1$,
yielding the stated explicit expression for $Z_G(\beta)$. 
\end{proof}

We now consider the $k$-XORSAT for ensembles 
$\G_{l,k}(n,m)$
of random \emph{$(l,k)$-regular graphs}
drawn from the corresponding configuration model.
Such an ensemble is defined whenever $nl=mk$ as follows. 
Attach $l$ half-edges to each variable node $i\in V$, 
and $k$ half-edges to each function node $a\in F$. 
Draw a uniformly random permutation over $nl$ elements, 
and connect edges on the two sides accordingly. 
We then have the following result about uniqueness 
of solutions of random regular linear systems.
\begin{thm}\label{thm:Solutions}
Let $\H$ denote the $m \times n$ parity matrix 
of a random $(l,k)$-regular factor graph
from $\G_{l,k}(n,m)$, with $l>k \ge 2$.
Then, the linear system 
$\H \ux=\uu{0} \mod 2$ has, with high probability as
$n \to \infty$, the unique solution $\ux=\uu{0}$.  
\end{thm}
\begin{proof}
Let $Z_{\H}(w)$ denote the number of solutions 
of $\H \ux=\uu{0}$ with $w$ non-zero entries. 
Such solution corresponds to a coloring, by say red, 
$w$ vertices of the multi-graph $G$, and by, say blue, 
the remaining $n-w$ vertices, while having 
an even number of red half-edges at  
each function node. A convenient way to compute 
$\E\, Z_{\H}(w)$ is thus to divide the number of 
possible graph colorings with this property by 
the total size $(nl)!$ of the ensemble $\G_{l,k}(n,m)$. 
Indeed, per integers $(m_r, r=0,\ldots,k)$ that 
sum to $m$ there are 
$$
\binom{m}{m_0,\ldots,m_k} \prod_{r=0}^k \binom{k}{r}^{m_r}
$$
ways of coloring the $m k$ half-edges of function nodes 
such that $m_r$ nodes have $r$ red half-edges, for 
$r=0,\ldots,k$. There are $\binom{n}{w}$ ways of 
selecting red vertices and $(wl)!(nl-wl)!$ color 
consistent ways of matching half-edges of factor nodes 
with those of vertices, provided one has the same 
number of red half-edges in both collections, that is
$\sum_r r m_r = wl$. The value of $\E Z_{\H}(w)$ is 
thus obtained by putting all this together, 
and summing over all choices of 
$(m_0,\ldots,m_k)$ such that $m_r=0$ for odd values of $r$.
Setting $w=n \omega$ and using Stirling's formula one 
finds that $\E Z_{\H}(w) =\exp(n \phi(\omega) +o(n))$ 
for any fixed $\omega \in (0,1)$, with an explicit
expression for $\phi(\cdot)$ 
(c.f. \cite[Section 11.2.1]{MezardMontanari} and the
references therein). For $l > k \ge 2$
the only local maximum of $\phi(\omega)$ 
is at $\omega=1/2$, with $\phi(0)=\phi(1)=0$ 
and $\phi(1/2)<0$. Hence, in this case 
$\phi(\omega)<0$ for all $\omega \in (0,1)$. Further, 
from the formula for $\E Z_{\H}(w)$ one can show that 
for $\kappa>0$ small enough the sum of $\E Z_{\H}(w)$ 
over $1 \le w \le \kappa n$ and $n-w \le \kappa n$
decays to zero with $n$. Therefore,
$$
\lim_{n\to\infty}\sum_{w=1}^n\E\, Z_{\H}(w) = 0 \,,
$$
which clearly implies our thesis.
\end{proof}

We have the following consequence about satisfiability 
of XORSAT for random $(l,k)$-regular factor graphs.
\begin{coro}\label{thm:MainXORReg}
Choosing a random $(l,k)$-regular factor graph
$G$ from the ensemble $\G_{l,k}(n,m)$, with $k>l\ge 2$. Then,
the probability that $Z_G(\beta)= 2^n(\cosh \beta)^{ln/k}$ goes 
to one as $n \to \infty$ and so does the probability that 
$k$-XORSAT with the corresponding parity matrix
$\H$ has $2^{sn}$ solutions for $s=1-l/k$.
\end{coro}
\begin{proof} Let $\H$ be the parity matrix of a randomly chosen 
$(l,k)$-regular factor graph $G$ from ensemble $\G_{l,k}(n,m)$.
Then, $\H^T$ has the law of the parity matrix for
a random $(k,l)$-regular factor graph from ensemble $\G_{k,l}(m,n)$.
Thus, from Theorem \ref{thm:Solutions} 
we know that $\prob(Z_{\H^T} = 1) \to 1$ as $n \to \infty$ and 
by Lemma \ref{lem:ht-exp}
with same probabilities also $Z_G(\beta)= 2^n(\cosh \beta)^{ln/k}$
(as here $|V|=n$ and $|F|=nl/k$). We complete the proof upon 
recalling that 
if $Z_{\H^T} = 1$ then $r(\H)=m=|F|$ and there are 
$2^{n-m}$ solutions $\ux$ of the corresponding XORSAT problem
(for any choice of $\ub$).
\end{proof}

See \cite[Chapter 18]{MezardMontanari} for more information on XORSAT models, 
focusing on the zero-temperature case.

\subsection{Hyper-loops, hypergraphs, cores and a peeling algorithm}
\label{sec:RandomCore}

As we saw in the previous section, if the bipartite graph $G$
associated to the matrix $\H$ does not contain hyper-loops,
then the linear system $\H\ux = \ub$ is solvable for any $\ub$.
This is what happens for $\alpha<\alpha_{\rm s}(d)$:
a random matrix $\H$ with $l$ non-zero elements per row is,
with high probability, free from hyper-loops.
Vice versa, when $\alpha>\alpha_{\rm s}(d)$ the matrix $\H$ 
contains, with high probability, $\Theta(n)$ hyper-loops.
Consequently, the linear system $\H\ux =\ub$ is solvable only
for $2^{m-\Theta(n)}$ of the vectors $\ub$. If $\ub$
is chosen uniformly at random, this implies that
$\prob\{\cF \mbox{ is satisfiable}\}\to 0$ as $n\to\infty$.

Remarkably, the clustering threshold $\alpha_{\rm d}(l)$
coincides with the threshold for appearance of
a specific subgraph of the bipartite graph $G$, called 
the \emph{core} of $G$. The definition of the
core is more conveniently given in the language of 
hypergraphs. 
This is an equivalent description of factor
graphs, where the hypergraph corresponding to
$G=(V,F,E)$ is formed by associating with
each factor node $a\in F$ the hyper-edge 
(i.e. a subset of vertices in $V$),
$\partial a$ consisting of all 
vertices $i \in V$ such that $(i,a)\in E$. The
same applies for factor multi-graphs, in which 
case a vertex $i \in V$ may appear with 
multiplicity larger than one in some hyper-edges.
\begin{definition}\label{def:r-core}
The \emph{$r$-core} of hyper-graph $G$ is the unique
subgraph obtained by recursively removing all 
vertices of degree less than $r$ 
(when counting multiplicities of vertices 
in hyper-edges). In particular, the $2$-core, 
hereafter called the \emph{core} of $G$, 
is the maximal collection of hyper-edges 
having no vertex appearing in only one of them 
(and we use the same term for the induced subgraph).
\end{definition}

Obviously, if $G$ contains a non-empty hyper-loop, it also
contains a non-empty core. It turns out that the probability 
that a random hypergraph contains a non-empty core
grows sharply from 
near zero to near one
as the number of hyper-edges
crosses a threshold which coincides with the clustering
threshold $\alpha_{\rm d}(l)$ of XORSAT.
 
Beyond XORSAT, the core of a hyper-graph plays an 
important role in the analysis of many combinatorial problems.

For example, Karp and Sipser~\cite{karp} consider
the problem of finding
the largest possible matching (i.e. vertex disjoint set of edges) in a
graph $G$.  They propose a simple \emph{peeling 
algorithm} that recursively selects an edge $e=(i,j) \in G$ for 
which the vertex $i$ has degree one, as long as such an 
edge exists, and upon including $e$ in the matching, the algorithm  
removes it from $G$ together with all edges incident on $j$
(that can no longer belong to the matching).
Whenever the algorithm successfully matches all vertices, the resulting
matching can be shown to have maximal size. Note that this 
happens if an only if the core of the hyper-graph $\widetilde{G}$ 
is empty, where $\widetilde{G}$ has a c-node $\widetilde{e}$ 
per edge $e$ of $G$ and a v-node $\widetilde{i}$ per vertex 
$i$ of degree two or more in $G$ that is incident
on $\widetilde{e}$ in $\widetilde{G}$ if and only if $e$ is incident on $i$
in $G$. Consequently, the performance of the Karp-Sipser algorithm for 
a randomly selected graph 
has to do with the probability of non-empty core in the 
corresponding graph ensemble. For example, 
\cite{karp} analyze the asymptotics of this probability for a
uniformly chosen random graph of $N$ vertices 
and $M=\lfloor N c/2 \rfloor$ edges, as $N \to \infty$
(c.f.  \cite{aronson,dyer} for recent contributions).

A second example deals with the decoding of a noisy message when 
communicating over the binary erasure channel
with a low-density parity-check code ensemble. This amounts to 
finding the {\it unique} solution of a linear system over $\GF$
(the solution exists by construction, but is not necessarily unique, 
in which case decoding fails).
If the linear system includes an equation
with only one variable, we thus determine the value of
this variable, and substitute it throughout the system.
Repeated recursively, this procedure either determines all the
variables, thus yielding the unique solution of the system, or halts on a 
linear sub-system each of whose equations involves at least two variables.
While such an algorithm is not optimal (when it halts, the resulting
linear sub-system might still have a unique solution), it
is the simplest instance of the widely
used belief propagation decoding strategy, that has proved
extremely successful. For example, on properly optimized code
ensembles, this algorithm has been shown to achieve the theoretical
limits for reliable communication, i.e., Shannon's channel capacity
(see \cite{LMSS01}). Here a hyper-edge of the hyper-graph $G$ 
is associated to each variable, and a vertex is associated 
to each equation, or parity check, 
and the preceding decoding scheme successfully finds the 
unique solution if and only if the core of $G$ is empty. 

In coding theory one refers to each variable 
as a v-node of the corresponding bipartite factor 
graph representation of $G$ and to each parity-check 
as a c-node of this factor graph. 
This coding theory setting is dual to the one considered 
in XORSAT. Indeed, as we have already seen, satisfiability 
of an instance $(\H,\ub)$ for most choices of $\ub$ is
equivalent to the uniqueness of the solution of 
$\H^T\ux=0$. Hereafter we adopt this `dual' but equivalent 
coding theory language, considering a hyper-graph $G$ chosen 
uniformly from an ensemble $\G_l(n,m)$ with $n$ 
hyper-edges (or v-nodes), each of whom is a 
collection of $l \ge 3$ 
vertices (or c-nodes), from the vertex set $[m]$. 
Note that as we passed to the 
dual formulation, we also inverted the roles of $n$ and $m$. 
More precisely, each hyper-graph
in $\G=\G_l(n,m)$ is described by an ordered list of edges, 
i.e. couples $(i,a)$ with $i\in [n]$ and $a \in [m]$
\begin{eqnarray*}
E= [(1,a_1),(1,a_2),\dots ,(1,a_l);(2,a_{l+1}),\dots;(n,a_{(n-1)l+1}),
\dots,(n,a_{nl})]\, ,
\end{eqnarray*}
where a couple $(i,a)$ appears \emph{before} $(j,b)$ whenever $i<j$ and each
v-node $i$ appears \emph{exactly} $l$ times in the list, with $l\ge 3$ a
fixed integer parameter. In this configuration model 
the \emph{degree} of a v-node $i$ (or c-node $a$), refers
to the number of edges $(i,b)$ (respectively $(j,a)$) in $E$ to which 
it belongs (which corresponds to counting
hyper-edges and vertices \emph{with their multiplicity}). 
%

%
To sample $G$ from the uniform distribution over $\G$ 
consider the v-nodes in order, $i=1,\ldots,n$, choosing 
for each v-node and $j=1,\ldots,l$,
independently and uniformly at random a c-node
$a=a_{(i-1)l+j} \in [m]$ and adding the couple $(i,a)$ to the list $E$.
Alternatively, to sample from this distribution first attribute 
sockets $(i-1)l+1,\ldots,il$ to the $i$-th v-node, $i=1,\ldots,n$,
then attribute $k_a$ sockets to each c-node $a$, where $k_a$'s
are mutually independent Poisson($\zeta$) random variables,
conditioned upon their sum being $nl$ (these sockets
are ordered using any pre-established convention). Finally,
connect the v-node sockets to the c-node sockets according to
a permutation $\sigma$ of $\{1,\ldots,nl\}$ that is chosen
uniformly at random and independently of the choice of $k_a$'s.

In the sequel we take $m=\lfloor n \rho\rfloor$ 
for $\rho=l/\gamma>0$ bounded away from $0$ and $\infty$
and study the large $n$ asymptotics of the probability 
\begin{eqnarray}
P_l(n,\rho) \equiv\prob\left\{ G\in \G_l(n,m) \mbox{ has a non-empty core}
\right\}\, 
\end{eqnarray}
that a hyper-graph $G$ of this distribution has a non-empty core. 
Setting $\H^T$ as the parity matrix of a uniformly 
chosen $G$ from $\G_l(n,m)$ corresponds to a binary 
matrix $\H$ chosen uniformly (according to a 
configuration model), among all $n \times m$ matrices 
with $l$ non-zero entries per row. That is, the 
parameter $\rho$ corresponds to $1/\alpha$ in the $l$-XORSAT.

\subsection{The approximation by a smooth Markov kernel}\label{sec:mc-apx}
\label{sec:MarkovKernel}

Our approach to $P_l(n,\rho)$ is by analyzing whether 
the process of sequentially peeling, or decimating, 
c-nodes of degree one, corresponding to the
decoding scheme mentioned before, ends 
with an empty graph, or not. That is, 
consider the inhomogeneous Markov chain of 
graphs $\{G(\tau),\, \tau\ge 0\}$, where  
$G(0)$ is a uniformly random element of $\G_l(n,m)$ and 
for each $\tau=0,1,\dots$, if there is a non-empty set of 
c-nodes of degree $1$, choose one of 
them (let's say $a$) uniformly at random, deleting  
the corresponding edge $(i,a)$
together with all the edges incident to the
v-node $i$. The graph thus obtained is $G(\tau+1)$.
In the opposite case, where there are no c-nodes of degree $1$ in $G(\tau)$,
we set $G(\tau+1) = G(\tau)$.

\noindent {\bf Reducing the state space to $\integers_+^2$.}
We define furthermore the process 
$\{\vz(\tau)= (z_1(\tau),z_2(\tau)),\, \tau\ge 0\}$ on $\integers_+^2$,
where $z_1(\tau)$
and $z_2(\tau)$ are, respectively, the number of c-nodes in $G(\tau)$,
having degree one or larger than one. Necessarily, 
$(n-\htau ) l \geq z_1(\tau)+2z_2(\tau)$, with equality if
$z_2(\tau)=0$, where  
$\htau \equiv \min(\tau, \inf \{ \tau' \geq 0 : z_1(\tau') = 0 \})$, 
i.e. $\htau=\tau$ till  
the first $\tau'$ such that $z_1(\tau')=0$, after which 
$\htau$ is frozen (as the algorithm stops).

Fixing $l \ge 3$, $m$ and $n$, set $\vz \equiv (z_1,z_2) \in 
\integers_+^2$ and $\G(\vz,\tau)$ denote the
ensemble of possible bipartite graphs with $z_1$ c-nodes
of degree one and $z_2$ c-nodes of degree at least two,
after exactly $\tau$ removal steps of this process.
Then, $\G(\vz,\tau)$ is non-empty only if $z_1+2z_2\le (n-\tau)l$ with 
equality whenever $z_2=0$. Indeed, each element of 
$\G(\vz,\tau)$ is a bipartite graph $G = (U,V;R,S,T;E)$ where 
$U,V$ are disjoint subsets of $[n]$ with $U\cup V = [n]$ and
$R,S,T$ are disjoint subsets of $[m]$ with 
$R\cup\ S\cup T = [m]$, having the cardinalities
$|U|=\tau$, $|V|=n-\tau$, $|R| = m-z_1-z_2$, $|S| = z_1$, $|T|=z_2$
and the ordered list $E$ of $(n-\tau)l$ edges $(i,a)$ 
with $i$ a v-node and $a$ a c-node such that 
each $i\in V$ appears as the first coordinate
of exactly $l$ edges in $E$, while each $j\in U$ does not appear in
any of the couples in $E$. Similarly,
each $c\in R$ does not appear in $E$, each $b\in S$ appears as the
second coordinate of exactly one edge in $E$,
and each $a\in T$ appears in some $k_a \ge 2$ such edges.

The following observation allows us to focus on 
the much simpler process $\vz(\tau)$ on 
$\integers_+^2$ instead of the graph process
$G(\tau) \in \G(\vz,\htau)$.
\begin{lemma}\label{lemma:OriginalProcess}
Conditional on $\{\vz(\tau'), 0\le \tau'\le\tau\}$,
the graph $G(\tau)$ is uniformly distributed over $\G(\vz,\htau)$. 
Consequently, the process $\{\vz(\tau)\, \tau\ge 0\}$ 
is an inhomogeneous Markov process.
\end{lemma}
\emph{Proof outline:}  Fixing $\tau$, $\vz=\vz(\tau)$ 
such that $z_1>0$, $\vz'=\vz(\tau+1)$ 
and $G' \in \G(\vz',\tau+1)$, let $N(G'|\vz,\tau)$ 
count the pairs of graphs $G \in \G(\vz,\tau)$ and
choices of the deleted $c$-node from $S$ that 
result with $G'$ upon applying a single step of our algorithm.
Obviously, $G$ and $G'$ must be such that 
$R \subset R'$, $S \subseteq R' \cup S'$ and $T' \subseteq T$.
With $q_0 \equiv |R' \cap S|$, $p_0 \equiv |R' \cap T|$,
$q_1 \equiv |S' \cap T|$ and $q_2$ denoting the number of
$c$-nodes $a \in T'$ for which $k_a>k'_a$, 
it is shown in \cite[proof of Lemma 3.1]{hcore} that 
($p_0$, $q_0$, $q_1$, $q_2$) belongs to 
the subset ${\mathcal D}$ of $\integers_+^4$ where 
both the relations 
\begin{eqnarray}\label{eq7}
\left\{\begin{array}{rcl}
z_0 & = & z_0'-q_0-p_0\, ,\\
z_1 & = & z_1'+q_0-q_1\, ,\\
z_2 & = & z_2'+p_0+q_1\, ,
\end{array}\right.
\end{eqnarray}
for $z_0=m-z_1-z_2$, $z_0'=m-z_1'-z_2'$,
and the inequalities 
$(n-\tau)l-(z_1+2z_2) \ge l - (2p_0+q_0+q_1) \ge q_2$, $q_0+p_0\le z_0'$, 
$q_1\le z_1'$ (equivalently, $q_0 \le z_1$), $q_2\le z_2'$ (equivalently,
$p_0+q_1+q_2 \le z_2$) hold. In particular 
$|\mathcal D| \le (l+1)^4$. It is further shown there that
\begin{equation}\label{eq:nform}
N(G'|\vz,\tau)
= (\tau+1)\; l!\, \sum_{{\mathcal D}}
\binom{m-z_1'-z_2'}{q_0,p_0,\cdot}\binom{z_1'}{q_1}\binom{z_2'}{q_2}
c_l (q_0,p_0,q_1,q_2) \,,
\end{equation} 
depends on $G'$ only via $\vz'$, where 
$$
c_l (q_0,p_0,q_1,q_2) = q_0
\coeff[(e^{\bx}-1-{\bx})^{p_0}(e^{\bx}-1)^{q_1+q_2},{\bx}^{l-q_0}]\, .
$$

We start at $\tau=0$ with a uniform distribution of $G(0)$ within 
each possible ensemble $\G(\vz(0),0)$.
As $N(G'|\vw,\tau)$ depends on $G'$ only via $\vw'$ it follows by 
induction on $\tau=1,2,\ldots$ that 
conditional on $\{\vz(\tau'), 0\le \tau'\le\tau\}$,
the graph $G(\tau)$ is uniformly distributed over $\G(\vz,\htau)$ 
as long as $\htau=\tau$. Indeed, if $z_1(\tau)>0$, then 
with $h(\vz,\tau)$ denoting the number of graphs in $\G(\vz,\tau)$,
\begin{eqnarray*}
\prob\left\{G(\tau+1) = G'| \{\vz(\tau'), 0\le \tau'\le\tau\}\right\}
= \frac{1}{z_1} \frac{N(G'|\vz(\tau),\tau)}{h(\vz(\tau),\tau)} \,, 
\end{eqnarray*}
is the same for all $G' \in \G(\vz',\tau+1)$. Moreover, 
noting that $G(\tau)=G(\htau)$ and $\vz(\tau)=\vz(\htau)$ we
deduce that this property 
extends to the case of $\htau<\tau$ (i.e. $z_1(\tau)=0$). 
Finally, since there are
exactly $h(\vz',\tau+1)$ graphs in the ensemble $\G(\vz',\tau+1)$  
the preceding implies that $\{\vz(\tau),\, \tau\ge 0\}$ 
is an inhomogeneous Markov process whose transition probabilities 
\begin{eqnarray*}
W^+_{\tau}(\Delta \vz|\vz) \equiv
\prob\{ \vz(\tau+1) = \vz+\Delta \vz\, |
 \,\vz(\tau) = \vz\, \} \,,
\end{eqnarray*}
for $\Delta\vz\equiv(\Delta z_1,\Delta z_2)$ and
$z_1'=z_1+\Delta z_1$, $z_2'=z_2+\Delta z_2$ are
such that $W^+_{\tau}(\Delta \vz|\vz) = \ind(\Delta \vz =0)$ in case
$z_1=0$, whereas 
$W^+_{\tau}(\Delta \vz|\vz) =
h(\vz',\tau+1) N(G'|\vz,\tau)/(z_1 h(\vz,\tau))$
when $z_1>0$.
\endproof

To sample from the uniform distribution on $\G(\vz,\tau)$  
first partition $[n]$ into $U$ and $V$ uniformly at 
random under the constraints $|U| = \tau$ and $|V|=(n-\tau)$ (there are
$\binom{n}{\tau}$ ways of doing this), and independently partition $[m]$ to
$R\cup S\cup T$ uniformly at random under the constraints $|R| = m-z_1-z_2$,
$|S|= z_1$ and $|T|=z_2$ (of which
there are $\binom{m}{z_1,z_2,\cdot}$ possibilities).
Then, attribute $l$ v-sockets to each $i\in V$ and number them
from $1$ to $(n-\tau)l$ according to some pre-established convention.
Attribute one c-socket to each $a\in S$ and $k_a$ c-sockets to each $a\in T$,
where $k_a$ are mutually independent Poisson($\zeta$) random variables
conditioned upon $k_a\ge 2$, and further conditioned upon
$\sum_{a \in T} k_a$ being $(n-\tau)l-z_1$.
Finally, connect the v-sockets and c-sockets according to a uniformly
random permutation on $(n-\tau)l$ objects, chosen independently of the
$k_a$'s. Consequently, 
\begin{eqnarray}\label{eq:hform}
h(\vz,\tau) 
= \binom{m}{z_1,z_2,\cdot}\, \binom{n}{\tau}\,
\coeff[(e^{\bx}-1-\bx)^{z_2},\bx^{(n-\tau)l-z_1}]
((n-\tau)l)!
\end{eqnarray}
%
\noindent
{\bf Approximation by a smooth Markov transition kernel.}
Though the transition kernel $W_{\tau}^+(\cdot|\vz)$ of the
process $\vz(\cdot)$ is given explicitly via
(\ref{eq:nform}) and (\ref{eq:hform}), it is hard to get 
any insight from these formulas, or to use them directly 
for finding the probability of this process hitting the
line $z_1(\tau)=0$ at some $\tau<n$ (i.e. of the graph $G(0)$ 
having a non-empty core). 
Instead, we analyze the simpler transition probability kernel 
\begin{eqnarray}\label{SimpleKernel}
\W_{\theta}(\Delta \vz|\vx) \equiv 
\binom{l-1}{q_0-1,q_1,q_2}\p_0^{q_0-1}\p_1^{q_1}\p_2^{q_2} \, ,
\end{eqnarray}
with $q_0 = -\Delta z_1-\Delta z_2\ge 1$, $q_1=-\Delta z_2\ge 0$
and $q_2 = l+\Delta z_1 +2\Delta z_2\ge 0$, where    
\begin{eqnarray}
\p_0 = \frac{x_1}{l(1-\theta)}\, ,\;\;
\p_1 = \frac{x_2\lambda^2 e^{-\lambda}}
{l(1-\theta)(1-e^{-\lambda}-\lambda e^{-\lambda})}\,,
\;\;
\p_2 = 1 - \p_0 - \p_1,
\label{pdef} 
\end{eqnarray}
for each $\theta \in [0,1)$ and 
$\vx \in \reals_+^2$ such that $x_1+2x_2 \leq l(1-\theta)$.
In case $x_2>0$ we set $\lambda=\lambda(\vx,\theta)$ 
as the unique positive solution of  
\begin{eqnarray}
\frac{\lambda (1-e^{-\lambda})}
{1-e^{-\lambda}-\lambda e^{-\lambda})} 
= \frac{l(1-\theta)-x_1}{x_2}\,\label{eq:LambdaDef}
\end{eqnarray}
while for $x_2=0$ we set by continuity $\p_1 = 0$ 
(corresponding to $\lambda \to \infty$).

Intuitively, $(\p_0,\p_1,\p_2)$ are the probabilities
that each of the remaining $l-1$ edges emanating from the v-node 
to be deleted at the $\tau= n \theta$ step of the algorithm 
is connected to a $c$-node of degree $1$, $2$ 
and at least $3$, respectively. Indeed, 
of the $n l (1-\theta)$ v-sockets at that time,
precisely $z_1 = n x_1$ are connected to c-nodes of 
degree one, hence the formula for $\p_0$. 
Our formula for $\p_1$ corresponds to postulating 
that the $z_2=n x_2$ c-nodes of degree at least two
in the collection $T$ follow a $\poisson(\lambda)$ 
degree distribution, conditioned on having degree
at least two, setting $\lambda>0$  
to match the expected number of c-sockets per c-node 
in $T$ which is given by the right side of (\ref{eq:LambdaDef}).
To justify this assumption, note that 
\begin{equation}\label{eq:tval}
\coeff[(e^{\bx}-1-{\bx})^t,\bx^{s}] \lambda^s
(e^{\lambda}-1-\lambda)^{-t} = \prob(\sum_{i=1}^t N_i = s)\,,
\end{equation}
for i.i.d. random variables $N_i$, each having the law of a
$\poisson(\lambda)$ random variable conditioned to be at least two.
We thus get from (\ref{eq:nform}) and (\ref{eq:hform}), upon 
applying the local CLT for such partial sums, that the tight approximation 
\begin{eqnarray*}
\left| W_{\tau}^+(\Delta \vz|\vz)- \W_{\tau/n}(\Delta \vz|n^{-1} \vz)
\right|\le \frac{C(l,\epsilon)}{n}
\end{eqnarray*}
applies for
$(\vz,\tau)\in \Q_+(\epsilon)$, $\Delta z_1\in \{-l,\dots,l-2\}$,
$\Delta z_2\in\{-(l-1),\dots,0\}$, with  
\begin{align*}
\Q_+ (\epsilon) \equiv \big\{(\vz,\tau)\, :\; 1 
\le z_1\, ; \; n\epsilon\le z_2\, ; \; 
&0\le \tau\le n(1-\epsilon)\, ; \\
&n\epsilon\le (n-\tau)l- z_1-2z_2\big\}\, ,
\end{align*}
approaching (as $\epsilon \downarrow 0$) 
the set $\Q_+(0)\subset\integers^3$ in which 
the trajectory $(\vz(\tau),\tau)$ 
evolves till hitting one of its absorbing states
$\{(\vz,\tau) : z_1(\tau)=0, \tau \le n \}$
(c.f. \cite[Lemma 4.5]{hcore} for the proof, where
the restriction to $\Q_+(\epsilon)$ guarantees that 
the relevant values of $t$ in (\ref{eq:tval})
are of order $n$).

\noindent
{\bf The initial distribution.}
Considering $m = \lfloor n\rho \rfloor$, for 
$\rho=l/\gamma \in [\epsilon,1/\epsilon]$ and large $n$,
recall that 
\begin{eqnarray*}
\prob(\vz(0)=\vz) = \frac{h(\vz,0)}{m^{nl}} = 
\frac{\prob_{\gamma}\left\{\vS_m = (z_1,z_2,nl)\right\}}
{\prob_{\gamma}\left\{S_{m}^{(3)}=nl\right\}}
\end{eqnarray*}
where $\vS_m = \sum_{i=1}^{m}\vX_i$
for $\vX_i = (\ind_{N_i=1},\ind_{N_i\ge 2},N_i) \in \integers_+^3$ and
$N_i$ that are i.i.d. $\poisson(\gamma)$ random variables
(so $\E S_{m}^{(3)}=nl$ up to the 
quantization error of at most $\gamma$). Hence, 
using sharp local CLT estimates for $\vS_m$ we find that  
the law of $\vz(0)$ is 
well approximated by the multivariate Gaussian law 
$\gauss_2(\cdot|n\vy(0);n\covm(0))$ whose mean 
$n \vy(0) \equiv n \vy(\theta;\rho)$ 
consists of the first two coordinates of $n \rho \E \vX_1$, that is,
\begin{eqnarray}\label{eq:ODEInitial0}
\vy(0;\rho)= \rho  
(\gamma e^{-\gamma}, 1-e^{-\gamma}-\gamma e^{-\gamma})\, , 
\end{eqnarray}
and its positive definite covariance matrix $n \covm(0;\rho)$ 
equals $n \rho$ times the conditional 
covariance of the first two coordinates
of $\vX_1$ given its third coordinates. That is,
\begin{eqnarray}\label{eq:ODEInitial4}
\left\{\begin{array}{rcl}
Q_{11}(0) &= &\frac{l}{\gamma}\, \gamma\, e^{-2\gamma}(e^{\gamma}-1+\gamma-
\gamma^2)\,,\\
Q_{12}(0)&=&-\frac{l}{\gamma}\,\gamma\, e^{-2\gamma}(e^{\gamma}-1-\gamma^2)\,, \\
Q_{22}(0) & = & \frac{l}{\gamma}\, e^{-2\gamma}\,
[(e^{\gamma}-1)+\gamma(e^{\gamma}-2)-\gamma^2(1+\gamma)]\,.
\end{array}\right.
\end{eqnarray}
More precisely, as shown for example in \cite[Lemma 4.4]{hcore},
for all $n$, $r$ and $\rho \in [\epsilon,1/\epsilon]$,
\begin{align}
\sup_{\vu \in \reals^2} \sup_{x \in \reals}
\left| \prob\{ \vu \cdot \vz \leq x \}  - 
\gauss_2(\vu \cdot \vz \le  x|n\vy(0);n\covm(0)) \right|
\le \kappa(\epsilon) n^{-1/2} \,. \label{NewInitialConc2}
\end{align}

\noindent
{\bf Absence of small cores}.
A considerable simplification comes from the 
observation that a typical large random hyper-graph does not have 
a non-empty core of size below a certain threshold. Indeed,
a subset of v-nodes of 
a hyper-graph is called a \emph{stopping set} if the restriction 
of the hyper-graph to this subset has no $c$-node of 
degree one. With 
$N(s,r)$ counting the number of stopping sets in our random 
hyper-graph which involve exactly $s$ v-nodes and $r$ c-nodes, 
observe that necessarily $r \leq \lfloor l s/2\rfloor$. Further, 
adapting a result of \cite{Orlitsky} (and its proof) 
to our graph ensemble, it is shown in \cite[Lemma 4.7]{hcore} that
for $l \ge 3$ and any $\epsilon >0$ there exist 
$\kappa=\kappa(l,\epsilon)>0$ and 
$C=C(l,\epsilon)$ finite, such that 
for any $m \geq \epsilon n$
\begin{eqnarray*}
\E\, \Big[ \sum_{s = 1}^{m \kappa}\sum_{r=1}^{\lfloor l s/2\rfloor} 
N(s,r) \Big] 
\leq C m^{1-l/2} \,.
\end{eqnarray*}
Since the core is the stopping set including the 
maximal number of v-nodes, this implies that 
a random hyper-graph from the ensemble $\G_l(n,m)$ 
has a non-empty core of less than
$m\kappa$ v-nodes with probability that is at most $C m^{1-l/2}$
(alternatively, 
the probability of having a non-empty core with less than
$n\kappa$ v-nodes is at most $C\,n^{1-l/2}$).

\subsection{The ODE method and the critical value}\label{sec:ODE}

In view of the approximations of Section \ref{sec:mc-apx}
the asymptotics of $P_l(n,\rho)$ reduces to determining 
the probability $\hprob_{n,\rho}(z_1(\tau)=0$ for some 
$\tau<n)$ that the inhomogeneous Markov chain 
on $\integers_+^2$ 
with the transition kernel $\W_{\tau/n}(\Delta \vz|n^{-1} \vz)$
of (\ref{SimpleKernel})
and the initial distribution $\gauss_2(\cdot|n\vy(0);n\covm(0))$,
hits the line $z_1(\tau)=0$ for some $\tau<n$. 

The functions $(\vx,\theta)\mapsto \p_a(\vx,\theta)$, $a=0,1,2$ 
are of Lipschitz continuous partial derivatives
on each of the compact subsets 
\begin{align*}
\hq_+ (\epsilon)\equiv\left\{(\vx,\theta)\, :\; 
0\le x_1\, ; \; 0\le x_2\, ; \; \theta \in [0,1-\epsilon]\,;\;
0\le (1-\theta)l-x_1-2x_2\right\}\,,
\end{align*}
of $\reals^2 \times \reals_+$ where 
the rescaled (macroscopic) state 
and time variables $\vx \equiv n^{-1} \vz$ and $\theta \equiv \tau/n$
are whenever $(\vz,\tau)\in \Q_+(\epsilon)$. As a result, 
the transition kernels of (\ref{SimpleKernel}) can be extended
to any $\vx \in \reals^2$ such that 
for some $L = L(l,\epsilon)$ finite,
any $\theta,\theta' \in [0,1-\epsilon]$ and $\vx,\vxp \in \reals^2$ 
\begin{eqnarray*}
\l| \W_{\theta'}(\,\cdot\,|\vxp)-
 \W_{\theta}(\,\cdot\,|\vx)\r|_{\rm TV}
\le L\,\left( \l| \vxp-\vx\r|+|\theta'-\theta|\right)
\end{eqnarray*}
(with $||\, \cdot \, ||_{\rm TV}$ denoting the total variation
norm and $\l|\, \cdot\, \r|$ the Euclidean norm in $\reals^2$).

So, with the approximating chain of 
kernel $\W_{\theta}(\Delta \vz|\vx)$
having bounded increments ($=\Delta \vz$),
and its transition probabilities 
depending smoothly on $(\vx,\theta)$, 
the scaled process $n^{-1} \vz(\theta n)$ 
concentrates around the solution of the ODE
\begin{eqnarray}
\frac{\de\vy}{\de\theta}(\theta) & = & \vF(\vy(\theta),\theta)\, ,
\label{eq:ODEFirst}
\end{eqnarray}
starting at $\vy(0)$ of (\ref{eq:ODEInitial0}),
where $\vF(\vx,\theta) = (-1+(l-1)(\p_1-\p_0), -(l-1)\p_1)$ 
is the mean of $\Delta \vz$ under the transitions
of (\ref{SimpleKernel}).  This is  
shown for instance in \cite{cocco,LMSS01,mezard}.

We note in passing that  
this approach of using a deterministic ODE 
as an asymptotic approximation for slowly varying 
random processes goes back at least to \cite{kurtz},
and such degenerate (or zero-one) 
fluid-limits have been established for many other problems. 
For example, this was done in \cite{karp} for the 
largest possible matching and in \cite{PittelEtAl} 
for the size of 
$r$-core 
of random graphs (c.f. \cite{Molloy2} for a general
approach for deriving such results 
without recourse to ODE approximations).

Setting $h_\rho (u)\equiv u-1+\exp(-\gamma u^{l-1})$, with a bit of 
real analysis one verifies that for $\gamma=l/\rho$ finite,   
the ODE (\ref{eq:ODEFirst}) admits a unique solution
$\vy(\theta;\rho)$ subject to the initial condition (\ref{eq:ODEInitial0}) 
such that $y_1(\theta;\rho) = lu^{l-1} h_{\rho}(u)$ for 
$u(\theta)\equiv(1-\theta)^{1/l}$, as long as 
$h_\rho( u(\theta) ) \ge  0$.  Thus, 
if $\rho$ exceeds the finite and positive \emph{critical density}
$$
\rho_{\rm d} \equiv \inf\{\rho>0:\, h_\rho(u)>0 \quad
\forall u\in(0,1]\}  \,,
$$
then $y_1(\theta;\rho)$ is strictly positive for all $\theta \in [0,1)$,
while for any $\rho \le \rho_{\rm d}$ the solution  
$\vy(\theta;\rho)$ first hits the line 
$y_1=0$ at some $\theta_*(\rho)<1$. 

Returning to the XORSAT problem, \cite{cocco,mezard} prove that 
for a uniformly chosen linear system 
with $n$ equations and  $m=\rho n$ variables 
the leaf removal algorithm is successful with high probability 
if $\rho > \rho_{\rm d}$
and fails with high probability if $\rho<\rho_{\rm d}$. 
See \cite[Figure 1]{hcore} for an illustration of this phenomenon.
Similarly, in the context of decoding of a noisy message 
over the binary erasure channel (i.e. 
uniqueness of the solution for 
a given linear system over $\GF$),
\cite{LMSS01} show that with high probability this algorithm
successfully decimates the whole hyper-graph without ever running out
of degree one vertices if $\rho>\rho_{\rm d}$. 
Vice versa, for $\rho<\rho_{\rm d}$, the solution $\vy(\theta;\rho)$ 
crosses the $y_1=0$ plane near which point the algorithm stops
with high probability and returns a core of size $O(n)$. The value
of $\rho$ translates into noise level in this communication 
application, so \cite{LMSS01} in essence explicitly characterizes
the critical noise value, for a variety of codes (i.e. random 
hyper-graph ensembles). Though this result has been successfully 
used for code design, it is often a poor approximation for 
the moderate code block-length (say, $n=10^2$ to $10^5$) that 
are relevant in practice.

The first order phase transition in the size of the core at 
$\rho=\rho_{\rm d}$
where it abruptly changes from an empty core for 
$\rho>\rho_{\rm d}$ to
a core whose size is a 
positive fraction of $n$ for $\rho<\rho_{\rm d}$, 
has other important implications. For example, 
as shown in \cite{cocco,mezard} 
and explained before,
the structure of the
set of solutions of the linear system changes dramatically
at $\rho_{\rm d}$, exhibiting a `clustering effect' when
$\rho<\rho_{\rm d}$. More precisely, 
a typical instance of our ensemble has a core that corresponds
to $n (1-\theta_*(\rho))+o(n)$ equations in 
$n y_2(\theta_*(\rho))+o(n)$ variables. 
The approximately $2^{m-n}$ solutions of the original linear
system partition to about $2^{n \xi(\rho)}$ clusters according 
to their projection on the core, such that the distance between 
each pair of clusters is $O(n)$. 
This analysis also determines the location 
$\rho_{\rm s}$ of the satisfiability phase transition. That is, 
as long as
$\xi(\rho)=y_2(\theta_*(\rho))-(1-\theta_*(\rho))$ is positive,
with high probability the original system is solvable
(i.e the problem is satisfiable), whereas when $\xi(\rho)<0$ 
it is non-solvable with high probability.

We conclude this subsection with a `cavity type' 
direct prediction of the value of $\rho_{\rm d}$ 
without reference to a peeling algorithm (or any 
other stochastic dynamic). To this end, we
set $u$ to denote the probability that a typical c-node
of $\G_l(n,m)$, say $a$, is part of the core. If this is the
case, then an hyper-edge $i$ incident to $a$ is also 
part of the core iff all other $l-1$ sockets of $i$ 
are connected to c-nodes from the core. Using the Bethe ansatz
we consider the latter to be the intersection 
of $l-1$ independent events, each
of probability $u$. So, with probability $u^{l-1}$ 
an hyper-edge $i$ incident to $a$ from the core, 
is also in the core. As already seen,
a typical c-node in our graph ensemble
has $\poisson(\gamma)$ hyper-edges incident to it,
hence $\poisson(\gamma u^{l-1})$ of them shall be from 
the core. Recall that a c-node belongs to the core iff
at least one hyper-edge incident to it is in the
core. By self-consistency, this yields the identity 
$u=1-\exp(-\gamma u^{l-1})$, or alternatively, 
$h_\rho (u)=0$. As we have already seen, the existence 
of $u \in (0,1]$ for which $h_\rho (u)=0$ is equivalent
to $\rho \le \rho_{\rm d}$.
\subsection{Diffusion approximation and scaling window}

As mentioned before, the ODE asymptotics as in \cite{LMSS01} 
is of limited value for decoding with code block-length 
that are relevant in practice. For this reason, \cite{AMRU04}
go one step further and using a diffusion approximation, 
provide the probability of successful
decoding in the double limit of large size $n$
and noise level approaching the critical value
(i.e. taking $\rho_n \to \rho_{\rm d}$).
The resulting asymptotic characterization is of
finite-size scaling type.

Finite-size scaling has been the object of several investigations
in statistical physics and in combinatorics.
Most of these studies estimate the size of the corresponding
scaling window. That is, fixing a small value of $\ve>0$, they
find the amount of change in some control parameter which moves
the probability of a relevant event from $\ve$ to $1-\ve$.
A remarkably general result in this direction is the rigorous formulation
of a `Harris criterion' in \cite{Chayes,Wilson}.
Under mild assumptions, this implies that the scaling window has to
be at least $\Omega(n^{-1/2})$ for a properly defined control parameter
(for instance, the ratio $\rho$ of the number of nodes to
hyper-edges in our problem). A more precise result has recently been
obtained for the satisfiable-unsatisfiable phase transition for the
random $2$-SAT problem, yielding a window of size $\Theta(n^{-1/3})$
\cite{2SATScal}.
Note however that statistical physics arguments suggest that
the phase transition we consider here is not from the same universality
class as the satisfiable-unsatisfiable transition for random $2$-SAT problem.

If we fix $\rho>0$, 
the fluctuations of $\vz(n\theta)$ around $n\vy(\theta)$
are accumulated in $n\theta$ stochastic steps, hence are
of order $\sqrt{n}$. Further, applying the classical 
Stroock-Varadhan martingale characterization technique, 
one finds 
that the 
rescaled variable $(\vz(n\theta)-n\vy(\theta))/\sqrt{n}$ converges 
in law as $n \to \infty$ 
to a Gaussian random variable whose covariance matrix
$\covm(\theta;\rho) = \{Q_{ab}(\theta;\rho); 1\le a,b\le 2\}$
is the symmetric positive definite solution of the ODE:
\begin{eqnarray}
\frac{\de \covm(\theta)}{\de \theta} =
\covd(\vy(\theta),\theta) +
\A (\vy(\theta),\theta) \covm(\theta) +
\covm(\theta) 
\A (\vy(\theta),\theta)^{T}
\label{eq:FirstODECovariance}
\end{eqnarray}
(c.f. \cite{AMRU04}). Here
$\A (\vx,\theta) \equiv \{A_{ab}(\vx,\theta)
= \partial_{x_b} F_a(\vx,\theta) \,;\;1\le a,b\le 2\}$ is 
the matrix of derivatives of the drift term for the
mean ODE (\ref{eq:ODEFirst})
and $\covd(\vx,\theta)= \{G_{ab}(\vx,\theta):a,b\in\{1,2\}\}$
is the covariance of $\Delta \vz$ at $(\vx,\theta)$ 
under the transition kernel (\ref{SimpleKernel}). 
That is, the
non-negative definite symmetric matrix with entries
\begin{eqnarray}\label{eq:GDef}
\left\{\begin{array}{rcl}
G_{11}(\vx,\theta) & = & (l-1)[\p_0+\p_1-(\p_0-\p_1)^2]\, ,\\
G_{12}(\vx,\theta) & = & -(l-1)[\p_0\p_1+\p_1(1-\p_1)]\, ,\\
G_{22}(\vx,\theta) & = & (l-1)\p_1(1-\p_1)\,
\end{array}\right.
\end{eqnarray}
The dependence of 
$\covm(\theta) \equiv \covm(\theta;\rho)$ on $\rho$ is 
via the positive definite
initial condition $\covm(0;\rho)$ 
of (\ref{eq:ODEInitial4}) for the ODE
(\ref{eq:FirstODECovariance}) as well as 
the terms $\vy(\theta)=\vy(\theta;\rho)$ 
that appear in its right side.

Focusing hereafter on the critical case $\rho=\rho_{\rm d}$, 
there exists then a unique \emph{critical time}  
$\theta_{\rm d} \equiv \theta_*(\rho_{\rm d})$ in $(0,1)$
with $y_1(\theta_{\rm d})=y_1'(\theta_{\rm d})=0$
and $y_1''(\theta_{\rm d})>0$, while the 
smooth solution $\theta \mapsto y_1(\theta;\rho_{\rm d})$ 
is positive when $\theta \neq \theta_{\rm d}$ and $\theta \neq 1$ 
(for more on $\vy(\cdot;\cdot)$ see \cite[Proposition 4.2]{hcore}).

For $\rho_n = \rho_{\rm d}+  r n^{-1/2}$ 
the leading contribution to $P_l(n,\rho_n)$ is the 
probability $\hprob_{n,\rho_n}(z_1(n \theta_{\rm d}) \le 0)$ 
for the inhomogeneous Markov chain $\vz(\tau)$ on $\integers_+^2$ 
with transition kernel $\W_{\tau/n}(\Delta \vz|n^{-1} \vz)$
of (\ref{SimpleKernel})
and the initial distribution $\gauss_2(\cdot|n\vy(0);n\covm(0))$ 
at $\rho=\rho_n$. To estimate this contribution, 
note that $y_1(\theta_{\rm d};\rho_{\rm d})=0$, hence
$$
y_1(\theta_{\rm d};\rho_n) = r n^{-1/2} [
\frac{\partial y_1}{\partial \rho} (\theta_{\rm d};\rho_{\rm d}) + o(1)] \,.
$$
Thus, setting 
$\alp_l \equiv \frac{\partial y_1}{\partial \rho} /\sqrt{Q_{11}} $,
both evaluated at $\theta=\theta_{\rm d}$ and $\rho=\rho_{\rm d}$, 
by the preceding Gaussian approximation
\begin{equation}\label{eq:diff-appx}
P_l(n,\rho_n) = 
\hprob_{n,\rho_n}(z_1(n \theta_{\rm d}) \le 0) + o(1) = 
\gauss_1(- r \alp_l)+o(1) \,,
\end{equation}
as shown in \cite{AMRU04}. In particular, 
the phase transition scaling window around 
$\rho=\rho_{\rm d}$ is of size $\Theta(n^{-1/2})$. 

In a related work, \cite{norris1} determine the 
asymptotic core size for a random hyper-graph from an ensemble 
which is the `dual' of $\G_l(n,m)$. In their model 
the hyper-edges (i.e. v-nodes) are of random, Poisson
distributed sizes, which allows for a particularly simple Markovian
description of the peeling algorithm that constructs the core. Dealing
with random hyper-graphs at the critical point, where the
asymptotic core size exhibits a discontinuity, they describe the
fluctuations around the deterministic limit via a certain linear SDE.
In doing so, they heavily rely on the powerful theory of weak convergence,
in particular in the context of convergence of Markov processes.
For further results that are derived along this line of reasoning,
see \cite{norris2,goldschmidt1,goldschmidt2}.
%
\subsection{Finite size scaling correction to the critical value}
\label{sec:FSS}

In contrast with the preceding and closer in level of precision
to that for the scaling behavior in the emergence of the
giant component in Erd\"os-R\'enyi random graphs 
(see \cite{Giant} and references therein),
for $\G_l(n,m)$ and 
$\rho_n = \rho_{\rm d}+r n^{-1/2}$ inside the scaling window, 
it is conjectured in \cite{AMRU04} and proved in \cite{hcore}
that the leading correction to the diffusion approximation 
for $P_l(n,\rho_n)$ is of order $\Theta(n^{-1/6})$.
Comparing this finite size scaling expression
with numerical simulations, as illustrated in 
\cite[Figure 2]{hcore},
we see that it is 
very accurate even at $n\approx 100$.

Such finite size scaling result is 
beyond the scope of weak convergence theory, 
and while its proof involve delicate 
coupling arguments,
expanding and keeping track of the rate of decay of
approximation errors (in terms of $n$), 
similar results are expected for 
other phase transitions within the same class,
such as $k$-core percolation on random graphs (with $k\ge 3$),
or the pure literal rule threshold in random $k$-SAT (with $k\ge 3$,
c.f. \cite{Fr-Su}). In a different direction, the 
same approach provides rates of convergence (in the sup-norm) as $n$ grows, 
for distributions of many inhomogeneous Markov
chains on $\reals^d$ whose transition kernels
$W_{t,n}(x_{t+1}-x_t=y|x_t=x)$ are approximately (in $n$)
linear in $x$, and ``strongly-elliptic'' of uniformly bounded
support with respect to $y$. 

As a first step in proving the finite size scaling, the
following refinement of the left hand side of (\ref{eq:diff-appx})
is provided in \cite[Section 5]{hcore}. 
\begin{propo}\label{PropoApproximated}
Let $\betaw \in (3/4,1)$, $J_n = [n\theta_{\rm d}-n^{\betaw},
n\theta_{\rm d}+n^{\betaw}]$ and 
$|\rho-\rho_{\rm d}|\le  n^{\betaw'-1}$ with $\betaw'<2\betaw-1$. Then,
for $\ve_n = A\log n$ and $\delta_n = D\,n^{-1/2}(\log n)^2$,
\begin{align}
\hprob_{n,\rho}\Big\{\inf_{\tau\in J_n} z_1(\tau)  \le-\ve_n\Big\}-
\delta_n
&\le P_l(n,\rho) \nonumber \\
&\le \hprob_{n,\rho}\Big\{\inf_{\tau\in J_n} z_1(\tau)  \le\ve_n\Big\}
+\delta_n\,.
\end{align}
\end{propo}

At the critical point (i.e. for $\rho = \rho_{\rm d}$ and
$\theta = \theta_{\rm d}$) the solution of the ODE (\ref{eq:ODEFirst}) 
is tangent to the $y_1=0$ plane and fluctuations in the $y_1$ 
direction determine whether a non-empty (hence, large), core exists or not. 
Further, in a neighborhood of $\theta_{\rm d}$ we
have $y_1(\theta)\simeq\frac{1}{2}\tF (\theta-\theta_{\rm d})^2$,
for the positive constant 
\begin{eqnarray}
\tF \equiv \frac{\de^2 y_1}{\de\theta^2}(\theta_{\rm d};\rho_{\rm d})
=  \frac{\de F_1}{\de\theta}(\vy(\theta_{\rm d};\rho_{\rm d}),\theta_{\rm d}) =
\frac{\partial F_1}{\partial\theta}+
\frac{\partial F_1}{\partial y_2}\, F_{2}\, \label{eq:tFDefinition}
\end{eqnarray}
(omitting hereafter arguments that refer to the critical point).
In the same neighborhood, the contribution of fluctuations to 
$z_1(n\theta)-z_1(n\theta_{\rm d})$ 
is approximately $\sqrt{\tG n|\theta-\theta_{\rm d}|}$, with 
$\tG= G_{11}(\vy(\theta_{\rm d};\rho_{\rm d}),\theta_{\rm d})>0$.
Comparing these two contributions we see that the relevant 
scaling is $X_n(t) = n^{-1/3}[z_1(n\theta_{\rm d}+n^{2/3}t)-
z_1(n\theta_{\rm d})]$, which as shown in \cite[Section 6]{hcore}
converges for large $n$, by strong approximation,  to 
$X(t) = \frac{1}{2}\tF t^2+ \sqrt{\tG} W(t)$, for a standard two-sided
Brownian motion $W(t)$ (with $W(0)=0$). That is,
\begin{propo}\label{PropoGaussian}
Let $\xi(r)$ be a normal random variable of 
mean $\left(\frac{\partial y_1}{\partial \rho}\right) r$ and
variance $Q_{11}$ 
(both evaluated at $\theta=\theta_{\rm d}$ and $\rho=\rho_{\rm d}$),
which is independent of $W(t)$.

For some $\betaw \in (3/4,1)$, any $\eta<5/26$, all $A>0$, 
$r \in \reals$ and
$n$ large enough, if $\rho_n = \rho_{\rm d}+ r\, n^{-1/2}$ 
and $\ve_n = A\log n$, then  
\begin{eqnarray}\label{eq:PropoGaussian}
\Big|
\hprob_{n,\rho_n}\big\{ \inf_{\tau\in J_n} z_1(\tau)\le\pm
\ve_n\big\} 
 - \prob\big\{   n^{1/6} \xi+\inf_{t} X(t) \le 0
\big\} \Big|\le n^{-\eta}\, .
\end{eqnarray}
\end{propo}

We note in passing that within the scope of weak convergence
Aldous pioneered in \cite{aldous} the use of Brownian motion with quadratic
drift (ala $X(t)$ of Proposition \ref{PropoGaussian}),
to examine the near-critical behavior of the
giant component in Erd\"os-R\'enyi random graphs, and his 
method was extended in \cite{goldschmidt2}
to the giant set of identifiable vertices in  
Poisson random hyper-graph models.

Combining Propositions \ref{PropoApproximated} and \ref{PropoGaussian} we 
estimate $P_l(n,\rho_n)$
in terms of the distribution of the global minimum of the process
$\{X(t)\}$.
The latter has been determined already 
in \cite{groeneboom}, yielding the following conclusion. 
\begin{thm}\label{thm:Main}
For $l\ge 3$ set $\alp_l = 
\frac{\partial y_1}{\partial \rho}/\sqrt{Q_{11}}$, 
$\bt_l = \frac{1}{\sqrt{Q_{11}}}\tG^{2/3}\, \tF^{-1/3}$ and 
$\rho_n = \rho_{\rm d}+  r\, n^{-1/2}$. Then, for any $\eta<5/26$
\begin{eqnarray}\label{eq:Mainresult}
P_l(n,\rho_n) = \gauss_1(- r \alp_l)+\bt_l \const
\; \gauss_1'(- r \alp_l)
\; n^{-1/6}+ O(n^{-\eta})\, ,
\end{eqnarray}
for $\const \equiv \int_0^{\infty}\!\! [1-\K(z)^2]\; \de z$ and
an explicit function $\K(\cdot)$ (see \cite[equation (2.17)]{hcore}).
\end{thm}

\noindent{\it Proof outline.}
Putting together Propositions \ref{PropoApproximated} and \ref{PropoGaussian},
we get that 
\begin{eqnarray*}
P_l(n,\rho_n) = \prob\left\{   n^{1/6} \xi+\inf_{t} X(t) \le 0
\right\} + O(n^{-\eta})\, .
\end{eqnarray*}
By Brownian scaling, $X(t) = \tF^{-1/3}\tG^{2/3}\tX(\tF^{2/3}\tG^{-1/3}t)$,
where $\tX(t) =\frac{1}{2}t^2 +\tW(t)$ and $\tW(t)$ is also a two
sided standard Brownian motion. 
With $Z = \inf_{t} \tX(t)$,
and $Y$ a standard normal random variable which is independent of 
$\tX(t)$, we clearly have that 
\begin{align}\label{eq:third}
P_l(n,\rho_n) & = \prob\left\{n^{1/6}\left(\frac{\partial y_1}
{\partial \rho}\right) r+n^{1/6}
\sqrt{Q_{11}}Y +  \tF^{-1/3}\tG^{2/3}Z\le 0\right\}+O(n^{-\eta})
\nonumber\\
& = \E\Big\{\gauss_1 \big(- r \alp_l -
\bt_l n^{-1/6} Z\big)
\Big\}+ O(n^{-\eta})\, .
\end{align}
From \cite[Theorem 3.1]{groeneboom} we deduce that $Z$ has
the continuous distribution function $F_Z (z) = 1-\K(-z)^2 \ind(z<0)$, 
resulting after integration by parts with 
the explicit formula for $\const = - \E \, Z$ (and where
\cite[(5.2)]{groeneboom} provides the explicit expression of 
\cite[formula (2.17)]{hcore} for $\K(x)$). Further, as shown 
in \cite[proof of Theorem 2.3]{hcore} all moments of $Z$ 
are finite and the proof is thus completed by 
a first order Taylor expansion of $\gauss_1(\, \cdot\, )$ 
in (\ref{eq:third}) around $-r \alp_l$. 
\endproof

\begin{remark}\label{rem:accuracy}
The simulations in \cite[Figure 2]{hcore} suggest that
the approximation of $P_l(n,\rho_n)$ we provide in (\ref{eq:Mainresult})
is more accurate than the $O(n^{-5/26+\epsilon})$ correction term suggests. 
Our proof shows that one cannot hope for a better error estimate than
$\Theta(n^{-1/3})$ as we neglect the second order term in expanding 
$\Phi(-r \alp_l + Cn^{-1/6})$, see (\ref{eq:third}). We believe this is
indeed the order of the next term in the expansion
(\ref{eq:Mainresult}). Determining its form is an open problem.
\end{remark}

\begin{remark}\label{rem:emergence}
Consider the (time) 
evolution of the core for the hyper-graph process where 
one hyper-edge is added uniformly at random at each time step.
That is, $n$ increases with time, while the number of vertices
$m$ is kept fixed. Let $S(n)$ be the corresponding (random)  
number of hyper-edges in the core of the hyper-graph at time $n$ and 
$n_{\rm d} \equiv \min\{ n : S(n) \geq 1 \}$ the onset of 
a non-empty core. Recall that 
small cores are absent for a typical large random 
hyper-graph, whereas fixing $\rho<\rho_{\rm d}$ 
the probability of an empty core, i.e.
$S(m/\rho)=0$, decays in $m$. Thus, for large $m$
most trajectories $\{S(n)\}$ abruptly jump from having no core
for $n<n_{\rm d}$ to a linear in $m$ core size 
at the random critical edge number $n_{\rm d}$.
By the monotonicity of $S(n)$ we further see that 
$\prob_m\{n_{\rm d}\le m/\rho\} 
= P_{l}(\rho,m/\rho)$, hence 
Theorem \ref{thm:Main} determines the asymptotic
distribution of $n_{\rm d}$. Indeed, as detailed
in \cite[Remark 2.5]{hcore}, upon expressing $n$ 
in terms of $m$ in equation (\ref{eq:Mainresult}) 
we find that 
the distributions of $\widehat{n}_{\rm d} \equiv 
\alp_l (\rho_{\rm d} n_{\rm d}-m )/\sqrt{m/\rho_{\rm d}}
+ \bt_l \const \rho_{\rm d}^{1/6} m^{-1/6}$
converge point-wise to the standard normal law 
at a rate which is faster than $m^{-5/26+\epsilon}$.
\end{remark}
\begin{remark}\label{rem:size}
The same techniques are applicable for other
properties of the core in the `scaling regime'
$\rho_n = \rho_{\rm d}+r\,n^{-1/2}$. For example, 
as shown in \cite[Remark 2.6]{hcore}, for $m=n \rho_n$
and conditional to the existence of a non-empty core,
$(S(n)-n(1-\theta_{\rm d}))/n^{3/4}$ converges in distribution 
as $n \to \infty$ to $(4Q_{11}/\tF^2)^{1/4}\, Z_r$ where
$Z_r$ is a non-degenerate random variable (whose 
density is explicitly provided there). In particular, 
the $\Theta(n^{1/2})$ fluctuations of the core size
at fixed $\rho<\rho_{\rm d}$ are enhanced to $O(n^{3/4})$
fluctuations near the critical point.
%
\end{remark}

\bibliographystyle{amsalpha}

\end{document}